\documentclass[12pt]{article}
\usepackage{amssymb}
\usepackage{amsfonts}
\usepackage{cite}
\usepackage{amsmath}
\usepackage[top=3.8cm,bottom=3.8cm,textwidth=16cm,centering,a4paper]{geometry}

\numberwithin{equation}{section}
\newtheorem{theorem}{Theorem}[section]

\newtheorem{definition}[theorem]{Definition}
\newtheorem{lemma}[theorem]{Lemma}

\newtheorem{proposition}[theorem]{Proposition}

\newenvironment{proof}[1][Proof]{\noindent\textbf{#1.} }{\hfill $\square$}

\begin{document}

\title{Vectorial ground state solutions for a class of Hartree-Fock type
systems with the double coupled feature}
\date{}
\author{Juntao Sun$^{a}$, Tsung-fang Wu$^{b}$\thanks{Corresponding author.}  \\
{\footnotesize $^a$\emph{School of Mathematics and Statistics, Shandong
University of Technology, Zibo, 255049, P.R. China }}\\
{\footnotesize $^{b}$\emph{Department of Applied Mathematics, National
University of Kaohsiung, Kaohsiung 811, Taiwan }}}
\maketitle

\begin{abstract}
In this paper we study the Hartree-Fock type system as follows:
\begin{equation*}
\left\{
\begin{array}{ll}
-\Delta u+u+\lambda \phi _{u,v}u=\left\vert u\right\vert ^{p-2}u+\beta
\left\vert v\right\vert ^{\frac{p}{2}}\left\vert u\right\vert ^{\frac{p}{2}%
-2}u & \text{ in }\mathbb{R}^{3}, \\
-\Delta v+v+\lambda \phi _{u,v}v=\left\vert v\right\vert ^{p-2}v+\beta
\left\vert u\right\vert ^{\frac{p}{2}}\left\vert v\right\vert ^{\frac{p}{2}%
-2}v & \text{ in }\mathbb{R}^{3},%
\end{array}%
\right.
\end{equation*}%
where $\phi _{u,v}(x)=\int_{\mathbb{R}^{3}}\frac{u^{2}(y)+v^{2}\left(
y\right) }{|x-y|}dy,$ the parameters $\lambda,\beta >0$ and $2<p<4$. Such
system is viewed as an approximation of the Coulomb system with two
particles appeared in quantum mechanics, taking into account the Pauli
principle. Its characteristic feature lies on the presence of the double
coupled terms. When $2<p<3,$ we establish the existence and multiplicity of
nontrivial radial solutions, including vectorial ones, in the radial space $%
H_{r}$ by describing the internal relationship between the coupling
constants $\lambda $ and $\beta.$ When $2<p<4,$ we study the existence of
vectorial solutions in the non-radial space $H$ by developing a novel
constraint method, together with some new analysis techniques. In
particular, when $3\leq p<4,$ a vectorial ground state solution is found in $%
H$, which is innovative as it was not discussed at all in any previous
results. Our study can be regarded as an entire supplement in d'Avenia et
al. [J. Differential Equations 335 (2022) 580--614].
\end{abstract}

\footnotetext{ ~\textit{E-mail addresses }: jtsun@sdut.edu.cn(J. Sun), tfwu@nuk.edu.tw (T.-F. Wu).}



\section{Introduction}

Consider a system of $N$ coupled nonlinear Schr\"{o}dinger equations in $%
\mathbb{R}^{3}$:
\begin{equation}
\begin{array}{ll}
-\Delta \psi _{i}+V_{\text{ext}}\psi _{i}+\left( \int_{\mathbb{R}%
^{3}}|x-y|^{-1}\sum\limits_{j=1}^{N}|\psi _{j}(y)|^{2}dy\right) \psi
_{i}+(V_{\text{ex}}\psi )_{i}=E_{i}\psi _{i}, & \text{ }\forall i=1,2,...,N,%
\end{array}
\label{1-0}
\end{equation}%
where $\psi _{i}:\mathbb{R}^{3}\rightarrow\mathbb{C},$ $V_{\text{ext}}$ is a
given external potential, $(V_{\text{ex}}\psi )_{i}$ is the $i$'th component
of the \textit{crucial exchange potential} defined by%
\begin{equation*}
(V_{\text{ex}}\psi )_{i}=-\sum\limits_{j=1}^{N}\psi _{j}(y)\int_{\mathbb{R}%
^{3}}\frac{\psi _{i}(y)\bar{\psi}_{j}(y)}{|x-y|}dy,
\end{equation*}%
and $E_{i}$ is the $i$'th eigenvalue. Such system is called the Hartree-Fock
system which can be regarded as an approximation of the complex $(M+N)$-body
Schr\"{o}dinger equation originating from the study of a molecular system
made of $M$ nuclei interacting via the Coulomb potential with $N$ electrons.

Historically, the first effort made in this direction began from Hartree
\cite{H} by choosing some particular test functions without considering the
antisymmetry (i.e. the Pauli principle). Subsequently, Fock \cite{F1} and
Slater \cite{Sl1}, to take into account the Pauli principle, proposed
another class of test functions, i.e. the class of Slater determinants. A
further relevant free-electron approximation for the exchange potential $V_{%
\text{ex}}\psi $ is given by Slater \cite{Sl2} (see also Dirac \cite{D1} in
a different context), namely%
\begin{equation}
(V_{\text{ex}}\psi )_{i}=-C\left( \sum\limits_{j=1}^{N}|\psi
_{j}|^{2}\right) ^{1/3}\psi _{i},  \label{1-1}
\end{equation}%
where $C$ is a positive constant.

When $N=1,$ the exchange potential $(V_{\text{ex}}\psi )_{1}=-C|\psi
_{1}|^{2/3}\psi _{1}$ in (\ref{1-1}). If we consider $\psi _{1}$ as a real
function, renaming it as $u,$ and take, for simplicity, $C=1$, then System (%
\ref{1-0}) becomes Schr\"{o}dinger-Poisson-Slater equation as follows:

\begin{equation}
\begin{array}{ll}
-\Delta u+u+\phi _{u}(x)u=|u|^{2/3}u & \text{ in }\mathbb{R}^{3},%
\end{array}
\label{1-3}
\end{equation}%
where%
\begin{equation*}
\phi _{u}(x)=\int_{\mathbb{R}^{3}}\frac{u^{2}(y)}{|x-y|}dy.
\end{equation*}%
It describes the evolution of an electron ensemble in a semiconductor
crystal. S\'{a}nchez and Soler \cite{SS} used a minimization procedure in an
appropriate manifold to find a positive solution of Eq. (\ref{1-3}). If the
term $|u|^{2/3}u$ is replaced with $0$, then Eq. (\ref{1-3}) becomes the Schr%
\"{o}dinger--Poisson equation (also called Schr\"{o}dinger--Maxwell
equation). This type of equation appeared in semiconductor theory and has
been studied in \cite{BF,Lions}, and many others. In some recent works \cite%
{A,AP,R1,R2,SWF,SWF1,SWF2,ZZ}, a local nonlinear term $|u|^{p-2}u$ (or, more
generally, $f(u)$) has been added to the Schr\"{o}dinger--Poisson equation.
Those nonlinear terms have been traditionally used in the Schr\"{o}dinger
equation to model the interaction among particle (possibly nonradial).

In this paper we take $N=2$ and we assume that the exchange potential%
\begin{equation}
V_{\text{ex}}\psi =-C\binom{(\left\vert \psi _{1}\right\vert ^{p-2}+\beta
\left\vert \psi _{1}\right\vert ^{\frac{p}{2}-2}\left\vert \psi
_{2}\right\vert ^{\frac{p}{2}})\psi _{1}}{(\left\vert \psi _{2}\right\vert
^{p-2}+\beta \left\vert \psi _{1}\right\vert ^{\frac{p}{2}}\left\vert \psi
_{2}\right\vert ^{\frac{p}{2}-2})\psi _{2}},  \label{1-4}
\end{equation}%
where $\beta \geq 0$ and $2<p<6$. Note that, for $p=\frac{8}{3},$ (\ref{1-4}%
) becomes%
\begin{equation*}
V_{\text{ex}}\psi =-C\binom{(\left\vert \psi _{1}\right\vert ^{\frac{2}{3}%
}+\beta \left\vert \psi _{1}\right\vert ^{-\frac{2}{3}}\left\vert \psi
_{2}\right\vert ^{\frac{4}{3}})\psi _{1}}{(\left\vert \psi _{2}\right\vert ^{%
\frac{2}{3}}+\beta \left\vert \psi _{1}\right\vert ^{\frac{4}{3}}\left\vert
\psi _{2}\right\vert ^{-\frac{2}{3}})\psi _{2}},  \label{1-6}
\end{equation*}%
which is viewed as an approximation of the exchange potential (\ref{1-1})
proposed by Slater.

Considering $\psi _{1}$ and $\psi _{2}$ real functions, renaming them as $%
u,v,$ and taking, for simplicity, $C=1$, System (\ref{1-0}) becomes the
following%
\begin{equation}
\left\{
\begin{array}{ll}
-\Delta u+u+\lambda \phi _{u,v}u=\left\vert u\right\vert ^{p-2}u+\beta
\left\vert v\right\vert ^{\frac{p}{2}}\left\vert u\right\vert ^{\frac{p}{2}%
-2}u & \text{ in }\mathbb{R}^{3}, \\
-\Delta v+v+\lambda \phi _{u,v}v=\left\vert v\right\vert ^{p-2}v+\beta
\left\vert u\right\vert ^{\frac{p}{2}}\left\vert v\right\vert ^{\frac{p}{2}%
-2}v & \text{ in }\mathbb{R}^{3},%
\end{array}%
\right.  \tag*{$\left( E_{\lambda ,\beta }\right) $}
\end{equation}%
where%
\begin{equation}
\phi _{u,v}(x)=\int_{\mathbb{R}^{3}}\frac{u^{2}(y)+v^{2}\left( y\right) }{%
|x-y|}dy.  \label{1-2}
\end{equation}%
It is easily seen that System $\left( E_{\lambda ,\beta }\right) $ is
variational and its solutions are critical points of the corresponding
energy functional $J_{\lambda ,\beta }:H\rightarrow \mathbb{R}$ defined as%
\begin{equation*}
J_{\lambda ,\beta }(u,v)=\frac{1}{2}\left\Vert \left( u,v\right) \right\Vert
_{H}^{2}+\frac{\lambda }{4}\int_{\mathbb{R}^{3}}\phi _{u,v}\left(
u^{2}+v^{2}\right) dx-\frac{1}{p}\int_{\mathbb{R}^{3}}\left( \left\vert
u\right\vert ^{p}+\left\vert v\right\vert ^{p}+2\beta \left\vert
u\right\vert ^{\frac{p}{2}}\left\vert v\right\vert ^{\frac{p}{2}}\right) dx,
\end{equation*}%
where $\left\Vert \left( u,v\right) \right\Vert _{H}=\left[ \int_{\mathbb{R}%
^{3}}\left( |\nabla u|^{2}+u^{2}+|\nabla v|^{2}+v^{2}\right) dx\right]
^{1/2} $ is the standard norm in $H.$ Clearly, $J_{\lambda ,\beta }$ is a
well-defined and $C^{1}$ functional on $H.$ For a solution $\left(
u,v\right) $ of System $\left( E_{\lambda ,\beta }\right) ,$ we here need to
introduce some concepts of its triviality and positiveness.

\begin{definition}
\label{D1}A vector function $\left( u,v\right) $ is said to be\newline
$\left( i\right) $ nontrivial if either $u\neq 0$ or $v\neq 0;$\newline
$\left( ii\right) $ semitrivial if it is nontrivial but either $u=0$ or $%
v=0; $\newline
$\left( iii\right) $ vectorial if both of $u$ and $v$ are not zero;\newline
$\left( iv\right) $ nonnegative if $u\geq 0$ and $v\geq 0;$\newline
$\left( v\right) $ positive if $u>0$ and $v>0.$
\end{definition}

If $\lambda =0,$ then System $\left( E_{\lambda ,\beta }\right) $ is deduced
to the local weakly coupled nonlinear Schr\"{o}dinger system%
\begin{equation}
\left\{
\begin{array}{ll}
-\Delta u+u=\left\vert u\right\vert ^{p-2}u+\beta \left\vert v\right\vert ^{%
\frac{p}{2}}\left\vert u\right\vert ^{\frac{p}{2}-2}u & \text{ in }\mathbb{R}%
^{3}, \\
-\Delta v+v=\left\vert v\right\vert ^{p-2}v+\beta \left\vert u\right\vert ^{%
\frac{p}{2}}\left\vert v\right\vert ^{\frac{p}{2}-2}v & \text{ in }\mathbb{R}%
^{3},%
\end{array}%
\right.  \label{1-5}
\end{equation}%
which arises in the theory of Bose-Einstein condensates in two different
hyperfine states \cite{T1}. The coupling constant $\beta $ is the
interaction between the two components. As $\beta >0$, the interaction is
attractive, but the interaction is repulsive if $\beta <0$. The existence
and multiplicity of positive solutions for System (\ref{1-5}) have been the
subject of extensive mathematical studies in recent years, for example, \cite%
{AC,BDW,CZ1,CZ2,LW,LW1,MMP}. More efforts have been made on finding
vectorial solutions of the system by controlling the ranges of the parameter
$\beta .$

If $\lambda \neq 0,$ then a characteristic feature of System $\left(
E_{\lambda ,\beta }\right) $ lies on the presence of the double coupled
terms, including a Coulomb interacting term and a cooperative pure power
term. Very recently, based on the method of Nehari-Pohozaev manifold
developed by Ruiz \cite{R1}, d'Avenia, Maia and Siciliano \cite{DMS} firstly
studied the existence of radial ground state solutions for System $(
E_{\lambda ,\beta }),$ depending on the parameters $\beta $ and $p$. To be
precise, for $\lambda >0,$ they concluded that $(i)$ a semitrivial radial
ground state solution exists for $\beta =0$ and $3<p<6,$ or for $0<\beta
<2^{2q-1}-1$ and $4\leq p<6;$ $(ii)$ a vectorial radial ground state
solution exists for $\beta >0$ and $3<p<4,$ or for $\beta \geq 2^{2q-1}-1$
and $4\leq p<6;$ $(iii)$ both semitrivial and vectorial radial ground state
solutions exist for $\beta =2^{2q-1}-1$ and $4\leq p<6.$ It is pointed out
that the definition of ground state solutions involved here is confined to
the space of radial functions $H_{r}:=H_{rad}^{1}(\mathbb{R}^{3})\times
H_{rad}^{1}(\mathbb{R}^{3}),$ namely, a radial ground state solution is a
radial solution of System $\left( E_{\lambda ,\beta }\right) $ whose energy
is minimal among all radial ones.

As we can see, the previous results leave a gap, say, the case $2<p\leq 3$.
We remark that an approximation of the exchange potential (\ref{1-1})
proposed by Slater, i.e. (\ref{1-5}), is included in this gap. The first aim
of this work is to fill this gap and to study nontrivial radial solutions,
including vectorial ones, of System $\left( E_{\lambda ,\beta }\right) $
when $2<p<3.$ On the other hand, we also notice that all nontrivial
solutions are obtained in the radial space $H_{r}$ in \cite{DMS}. In view of
this, the second aim of this work is to find vectorial solutions of System $%
\left( E_{\lambda ,\beta }\right) $ when $2<p<4$ in the space $H:=H^{1}(%
\mathbb{R}^{3})\times H^{1}(\mathbb{R}^{3}).$ And on this basis we shall
further find vectorial ground state solutions in $H,$ which is totally
different from that of \cite{DMS}. In particular, the existence of vectorial
ground state solutions is proved in the case $p=3,$ which seems to be an
very interesting and novel result, even in the study of Schr\"{o}%
dinger--Poisson equations.

Compared with the existing results in \cite{DMS}, there seems to be more
challenging in our study. Firstly, the method of Nehari-Pohozaev manifold
used in \cite{DMS} is not a ideal choice when we deal with the case $2<p\leq
3,$ whether in $H_{r}$ or $H.$ Secondly, we find the interaction effect
between the double coupled terms is significant for the case $2<p\leq 3$. As
a result, the analysis of the internal relationship between the coupling
constants $\lambda $ and $\beta $ is a difficult problem. Thirdly, it is
complicated to determine the vectorial ground state solutions in $H$ for the
case $3\leq p<4$. In order to overcome these considerable difficulties, new
ideas and techniques have been explored. More details will be discussed in
the next subsection.

\subsection{Main results}

First of all, we consider the following maximization problems:
\begin{equation}
\Lambda \left( \beta \right) :=\sup_{u\in H_{r}\setminus \left\{ \left(
0,0\right) \right\} }\frac{\frac{1}{p}\int_{\mathbb{R}^{3}}F_{\beta }\left(
u,v\right) dx-\frac{1}{2}\left\Vert \left( u,v\right) \right\Vert _{H}^{2}}{%
\int_{\mathbb{R}^{3}}\phi _{u,v}\left( u^{2}+v^{2}\right) dx}  \label{1-7}
\end{equation}%
and%
\begin{equation*}
\overline{\Lambda }\left( \beta \right) :=\sup_{u\in H\setminus \left\{
\left( 0,0\right) \right\} }\frac{\int_{\mathbb{R}^{3}}F_{\beta }\left(
u,v\right) dx-\left\Vert \left( u,v\right) \right\Vert _{H}^{2}}{\int_{%
\mathbb{R}^{3}}\phi _{u,v}\left( u^{2}+v^{2}\right) dx},
\end{equation*}%
where $F_{\beta }\left( u,v\right) :=\left\vert u\right\vert ^{p}+\left\vert
v\right\vert ^{p}+2\beta \left\vert u\right\vert ^{\frac{p}{2}}\left\vert
v\right\vert ^{\frac{p}{2}}$ with $2<p<3$ and $\beta \geq 0.$ Then we have
the following proposition.

\begin{proposition}
\label{p1} Let $2<p<3$ and $\beta \geq 0.$ Then we have\newline
$(i)$ $0<\Lambda \left( \beta \right) <\infty $ and $0<\overline{\Lambda }%
\left( \beta \right) <\infty ;$\newline
$(ii)$ $\Lambda \left( \beta \right) $ and $\overline{\Lambda }\left( \beta
\right) $ are both achieved.
\end{proposition}

About its proof, we refer the reader to Theorems \ref{AT-1} and \ref{AT-2}
in Appendix. With the help of Proposition \ref{p1}, we have the following
two theorems.

\begin{theorem}
\label{t0} Let $2<p<3.$ Then for every $\beta \geq 0$ and $\lambda =4\Lambda
\left( \beta \right) ,$ System $(E_{\lambda ,\beta })$ admits two nontrivial
nonnegative radial solutions $\left( u_{\lambda ,\beta }^{\left( 1\right)
},v_{\lambda ,\beta }^{\left( 1\right) }\right) ,\left( u_{\lambda ,\beta
}^{\left( 2\right) },v_{\lambda ,\beta }^{\left( 2\right) }\right) \in
H_{r}\setminus \left\{ \left( 0,0\right) \right\} $ satisfying
\begin{equation*}
J_{\lambda ,\beta }\left( u_{\lambda ,\beta }^{\left( 2\right) },v_{\lambda
,\beta }^{\left( 2\right) }\right) =0<J_{\lambda ,\beta }\left( u_{\lambda
,\beta }^{\left( 1\right) },v_{\lambda ,\beta }^{\left( 1\right) }\right) .
\end{equation*}%
Furthermore, if $\beta >0,$ then $\left( u_{\lambda ,\beta }^{\left(
2\right) },v_{\lambda ,\beta }^{\left( 2\right) }\right) $ is vectorial and
positive.
\end{theorem}

\begin{theorem}
\label{t1} Let $2<p<3$ and $\beta \geq 0.$ Then the following statements are
true.\newline
$\left( i\right) $ For every $0<\lambda <4\Lambda \left( \beta \right) ,$
System $(E_{\lambda ,\beta })$ admits two nontrivial nonnegative radial
solutions $\left( u_{\lambda ,\beta }^{\left( 1\right) },v_{\lambda ,\beta
}^{\left( 1\right) }\right) ,$ $\left( u_{\lambda ,\beta }^{\left( 2\right)
},v_{\lambda ,\beta }^{\left( 2\right) }\right) \in H_{r}$ satisfying%
\begin{equation*}
J_{\lambda ,\beta }\left( u_{\lambda ,\beta }^{\left( 2\right) },v_{\lambda
,\beta }^{\left( 2\right) }\right) <0<J_{\lambda ,\beta }\left( u_{\lambda
,\beta }^{\left( 1\right) },v_{\lambda ,\beta }^{\left( 1\right) }\right) .
\end{equation*}%
Furthermore, if $\beta >0,$ then $\left( u_{\lambda ,\beta }^{\left(
2\right) },v_{\lambda ,\beta }^{\left( 2\right) }\right) $ is vectorial and
positive.\newline
$\left( ii\right) $ For every $\lambda >\overline{\Lambda }\left( \beta
\right) ,$ $\left( u,v\right) =\left( 0,0\right) $ is the unique solution of
System $(E_{\lambda ,\beta }).$
\end{theorem}

In the proofs of Theorems \ref{t0} and \ref{t1}, the key point is to
establish Lions type inequalities in the context of the vector functions
(see (\ref{2-0}) and (\ref{2-00}) below). By using these, together with
Strauss's inequality in $H_{r},$ we can prove that the functional $%
J_{\lambda ,\beta }$ is coercive and bounded below on $H_{r}.$

Next, we focus on vectorial solutions of System $(E_{\lambda ,\beta })$ on $%
H.$ Define
\begin{equation*}
\beta \left( \lambda \right) :=\left\{
\begin{array}{ll}
\max \left\{ \frac{p-2}{2},\left[ 1+\sqrt{1+\frac{2pS_{p}^{2p/\left(
p-2\right) }}{\left( p-2\right) \overline{S}^{2}S_{12/5}^{4}}\left( \frac{2}{%
4-p}\right) ^{\frac{4}{p-2}}\lambda }\right] ^{\left( p-2\right)
/2}-1\right\} , & \text{ if }\lambda <\rho _{p}, \\
\max \left\{ \frac{p-2}{2},\left[ \frac{2\left( 4-p\right) S_{p}^{2p/\left(
p-2\right) }}{\left( p-2\right) \overline{S}^{2}S_{12/5}^{4}}\left( 1+\sqrt{%
1+\frac{p2^{4/\left( p-2\right) }}{\left( 4-p\right) ^{\left( p+2\right)
/\left( p-2\right) }}}\right) \lambda \right] ^{\left( p-2\right)
/2}-1\right\} , & \text{ if }\lambda \geq \rho _{p},%
\end{array}%
\right.
\end{equation*}%
where $S_{p}$ is the best Sobolev constant for the embedding of $H^{1}(%
\mathbb{R}^{3})$ in $L^{p}(\mathbb{R}^{3}),$ $\overline{S}$ is the best
Sobolev constant for the embedding of $D^{1,2}(\mathbb{R}^{3})$ in $L^{6}(%
\mathbb{R}^{3})$ and $\rho _{p}:=\frac{\left( p-2\right) \overline{S}%
^{2}S_{12/5}^{4}}{2(4-p)S_{p}^{2p/\left( p-2\right) }}.$ Then we have the
following results.

\begin{theorem}
\label{t2} Let $2<p<4$ and $\lambda >0$. Then the following statements are
true.\newline
$\left( i\right) $ If $2<p<3,$ then for every $\beta >\beta \left( \lambda
\right) ,$ System $(E_{\lambda ,\beta })$ admits two vectorial positive
solutions $\left( u_{\lambda ,\beta }^{\left( 1\right) },v_{\lambda ,\beta
}^{\left( 1\right) }\right) \in H$ and $\left( u_{\lambda ,\beta }^{\left(
2\right) },v_{\lambda ,\beta }^{\left( 2\right) }\right) \in H_{r}$
satisfying
\begin{equation*}
J_{\lambda ,\beta }\left( u_{\lambda ,\beta }^{\left( 2\right) },v_{\lambda
,\beta }^{\left( 2\right) }\right) <0<J_{\lambda ,\beta }\left( u_{\lambda
,\beta }^{\left( 1\right) },v_{\lambda ,\beta }^{\left( 1\right) }\right);
\end{equation*}%
\newline
$\left( ii\right) $ If $3\leq p<4,$ then for every $\beta >\beta \left(
\lambda \right) ,$ System $(E_{\lambda ,\beta })$ admits a vectorial
positive solution $\left( u_{\lambda ,\beta }^{\left( 1\right) },v_{\lambda
,\beta }^{\left( 1\right) }\right) \in H$ satisfying $J_{\lambda ,\beta
}\left( u_{\lambda ,\beta }^{\left( 1\right) },v_{\lambda ,\beta }^{\left(
1\right) }\right) >0.$
\end{theorem}

We note that the arguments in Theorems \ref{t0} and \ref{t1} are
inapplicable to Theorem \ref{t2}, since the functional $J_{\lambda ,\beta }$
is restricted to the space $H.$ In view of this, we expect to find critical
points by applying a novel constraint method introduced by us, together with
some new analysis techniques.

Finally, we establish the existence of vectorial ground state solution of
System $(E_{\lambda ,\beta }).$

\begin{theorem}
\label{t3} Let $3\leq p<4$ and $\lambda >0.$ Then for every%
\begin{equation*}
0<\lambda <\lambda _{0}:=\frac{6p\sqrt{3p}\left( p-2\right) \pi }{8\sqrt[3]{2%
}\left( 4-p\right) \left( 6-p\right) ^{3/2}S_{p}^{2p/\left( p-2\right) }}
\end{equation*}%
and $\beta >\beta \left( \lambda \right) ,$ System $(E_{\lambda ,\beta })$
admits a vectorial ground state solution $\left( u_{\lambda ,\beta
},v_{\lambda ,\beta }\right) \in H$ satisfying $J_{\lambda ,\beta }\left(
u_{\lambda ,\beta },v_{\lambda ,\beta }\right) >0.$
\end{theorem}

\begin{theorem}
\label{t3-2} Let $3.18\approx \frac{1+\sqrt{73}}{3}\leq p<4$ and $\lambda
>0. $ Then for every $\beta >\beta \left( \lambda \right) ,$ System $%
(E_{\lambda ,\beta })$ admits a vectorial ground state solution $\left(
u_{\lambda ,\beta },v_{\lambda ,\beta }\right) \in H$ satisfying $J_{\lambda
,\beta }\left( u_{\lambda ,\beta },v_{\lambda ,\beta }\right) >0.$
\end{theorem}

The study of the vectorial ground state solution is considered by us from
different perspectives. In Theorem \ref{t3} we analyze the energy levels of
the solutions by controlling the range of $\lambda $, and in Theorem \ref%
{t3-2} we locate the solutions by reducing the scope of $p.$

The rest of this paper is organized as follows. After introducing some
preliminary results in Section 2, we give the proofs of Theorems \ref{t0}
and \ref{t1} in Section 3. In Section 4, we prove Theorem \ref{t2}. Finally,
we give the proofs of Theorems \ref{t3} and \ref{t3-2} in Section 5.

\section{Preliminary results}

\begin{lemma}
\label{L2-1}Let $2<p<4$ and $\beta >0.$ Let $g_{\beta }\left( s\right) =s^{%
\frac{p}{2}}+\left( 1-s\right) ^{\frac{p}{2}}+2\beta s^{\frac{p}{4}}\left(
1-s\right) ^{\frac{p}{4}}$ for $s\in \left[ 0,1\right] .$ Then there exists $%
s_{\beta }\in \left( 0,1\right) $ such that $g_{\beta }\left( s_{\beta
}\right) =\max_{s\in \left[ 0,1\right] }g_{\beta }\left( s\right) >1.$ In
particular, if $\beta \geq \frac{p-2}{2},$ then $s_{\beta }=\frac{1}{2}.$
\end{lemma}

\begin{proof}
The proof is similar to the argument in \cite[Lemma 2.4]{DMS}, and we omit
it here.
\end{proof}

\begin{lemma}
\label{L2-2}Let $2<p<4,\lambda >0$ and $\beta >0.$ Then for each $z\in
H^{1}\left( \mathbb{R}^{3}\right) \backslash \left\{ 0\right\} $, there
exists $s_{z}\in \left( 0,1\right) $ such that
\begin{equation*}
J_{\lambda ,\beta }\left( \sqrt{s_{z}}z,\sqrt{1-s_{z}}z\right) <J_{\lambda
,\beta }\left( z,0\right) =J_{\lambda ,\beta }\left( 0,z\right) =I_{\lambda
}\left( z\right) ,
\end{equation*}%
where
\begin{equation*}
I_{\lambda }(z):=\frac{1}{2}\int_{\mathbb{R}^{3}}\left( |\nabla
z|^{2}+z^{2}\right) dx+\frac{\lambda }{4}\int_{\mathbb{R}^{3}}\phi
_{z}z^{2}dx-\frac{1}{p}\int_{\mathbb{R}^{3}}\left\vert z\right\vert ^{p}dx.
\end{equation*}
\end{lemma}

\begin{proof}
Let $\left( u,v\right) =\left( \sqrt{s}z,\sqrt{1-s}z\right) $ for $z\in
H^{1}\left( \mathbb{R}^{3}\right) \backslash \left\{ 0\right\} $ and $s\in
[0,1].$ A direct calculation shows that%
\begin{equation*}
\left\Vert \left( u,v\right) \right\Vert _{H}^{2}=s\left\Vert z\right\Vert
_{H^{1}}^{2}+\left( 1-s\right) \left\Vert z\right\Vert
_{H^{1}}^{2}=\left\Vert z\right\Vert _{H^{1}}^{2}
\end{equation*}%
and%
\begin{equation*}
\int_{\mathbb{R}^{3}} \left( u^{2}+v^{2}\right) \phi _{u,v}dx=\int_{\mathbb{R%
}^{3}} \left( sz^{2}+\left( 1-s\right) z^{2}\right) \phi _{u,v}dx=\int_{%
\mathbb{R}^{3}} \phi _{z}z^{2}dx.
\end{equation*}%
Moreover, by Lemma \ref{L2-1}, there exists $s_{z}\in \left( 0,1\right) $
such that
\begin{equation*}
\int_{\mathbb{R}^{3}}\left(\left\vert u\right\vert ^{p}+\left\vert
v\right\vert ^{p}+2\beta\left\vert u\right\vert ^{\frac{p}{2}}\left\vert
v\right\vert ^{\frac{p}{2}}\right)dx=\left[ s_{z}^{\frac{p}{2}}+\left(
1-s_{z}\right) ^{\frac{p}{2}}+2\beta s_{z}^{\frac{p}{4}}\left(
1-s_{z}\right) ^{\frac{p}{4}}\right] \int_{\mathbb{R}^{3}}\left\vert
z\right\vert ^{p}dx>\int_{\mathbb{R}^{3}}\left\vert z\right\vert ^{p}dx.
\end{equation*}%
Thus, we have
\begin{eqnarray*}
J_{\lambda ,\beta }\left( \sqrt{s_{z}}z,\sqrt{1-s_{z}}z\right) &=&\frac{1}{2}%
\left\Vert z\right\Vert _{H^{1}}^{2}+\frac{\lambda }{4}\int_{\mathbb{R}%
^{3}}\phi _{z}z^{2}dx-\frac{1}{p}\left[ s_{z}^{\frac{p}{2}}+\left(
1-s_{z}\right) ^{\frac{p}{2}}+2\beta s_{z}^{\frac{p}{4}}\left(
1-s_{z}\right) ^{\frac{p}{4}}\right] \int_{\mathbb{R}^{3}}\left\vert
z\right\vert ^{p}dx \\
&<&\frac{1}{2}\left\Vert z\right\Vert _{H^{1}}^{2}+\frac{\lambda }{4}\int_{%
\mathbb{R}^{3}}\phi _{z}z^{2}dx-\frac{1}{p}\int_{\mathbb{R}^{3}}\left\vert
z\right\vert ^{p}dx \\
&=&J_{\lambda ,\beta }\left( z,0\right) =J_{\lambda ,\beta }\left(
0,z\right) =I_{\lambda }\left( z\right) .
\end{eqnarray*}%
The proof is complete.
\end{proof}

By vitue of Lemma \ref{L2-2}, we have the following result.

\begin{theorem}
\label{t5}Let $2<p<3,\lambda >0$ and $\beta >0.$ Let $\left(
u_{0},v_{0}\right) \in H_{r}$ be a minimizer of the minimum problem $%
\inf_{\left( u,v\right) \in H_{r}}J_{\lambda ,\beta }(u,v)$ such that $%
J_{\lambda ,\beta }(u_{0},v_{0})<0.$ Then we have $u_{0}\neq 0$ and $%
v_{0}\neq 0.$
\end{theorem}

The function $\phi _{u,v}$ defined as $\left( \ref{1-2}\right) $ possesses
certain properties \cite{AP,R1}.

\begin{lemma}
\label{L2-3}For each $\left( u,v\right) \in H$, the following two
inequalities are true.
\end{lemma}

\begin{itemize}
\item[$\left( i\right) $] $\phi _{u,v}\geq 0;$

\item[$\left( ii\right) $] $\int_{\mathbb{R}^{3}}\phi _{u,v}\left(
u^{2}+v^{2}\right) dx\leq \overline{S}^{-2}S_{12/5}^{-4}\left\Vert \left(
u,v\right) \right\Vert _{H}^{4}.$
\end{itemize}

Following the idea of Lions \cite{Lions}, we have%
\begin{eqnarray}
\frac{\sqrt{\lambda }}{4}\int_{\mathbb{R}^{3}}(\left\vert u\right\vert
^{3}+v^{2}\left\vert u\right\vert )dx &=&\frac{\sqrt{\lambda }}{4}\int_{%
\mathbb{R}^{3}}\left( -\Delta \phi _{u,v}\right) \left\vert u\right\vert dx
\notag \\
&=&\frac{\sqrt{\lambda }}{4}\int_{\mathbb{R}^{3}}\left\langle \nabla \phi
_{u,v},\nabla \left\vert u\right\vert \right\rangle dx  \notag \\
&\leq &\frac{1}{4}\int_{\mathbb{R}^{3}}\left\vert \nabla u\right\vert ^{2}dx+%
\frac{\lambda }{16}\int_{\mathbb{R}^{3}}\phi _{u,v}\left( u^{2}+v^{2}\right)
dx  \label{2-0}
\end{eqnarray}%
and%
\begin{eqnarray}
\frac{\sqrt{\lambda }}{4}\int_{\mathbb{R}^{3}}(u^{2}\left\vert v\right\vert
+\left\vert v\right\vert ^{3})dx &=&\frac{\sqrt{\lambda }}{4}\int_{\mathbb{R}%
^{3}}\left( -\Delta \phi _{u,v}\right) \left\vert v\right\vert dx  \notag \\
&=&\frac{\sqrt{\lambda }}{4}\int_{\mathbb{R}^{3}}\left\langle \nabla \phi
_{u,v},\nabla \left\vert v\right\vert \right\rangle dx  \notag \\
&\leq &\frac{1}{4}\int_{\mathbb{R}^{3}}\left\vert \nabla v\right\vert ^{2}dx+%
\frac{\lambda }{16}\int_{\mathbb{R}^{3}}\phi _{u,v}\left( u^{2}+v^{2}\right)
dx  \label{2-00}
\end{eqnarray}%
for all $\left( u,v\right) \in H_{r},$ which imply that%
\begin{eqnarray}
J_{\lambda ,\beta }(u,v) &=&\frac{1}{2}\left\Vert \left( u,v\right)
\right\Vert _{H}^{2}+\frac{\lambda }{4}\int_{\mathbb{R}^{3}}\phi
_{u,v}\left( u^{2}+v^{2}\right) dx  \notag \\
&&-\frac{1}{p}\int_{\mathbb{R}^{3}}\left( \left\vert u\right\vert
^{p}+\left\vert v\right\vert ^{p}+2\beta \left\vert u\right\vert ^{\frac{p}{2%
}}\left\vert v\right\vert ^{\frac{p}{2}}\right) dx  \notag \\
&\geq &\frac{1}{4}\left\Vert \left( u,v\right) \right\Vert _{H}^{2}+\frac{1}{%
4}\int_{\mathbb{R}^{3}}(u^{2}+v^{2})dx+\frac{\lambda }{8}\int_{\mathbb{R}%
^{3}}\phi _{u,v}\left( u^{2}+v^{2}\right) dx  \notag \\
&&+\frac{\sqrt{\lambda }}{4}\int_{\mathbb{R}^{3}}(\left\vert u\right\vert
^{3}+\left\vert v\right\vert ^{3})dx-\frac{1}{p}\int_{\mathbb{R}^{3}}\left(
\left\vert u\right\vert ^{p}+\left\vert v\right\vert ^{p}+2\beta \left\vert
u\right\vert ^{\frac{p}{2}}\left\vert v\right\vert ^{\frac{p}{2}}\right) dx
\notag \\
&=&\frac{1}{4}\left\Vert \left( u,v\right) \right\Vert _{H}^{2}+\frac{%
\lambda }{8}\int_{\mathbb{R}^{3}}\phi _{u,v}\left( u^{2}+v^{2}\right) dx
\notag \\
&&+\int_{\mathbb{R}^{3}}\left( \frac{1}{4}u^{2}+\frac{\sqrt{\lambda }}{4}%
\left\vert u\right\vert ^{3}-\frac{1+\beta }{p}\left\vert u\right\vert
^{p}\right) dx  \notag \\
&&+\int_{\mathbb{R}^{3}}\left( \frac{1}{4}v^{2}+\frac{\sqrt{\lambda }}{4}%
\left\vert v\right\vert ^{3}-\frac{1+\beta }{p}\left\vert v\right\vert
^{p}\right) dx.  \label{2-1}
\end{eqnarray}%
Then we have the following results.

\begin{lemma}
\label{m5}Let $2<p<3,\lambda >0$ and $\beta \geq 0.$ Then $J_{\lambda ,\beta
}$ is coercive and bounded below on $H_{r}.$
\end{lemma}

\begin{proof}
By $\left( \ref{2-1}\right) $ and applying the argument in Ruiz \cite[%
Theorem 4.3]{R1}, $J_{\lambda ,\beta }$ is coercive on $H_{r}$ and there
exists $M>0$ such that%
\begin{equation*}
\inf_{\left( u,v\right) \in H_{r}}J_{\lambda ,\beta }(u,v)\geq -M.
\end{equation*}%
This completes the proof.
\end{proof}

\section{Proofs of Theorems \protect\ref{t0} and \protect\ref{t1}}

\textbf{We are now ready to prove Theorem \ref{t0}. }By Theorem \ref{AT-1},
there exists $\left( u_{\lambda ,\beta }^{\left( 2\right) },v_{\lambda
,\beta }^{\left( 2\right) }\right) \in H_{r}\setminus \left\{ \left(
0,0\right) \right\} $ such that
\begin{equation*}
\frac{\frac{1}{p}\int_{\mathbb{R}^{3}}F_{\beta }\left( u_{\lambda ,\beta
}^{\left( 2\right) },v_{\lambda ,\beta }^{\left( 2\right) }\right) dx-\frac{1%
}{2}\left\Vert \left( u_{\lambda ,\beta }^{\left( 2\right) },v_{\lambda
,\beta }^{\left( 2\right) }\right) \right\Vert _{H}^{2}}{\int_{\mathbb{R}%
^{3}}\phi _{u_{\lambda ,\beta }^{\left( 2\right) },v_{\lambda ,\beta
}^{\left( 2\right) }}\left( \left[ u_{\lambda ,\beta }^{\left( 2\right) }%
\right] ^{2}+\left[ v_{\lambda ,\beta }^{\left( 2\right) }\right]
^{2}\right) dx}=\Lambda (\beta ).
\end{equation*}%
It follows that%
\begin{equation*}
\frac{\left\langle J_{4\Lambda \left( \beta \right) ,\beta }^{\prime }\left(
u_{\lambda ,\beta }^{\left( 2\right) },v_{\lambda ,\beta }^{\left( 2\right)
}\right) ,\left( \phi ,\psi \right) \right\rangle }{\int_{\mathbb{R}%
^{3}}\phi _{u_{\lambda ,\beta }^{\left( 2\right) },v_{\lambda ,\beta
}^{\left( 2\right) }}\left( \left[ u_{\lambda ,\beta }^{\left( 2\right) }%
\right] ^{2}+\left[ v_{\lambda ,\beta }^{\left( 2\right) }\right]
^{2}\right) dx}=0\text{ for all }\left( \phi ,\psi \right) \in
H_{r}\setminus \{\left( 0,0\right) \}.
\end{equation*}%
Moreover, by Palais\ criticality principle \cite{P}, we have
\begin{equation*}
\frac{\left\langle J_{4\Lambda \left( \beta \right) ,\beta }^{\prime }\left(
u_{\lambda ,\beta }^{\left( 2\right) },v_{\lambda ,\beta }^{\left( 2\right)
}\right) ,\left( \phi ,\psi \right) \right\rangle }{\int_{\mathbb{R}%
^{3}}\phi _{u_{\lambda ,\beta }^{\left( 2\right) },v_{\lambda ,\beta
}^{\left( 2\right) }}\left( \left[ u_{\lambda ,\beta }^{\left( 2\right) }%
\right] ^{2}+\left[ v_{\lambda ,\beta }^{\left( 2\right) }\right]
^{2}\right) dx}=0\text{ for all }\left( \phi ,\psi \right) \in H\setminus
\{\left( 0,0\right) \}.
\end{equation*}%
Hence, $\left( u_{\lambda ,\beta }^{\left( 2\right) },v_{\lambda ,\beta
}^{\left( 2\right) }\right) $ is a critical point of $J_{4\Lambda \left(
\beta \right) ,\beta }$ for $\beta \geq 0$ and $J_{4\Lambda \left( \beta
\right) ,\beta }\left( u_{\lambda ,\beta }^{\left( 2\right) },v_{\lambda
,\beta }^{\left( 2\right) }\right) =0,$ then so is $\left( \left\vert
u_{\lambda ,\beta }^{\left( 2\right) }\right\vert ,\left\vert v_{\lambda
,\beta }^{\left( 2\right) }\right\vert \right) .$ Thus, we may assume that $%
\left( u_{\lambda ,\beta }^{\left( 2\right) },v_{\lambda ,\beta }^{\left(
2\right) }\right) $ is a nonnegative nontrivial critical point of $%
J_{\lambda ,\beta }.$ Next, we claim that $u_{\lambda ,\beta }^{\left(
2\right) }\neq 0$ and $v_{\lambda ,\beta }^{\left( 2\right) }\neq 0$ for $%
\beta >0.$ If not, we may assume that $v_{\lambda ,\beta }^{\left( 2\right)
}\equiv 0.$ Then by Lemma \ref{L2-2}, there exists $s_{0}\in \left(
0,1\right) $ such that $\left( \sqrt{s_{0}}u_{\lambda ,\beta }^{\left(
2\right) },\sqrt{1-s_{0}}u_{\lambda ,\beta }^{\left( 2\right) }\right) \in
H_{r}$ and%
\begin{equation*}
J_{4\Lambda \left( \beta \right) ,\beta }\left( \sqrt{s_{0}}u_{\lambda
,\beta }^{\left( 2\right) },\sqrt{1-s_{0}}u_{\lambda ,\beta }^{\left(
2\right) }\right) <J_{4\Lambda \left( \beta \right) ,\beta }\left(
u_{\lambda ,\beta }^{\left( 2\right) },0\right) =J_{4\Lambda \left( \beta
\right) ,\beta }\left( 0,u_{\lambda ,\beta }^{\left( 2\right) }\right)
=\alpha _{4\Lambda \left( \beta \right) ,\beta },
\end{equation*}%
which is a contradiction. Moreover, it follows from the Sobolev embedding
theorem that%
\begin{eqnarray*}
J_{4\Lambda \left( \beta \right) ,\beta }(u,v) &\geq &\frac{1}{2}\left\Vert
\left( u,v\right) \right\Vert _{H}^{2}-\frac{C_{\beta }}{p}\int_{\mathbb{R}%
^{3}}(\left\vert u\right\vert ^{p}+\left\vert v\right\vert ^{p})dx \\
&\geq &\frac{1}{2}\left\Vert \left( u,v\right) \right\Vert _{H}^{2}-\frac{%
C_{\beta }}{pS_{p}^{p}}\left\Vert \left( u,v\right) \right\Vert _{H}^{p}%
\text{ for all }\left( u,v\right) \in H_{r},
\end{eqnarray*}%
which implies that there exist $\eta ,\kappa >0$ such that $\Vert \left(
u,v\right) \Vert _{H}>\eta $ and%
\begin{equation*}
\max \{J_{4\Lambda \left( \beta \right) ,\beta }(0,0),J_{4\Lambda \left(
\beta \right) ,\beta }\left( u_{\lambda ,\beta }^{\left( 2\right)
},v_{\lambda ,\beta }^{\left( 2\right) }\right) \}=0<\kappa \leq \inf_{\Vert
\left( u,v\right) \Vert _{H}=\eta }J_{4\Lambda \left( \beta \right) ,\beta
}(u,v).
\end{equation*}%
Define%
\begin{equation*}
\theta _{4\Lambda \left( \beta \right) ,\beta }=\inf_{\gamma \in \Gamma
}\max_{0\leq \tau \leq 1}J_{4\Lambda \left( \beta \right) ,\beta }(\gamma
(\tau )),
\end{equation*}%
where $\Gamma =\left\{ \gamma \in C([0,1],H_{r}):\gamma (0)=\left(
0,0\right) ,\gamma (1)=\left( u_{\lambda ,\beta }^{\left( 2\right)
},v_{\lambda ,\beta }^{\left( 2\right) }\right) \right\} .$ Then by the
mountain pass theorem \cite{E2,R} and Palais\ criticality principle, there
exists a sequence $\{\left( u_{n},v_{n}\right) \}\subset H_{r}$ such that
\begin{equation*}
J_{4\Lambda \left( \beta \right) ,\beta }\left( u_{n},v_{n}\right)
\rightarrow \theta _{4\Lambda \left( \beta \right) ,\beta }\geq \kappa \quad
\text{and}\quad \Vert J_{4\Lambda \left( \beta \right) ,\beta }^{\prime
}\left( u_{n},v_{n}\right) \Vert _{H^{-1}}\rightarrow 0\quad \text{as}\
n\rightarrow \infty ,
\end{equation*}%
and using an argument similar to that in \cite[Theorem 4.3]{R1}, there exist
a subsequence $\{\left( u_{n},v_{n}\right) \}$ and $\left( u_{\lambda ,\beta
}^{\left( 1\right) },v_{\lambda ,\beta }^{\left( 1\right) }\right) \in
H_{r}\setminus \{\left( 0,0\right) \}$ such that $\left( u_{n},v_{n}\right)
\rightarrow \left( u_{\lambda ,\beta }^{\left( 1\right) },v_{\lambda ,\beta
}^{\left( 1\right) }\right) $ strongly in $H_{r}$ and $\left( u_{\lambda
,\beta }^{\left( 1\right) },v_{\lambda ,\beta }^{\left( 1\right) }\right) $
is a solution of System $(E_{4\Lambda \left( \beta \right) ,\beta }).$ This
indicates that
\begin{equation*}
J_{4\Lambda \left( \beta \right) ,\beta }\left( u_{\lambda ,\beta }^{\left(
1\right) },v_{\lambda ,\beta }^{\left( 1\right) }\right) =\theta _{4\Lambda
\left( \beta \right) ,\beta }\geq \kappa >0.
\end{equation*}%
The proof is complete.

\textbf{We are now ready to prove Theorem \ref{t1}. }$\left( i\right) $ By
Theorem \ref{AT-1}, there exists $\left( u_{0},v_{0}\right) \in
H_{r}\setminus \{\left( 0,0\right) \}$ such that
\begin{equation*}
\frac{\frac{1}{p}\int_{\mathbb{R}^{3}}F_{\beta }\left( u_{0},v_{0}\right) dx-%
\frac{1}{2}\left\Vert \left( u_{0},v_{0}\right) \right\Vert _{H}^{2}}{\int_{%
\mathbb{R}^{3}}\phi _{u_{0},v_{0}}\left( u_{0}^{2}+v_{0}^{2}\right) dx}%
=\Lambda (\beta ).
\end{equation*}%
This implies that for each $\lambda <4\Lambda \left( \beta \right) ,$
\begin{equation}
J_{\lambda ,\beta }\left( u_{0},v_{0}\right) =\frac{1}{2}\left\Vert \left(
u_{0},v_{0}\right) \right\Vert _{H}^{2}+\frac{\lambda }{4}\int_{\mathbb{R}%
^{3}}\phi _{u_{0},v_{0}}\left( u_{0}^{2}+v_{0}^{2}\right) dx-\frac{1}{p}%
\int_{\mathbb{R}^{3}}F_{\beta }\left( u_{0},v_{0}\right) dx<0.  \label{3-1}
\end{equation}%
Using $\left( \ref{3-1}\right) $, together with Lemma \ref{m5}, we have%
\begin{equation*}
-\infty <\alpha _{\lambda ,\beta }:=\inf_{\left( u,v\right) \in
H_{r}}J_{\lambda ,\beta }(u,v)<0.
\end{equation*}%
Then by the Ekeland variational principle \cite{E} and Palais\ criticality
principle \cite{P}, there exists a sequence $\{\left( u_{n},v_{n}\right)
\}\subset H_{r}$ such that
\begin{equation*}
J_{\lambda ,\beta }(u_{n},v_{n})=\alpha _{\lambda ,\beta }+o(1)\text{ and }%
J_{\lambda ,\beta }^{\prime }(u_{n},v_{n})=o(1)\text{ in }H^{-1}.
\end{equation*}%
Again, adopting the argument used in \cite[Theorem 4.3]{R1}, there exist a
subsequence $\{\left( u_{n},v_{n}\right) \}$ and $\left( u_{\lambda ,\beta
}^{\left( 2\right) },v_{\lambda ,\beta }^{\left( 2\right) }\right) \in
H_{r}\setminus \{\left( 0,0\right) \}$ such that $\left( u_{n},v_{n}\right)
\rightarrow \left( u_{\lambda ,\beta }^{\left( 2\right) },v_{\lambda ,\beta
}^{\left( 2\right) }\right) $ strongly in $H_{r}$ and $\left( u_{\lambda
,\beta }^{\left( 2\right) },v_{\lambda ,\beta }^{\left( 2\right) }\right) $
is a nontrivial critical point of $J_{\lambda ,\beta }.$ This indicates that
\begin{equation*}
J_{\lambda ,\beta }\left( u_{\lambda ,\beta }^{\left( 2\right) },v_{\lambda
,\beta }^{\left( 2\right) }\right) =\alpha _{\lambda ,\beta }=\inf_{\left(
u,v\right) \in H_{r}}J_{\lambda ,\beta }(u,v)<0,
\end{equation*}%
then so is $\left( \left\vert u_{\lambda ,\beta }^{\left( 2\right)
}\right\vert ,\left\vert v_{\lambda ,\beta }^{\left( 2\right) }\right\vert
\right) .$ Thus, we may assume that $\left( u_{\lambda ,\beta }^{\left(
2\right) },v_{\lambda ,\beta }^{\left( 2\right) }\right) $ is a nonnegative
nontrivial critical point of $J_{\lambda ,\beta }$. Next, we claim that $%
u_{\lambda ,\beta }^{\left( 2\right) }\neq 0$ and $v_{\lambda ,\beta
}^{\left( 2\right) }\neq 0$ for $\beta >0.$ If not, we may assume that $%
v_{\lambda ,\beta }^{\left( 2\right) }\equiv 0.$ Then by Lemma \ref{L2-2},
there exists $s_{\lambda }\in \left( 0,1\right) $ such that $\left( \sqrt{%
s_{\lambda }}u_{\lambda ,\beta }^{\left( 2\right) },\sqrt{1-s_{\lambda }}%
u_{\lambda ,\beta }^{\left( 2\right) }\right) \in H_{r}$ and%
\begin{equation*}
J_{\lambda ,\beta }\left( \sqrt{s_{\lambda }}u_{\lambda ,\beta }^{\left(
2\right) },\sqrt{1-s_{\lambda }}u_{\lambda ,\beta }^{\left( 2\right)
}\right) <J_{\lambda ,\beta }\left( u_{\lambda ,\beta }^{\left( 2\right)
},0\right) =J_{\lambda ,\beta }\left( 0,u_{\lambda ,\beta }^{\left( 2\right)
}\right) =\alpha _{\lambda ,\beta },
\end{equation*}%
which is a contradiction. Moreover, by the Sobolev embedding theorem, we have%
\begin{eqnarray*}
J_{\lambda ,\beta }(u,v) &\geq &\frac{1}{2}\left\Vert \left( u,v\right)
\right\Vert _{H}^{2}-\frac{C_{\beta }}{p}\int_{\mathbb{R}^{3}}(\left\vert
u\right\vert ^{p}+\left\vert v\right\vert ^{p})dx \\
&\geq &\frac{1}{2}\left\Vert \left( u,v\right) \right\Vert _{H}^{2}-\frac{%
C_{\beta }}{pS_{p}^{p}}\left\Vert \left( u,v\right) \right\Vert _{H}^{p}%
\text{ for all }\left( u,v\right) \in H_{r}.
\end{eqnarray*}%
This implies that there exist $\eta ,\kappa >0$ such that $\Vert \left(
u_{\lambda ,\beta }^{\left( 2\right) },v_{\lambda ,\beta }^{\left( 2\right)
}\right) \Vert _{H}>\eta $ and%
\begin{equation*}
\max \left\{ J_{\lambda ,\beta }(0,0),J_{\lambda ,\beta }\left( u_{\lambda
,\beta }^{\left( 2\right) },v_{\lambda ,\beta }^{\left( 2\right) }\right)
\right\} =0<\kappa \leq \inf_{\Vert \left( u,v\right) \Vert _{H}=\eta
}J_{\lambda ,\beta }(u,v).
\end{equation*}%
Define%
\begin{equation*}
\theta _{\lambda ,\beta }=\inf_{\gamma \in \Gamma }\max_{0\leq \tau \leq
1}J_{\lambda ,\beta }(\gamma (\tau )),
\end{equation*}%
where $\Gamma =\left\{ \gamma \in C([0,1],H_{r}):\gamma (0)=\left(
0,0\right) ,\gamma (1)=\left( u_{\lambda ,\beta }^{\left( 2\right)
},v_{\lambda ,\beta }^{\left( 2\right) }\right) \right\} .$ Then by the
mountain pass theorem \cite{E2,R} and Palais\ criticality principle, there
exists a sequence $\{\left( u_{n},v_{n}\right) \}\subset H_{r}$ such that
\begin{equation*}
J_{\lambda ,\beta }\left( u_{n},v_{n}\right) \rightarrow \theta _{\lambda
,\beta }\geq \kappa \quad \text{and}\quad \Vert J_{\lambda ,\beta }^{\prime
}\left( u_{n},v_{n}\right) \Vert _{H^{-1}}\rightarrow 0\quad \text{as}\
n\rightarrow \infty ,
\end{equation*}%
and using an argument similar to that in \cite[Theorem 4.3]{R1}, there exist
a subsequence $\{\left( u_{n},v_{n}\right) \}$ and $\left( u_{\lambda ,\beta
}^{\left( 1\right) },v_{\lambda ,\beta }^{\left( 1\right) }\right) \in
H_{r}\setminus \{\left( 0,0\right) \}$ such that $\left( u_{n},v_{n}\right)
\rightarrow \left( u_{\lambda ,\beta }^{\left( 1\right) },v_{\lambda ,\beta
}^{\left( 1\right) }\right) $ strongly in $H_{r}$ and $\left( u_{\lambda
,\beta }^{\left( 1\right) },v_{\lambda ,\beta }^{\left( 1\right) }\right) $
is a solution of System $(E_{\lambda ,\beta }).$ This indicates that
\begin{equation*}
J_{\lambda ,\beta }\left( u_{\lambda ,\beta }^{\left( 1\right) },v_{\lambda
,\beta }^{\left( 1\right) }\right) =\theta _{\lambda ,\beta }\geq \kappa >0.
\end{equation*}%
\newline
$\left( ii\right) $ Suppose on the contrary. Let $\left( u_{0},v_{0}\right) $
be a nontrivial solution of System $(E_{\lambda ,\beta }).$ Then according
to the definition of $\overline{\Lambda }\left( \beta \right) ,$ for $\beta
\geq 0$ and $\lambda >\overline{\Lambda }\left( \beta \right) ,$ we have%
\begin{eqnarray*}
0 &=&\left\Vert \left( u_{0},v_{0}\right) \right\Vert _{H}^{2}+\lambda \int_{%
\mathbb{R}^{3}}\phi _{u_{0},v_{0}}\left( u_{0}^{2}+v_{0}^{2}\right) dx-\int_{%
\mathbb{R}^{3}}F_{\beta }\left( u_{0},v_{0}\right) dx \\
&>&\left\Vert \left( u_{0},v_{0}\right) \right\Vert _{H}^{2}+\overline{%
\Lambda }\left( \beta \right) \int_{\mathbb{R}^{3}}\phi _{u_{0},v_{0}}\left(
u_{0}^{2}+v_{0}^{2}\right) dx-\int_{\mathbb{R}^{3}}F_{\beta }\left(
u_{0},v_{0}\right) dx\geq 0,
\end{eqnarray*}%
which is a contradiction. The proof is complete.

\section{Proof of Theorem \protect\ref{t2}}

Define the Nehari manifold
\begin{equation*}
\mathbf{M}_{\lambda ,\beta }:=\{\left( u,v\right) \in H\backslash \{\left(
0,0\right) \}:\left\langle J_{\lambda ,\beta }^{\prime }\left( u,v\right)
,\left( u,v\right) \right\rangle =0\}.
\end{equation*}%
Then $u\in \mathbf{M}_{\lambda ,\beta }$ if and only if
\begin{equation*}
\left\Vert \left( u,v\right) \right\Vert _{H}^{2}+\lambda \int_{\mathbb{R}%
^{3}}\phi _{u,v}\left( u^{2}+v^{2}\right) dx-\int_{\mathbb{R}^{3}}\left(
|u|^{p}+\left\vert v\right\vert ^{p}+2\beta |u|^{\frac{^{p}}{2}}\left\vert
v\right\vert ^{\frac{p}{2}}\right) dx=0.
\end{equation*}%
It follows the Sobolev and Young inequalities that%
\begin{eqnarray*}
\left\Vert \left( u,v\right) \right\Vert _{H}^{2} &\leq &\left\Vert \left(
u,v\right) \right\Vert _{H}^{2}+\lambda \int_{\mathbb{R}^{3}}\phi
_{u,v}\left( u^{2}+v_{2}\right) dx \\
&=&\int_{\mathbb{R}^{3}}\left( |u|^{p}+\left\vert v\right\vert ^{p}+2\beta
|u|^{\frac{^{p}}{2}}\left\vert v\right\vert ^{\frac{p}{2}}\right) dx \\
&\leq &C_{\beta }\left\Vert \left( u,v\right) \right\Vert _{H}^{p}\text{ for
all }u\in \mathbf{M}_{\lambda ,\beta }.
\end{eqnarray*}%
So it leads to
\begin{equation}
\left\Vert \left( u,v\right) \right\Vert _{H}\geq C_{\beta }^{-1/\left(
p-2\right) }\text{ for all }u\in \mathbf{M}_{\lambda ,\beta }.  \label{2-2}
\end{equation}

The Nehari manifold $\mathbf{M}_{\lambda ,\beta }$ is closely linked to the
behavior of the function of the form $h_{\lambda ,\left( u,v\right)
}:t\rightarrow J_{\lambda ,\beta }\left( tu,tv\right) $ for $t>0.$ Such maps
are known as fibering maps introduced by Dr\'{a}bek-Pohozaev \cite{DP}, and
were further discussed by Brown-Zhang \cite{BZ} and Brown-Wu \cite{BW1,BW2}.
For $\left( u,v\right) \in H,$ we find that%
\begin{eqnarray*}
h_{\lambda ,\left( u,v\right) }\left( t\right) &=&\frac{t^{2}}{2}\left\Vert
\left( u,v\right) \right\Vert _{H}^{2}+\frac{\lambda t^{4}}{4}\int_{\mathbb{R%
}^{3}}\phi _{u,v}\left( u^{2}+v^{2}\right) dx-\frac{t^{p}}{p}\int_{\mathbb{R}%
^{3}}F_{\beta }\left( u,v\right) dx, \\
h_{\lambda ,\left( u,v\right) }^{\prime }\left( t\right) &=&t\left\Vert
\left( u,v\right) \right\Vert _{H}^{2}+\lambda t^{3}\int_{\mathbb{R}%
^{3}}\phi _{u,v}\left( u^{2}+v^{2}\right) dx-t^{p-1}\int_{\mathbb{R}%
^{3}}F_{\beta }\left( u,v\right) dx, \\
h_{\lambda ,\left( u,v\right) }^{\prime \prime }\left( t\right)
&=&\left\Vert \left( u,v\right) \right\Vert _{H}^{2}+3\lambda t^{2}\int_{%
\mathbb{R}^{3}}\phi _{u,v}\left( u^{2}+v^{2}\right) dx-\left( p-1\right)
t^{p-2}\int_{\mathbb{R}^{3}}F_{\beta }\left( u,v\right) dx.
\end{eqnarray*}%
A direct calculation shows that
\begin{equation*}
th_{\lambda ,\left( u,v\right) }^{\prime }\left( t\right) =\left\Vert \left(
tu,tv\right) \right\Vert _{H}^{2}+\lambda \int_{\mathbb{R}^{3}}\phi
_{tu,tv}\left( t^{2}u^{2}+t^{2}v^{2}\right) dx-\int_{\mathbb{R}^{3}}F_{\beta
}\left( tu,tv\right) dx
\end{equation*}%
and so, for $\left( u,v\right) \in H\backslash \left\{ \left( 0,0\right)
\right\} $ and $t>0,$ $h_{\lambda ,\left( u,v\right) }^{\prime }\left(
t\right) =0$ holds if and only if $\left( tu,tv\right) \in \mathbf{M}%
_{\lambda ,\beta }$. In particular, $h_{\lambda ,\left( u,v\right) }^{\prime
}\left( 1\right) =0$ holds if and only if $\left( u,v\right) \in \mathbf{M}%
_{\lambda ,\beta }.$ It becomes natural to split $\mathbf{M}_{\lambda ,\beta
}$ into three parts corresponding to the local minima, local maxima and
points of inflection. Following \cite{T}, we define
\begin{eqnarray*}
\mathbf{M}_{\lambda ,\beta }^{+} &=&\{u\in \mathbf{M}_{\lambda ,\beta
}:h_{\lambda ,\left( u,v\right) }^{\prime \prime }\left( 1\right) >0\}, \\
\mathbf{M}_{\lambda ,\beta }^{0} &=&\{u\in \mathbf{M}_{\lambda ,\beta
}:h_{\lambda ,\left( u,v\right) }^{\prime \prime }\left( 1\right) =0\}, \\
\mathbf{M}_{\lambda ,\beta }^{-} &=&\{u\in \mathbf{M}_{\lambda ,\beta
}:h_{\lambda ,\left( u,v\right) }^{\prime \prime }\left( 1\right) <0\}.
\end{eqnarray*}

\begin{lemma}
\label{g2}Suppose that $\left( u_{0},v_{0}\right) $ is a local minimizer for
$J_{\lambda ,\beta }$ on $\mathbf{M}_{\lambda ,\beta }$ and $\left(
u_{0},v_{0}\right) \notin \mathbf{M}_{\lambda ,\beta }^{0}.$ Then $%
J_{\lambda ,\beta }^{\prime }\left( u_{0},v_{0}\right) =0$ in $H^{-1}.$
\end{lemma}

\begin{proof}
The proof is essentially same as that in Brown-Zhang \cite[Theorem 2.3]{BZ},
so we omit it here.
\end{proof}

For each $\left( u,v\right) \in \mathbf{M}_{\lambda ,\beta },$ we find that
\begin{eqnarray}
h_{\lambda ,\left( u,v\right) }^{\prime \prime }\left( 1\right)
&=&\left\Vert \left( u,v\right) \right\Vert _{H}^{2}+3\lambda \int_{\mathbb{R%
}^{3}}\phi _{u,v}\left( u^{2}+v^{2}\right) dx-\left( p-1\right) \int_{%
\mathbb{R}^{3}}F_{\beta }\left( u,v\right) dx.  \notag \\
&=&-\left( p-2\right) \left\Vert \left( u,v\right) \right\Vert
_{H}^{2}+\lambda \left( 4-p\right) \int_{\mathbb{R}^{3}}\phi _{u,v}\left(
u^{2}+v^{2}\right) dx  \label{2-6-1} \\
&=&-2\left\Vert \left( u,v\right) \right\Vert _{H}^{2}+\left( 4-p\right)
\int_{\mathbb{R}^{3}}F_{\beta }\left( u,v\right) dx.  \label{2-6-2}
\end{eqnarray}%
For each $\left( u,v\right) \in \mathbf{M}_{\lambda ,\beta }^{-}$, using $%
\left( \ref{2-2}\right) $ and $(\ref{2-6-2})$ gives
\begin{eqnarray*}
J_{\lambda ,\beta }(u,v) &=&\frac{1}{4}\left\Vert \left( u,v\right)
\right\Vert _{H}^{2}-\frac{4-p}{4p}\int_{\mathbb{R}^{3}}F_{\beta }\left(
u,v\right) dx>\frac{p-2}{4p}\left\Vert \left( u,v\right) \right\Vert _{H}^{2}
\\
&\geq &\frac{p-2}{4p}C_{\beta }^{-1/\left( p-2\right) }>0.
\end{eqnarray*}%
For each $\left( u,v\right) \in \mathbf{M}_{\lambda ,\beta }^{+},$ by (\ref%
{2-6-1}) one has%
\begin{eqnarray*}
J_{\lambda ,\beta }(u,v) &=&\frac{p-2}{2p}\left\Vert \left( u,v\right)
\right\Vert _{H}^{2}-\frac{\lambda (4-p)}{4p}\int_{\mathbb{R}^{3}}\phi
_{u,v}\left( u^{2}+v^{2}\right) dx \\
&<&\frac{p-2}{4p}\left\Vert \left( u,v\right) \right\Vert _{H}^{2}.
\end{eqnarray*}%
Hence, we have the following result.

\begin{lemma}
\label{g5}The energy functional $J_{\lambda ,\beta }$ is coercive and
bounded below on $\mathbf{M}_{\lambda ,\beta }^{-}.$ Furthermore, for all $%
u\in \mathbf{M}_{\lambda ,\beta }^{-}$, there holds
\begin{equation*}
J_{\lambda ,\beta }(u,v)>\frac{p-2}{4p}C_{\beta }^{-1/\left( p-2\right) }>0.
\end{equation*}
\end{lemma}

Let $\left( u,v\right) \in \mathbf{M}_{\lambda ,\beta }$ with $J_{\lambda
,\beta }\left( u,v\right) <\frac{\left( p-2\right) ^{2}\overline{S}%
^{2}S_{12/5}^{4}}{4\lambda p(4-p)},$ we deduce that%
\begin{eqnarray*}
\frac{\left( p-2\right) ^{2}\overline{S}^{2}S_{12/5}^{4}}{4\lambda p(4-p)}
&>&J_{\lambda ,\beta }(u,v)=\frac{p-2}{2p}\left\Vert \left( u,v\right)
\right\Vert _{H}^{2}-\frac{\lambda (4-p)}{4p}\int_{\mathbb{R}^{3}}\phi
_{u,v}\left( u^{2}+v^{2}\right) dx \\
&\geq &\frac{p-2}{2p}\left\Vert \left( u,v\right) \right\Vert _{H}^{2}-\frac{%
\lambda (4-p)}{4p\overline{S}^{2}S_{12/5}^{4}}\left\Vert \left( u,v\right)
\right\Vert _{H}^{4}.
\end{eqnarray*}%
Since the function
\begin{equation*}
f\left( x\right) :=\frac{p-2}{2p}x^{2}-\frac{\lambda (4-p)}{4p\overline{S}%
^{2}S_{12/5}^{4}}x^{4}
\end{equation*}%
have the maximum at $x_{0}=\left( \frac{\left( p-2\right) \overline{S}%
^{2}S_{12/5}^{4}}{\lambda (4-p)}\right) ^{1/2},$ we have%
\begin{equation*}
\max_{x\geq 0}f\left( x\right) =f\left( x_{0}\right) =\frac{\left(
p-2\right) ^{2}\overline{S}^{2}S_{12/5}^{4}}{4\lambda p(4-p)}.
\end{equation*}%
Thus,
\begin{equation*}
\mathbf{M}_{\lambda ,\beta }\left[ \frac{\left( p-2\right) ^{2}\overline{S}%
^{2}S_{12/5}^{4}}{4\lambda p(4-p)}\right] =\mathbf{M}_{\lambda ,\beta }^{(1)}%
\left[ \frac{\left( p-2\right) ^{2}\overline{S}^{2}S_{12/5}^{4}}{4\lambda
p(4-p)}\right] \cup \mathbf{M}_{\lambda ,\beta }^{(2)}\left[ \frac{\left(
p-2\right) ^{2}\overline{S}^{2}S_{12/5}^{4}}{4\lambda p(4-p)}\right] ,
\end{equation*}%
where%
\begin{equation*}
\mathbf{M}_{\lambda ,\beta }[D]:=\left\{ u\in \mathbf{M}_{\lambda ,\beta
}:J_{\lambda ,\beta }\left( u,v\right) <D\right\} ,
\end{equation*}%
\begin{equation*}
\mathbf{M}_{\lambda ,\beta }^{(1)}[D]:=\left\{ u\in \mathbf{M}_{\lambda
,\beta }[D]:\left\Vert \left( u,v\right) \right\Vert _{H}<\left( \frac{%
\left( p-2\right) \overline{S}^{2}S_{12/5}^{4}}{\lambda (4-p)}\right)
^{1/2}\right\}
\end{equation*}%
and
\begin{equation*}
\mathbf{M}_{\lambda ,\beta }^{(2)}[D]:=\left\{ u\in \mathbf{M}_{\lambda
,\beta }[D]:\left\Vert \left( u,v\right) \right\Vert _{H}>\left( \frac{%
\left( p-2\right) \overline{S}^{2}S_{12/5}^{4}}{\lambda (4-p)}\right)
^{1/2}\right\}
\end{equation*}%
for $D>0.$ For convenience, we always set%
\begin{equation*}
\mathbf{M}_{\lambda ,\beta }^{(1)}:=\mathbf{M}_{\lambda ,\beta }^{(1)}\left[
\frac{\left( p-2\right) ^{2}\overline{S}^{2}S_{12/5}^{4}}{4\lambda p(4-p)}%
\right] \text{ and }\mathbf{M}_{\lambda ,\beta }^{(2)}:=\mathbf{M}_{\lambda
,\beta }^{(2)}\left[ \frac{\left( p-2\right) ^{2}\overline{S}^{2}S_{12/5}^{4}%
}{4\lambda p(4-p)}\right] .
\end{equation*}%
By $\left( \ref{2-6-1}\right) ,$ the Sobolev inequality and Lemma \ref{L2-3}%
, it follows that
\begin{eqnarray*}
h_{\lambda ,\left( u,v\right) }^{\prime \prime }\left( 1\right) &=&-\left(
p-2\right) \left\Vert \left( u,v\right) \right\Vert _{H}^{2}+\lambda \left(
4-p\right) \int_{\mathbb{R}^{3}}\phi _{u,v}\left( u^{2}+v^{2}\right) dx \\
&\leq &\left\Vert \left( u,v\right) \right\Vert _{H}^{2}\left[ \lambda
\overline{S}^{-2}S_{12/5}^{-4}(4-p)\left\Vert \left( u,v\right) \right\Vert
_{H}^{2}-\left( p-2\right) \right] \\
&<&0\text{ for all }u\in \mathbf{M}_{\lambda ,\beta }^{(1)}.
\end{eqnarray*}%
Using $\left( \ref{2-6-2}\right) $ we derive that
\begin{eqnarray*}
\frac{1}{4}\left\Vert \left( u,v\right) \right\Vert _{H}^{2}-\frac{4-p}{4p}%
\int_{\mathbb{R}^{3}}F_{\beta }\left( u,v\right) dx &=&J_{\lambda ,\beta
}\left( u,v\right) <\frac{\left( p-2\right) ^{2}\overline{S}^{2}S_{12/5}^{4}%
}{4p(4-p)\lambda } \\
&<&\frac{p-2}{4p}\left\Vert \left( u,v\right) \right\Vert _{H}^{2}\text{ for
all }u\in \mathbf{M}_{\lambda ,\beta }^{(2)},
\end{eqnarray*}%
which implies that if $u\in \mathbf{M}_{\lambda ,\beta }^{(2)},$ then we have%
\begin{equation*}
h_{\lambda ,\left( u,v\right) }^{\prime \prime }\left( 1\right)
=-2\left\Vert \left( u,v\right) \right\Vert _{H}^{2}+\left( 4-p\right) \int_{%
\mathbb{R}^{3}}F_{\beta }\left( u,v\right) dx>0.
\end{equation*}%
Hence, we have the following result.

\begin{lemma}
\label{g7}If $\lambda >0$ and $\beta >0,$ then $\mathbf{M}_{\lambda ,\beta
}^{(1)}\subset \mathbf{M}_{\lambda ,\beta }^{-}$ and $\mathbf{M}_{\lambda
,\beta }^{(2)}\subset \mathbf{M}_{\lambda ,\beta }^{+}$ are $C^{1}$
sub-manifolds. Furthermore, each local minimizer of the functional $%
J_{\lambda ,\beta }$ in the sub-manifolds $\mathbf{M}_{\lambda ,\beta
}^{(1)} $ and $\mathbf{M}_{\lambda ,\beta }^{(2)}$ is a critical point of $%
J_{\lambda ,\beta }$ in $H.$
\end{lemma}

Let $w_{\beta }$ be the unique positive radial solution of the following Schr%
\"{o}dinger equation%
\begin{equation}
\begin{array}{ll}
-\Delta u+u=g_{\beta }\left( s_{\beta }\right) \left\vert u\right\vert
^{p-2}u & \text{ in }\mathbb{R}^{3},%
\end{array}
\tag*{$\left( E_{\beta }^{\infty }\right) $}
\end{equation}%
where $g_{\beta }\left( s_{\beta }\right) =\max_{s\in \left[ 0,1\right]
}g_{\beta }\left( s\right) >1$ as in Lemma \ref{L2-1}. Note that $s_{\beta }=%
\frac{1}{2}$ and $g_{\beta }\left( \frac{1}{2}\right) =\left( \frac{1}{2}%
\right) ^{\frac{p-2}{2}}+\left( \frac{1}{2}\right) ^{\frac{p-2}{2}}\beta $
for all $\beta \geq \frac{p-2}{2}.$ From \cite{K}, we see that
\begin{equation*}
w_{\beta }\left( 0\right) =\max_{x\in \mathbb{R}^{3}}w_{\beta }(x),\text{ }%
\left\Vert w_{\beta }\right\Vert _{H^{1}}^{2}=\int_{\mathbb{R}^{3}}g_{\beta
}\left( s_{\beta }\right) \left\vert w_{\beta }\right\vert ^{p}dx=\left(
\frac{S_{p}^{p}}{g_{\beta }\left( s_{\beta }\right) }\right) ^{2/\left(
p-2\right) }
\end{equation*}%
and%
\begin{equation}
\alpha _{\beta }^{\infty }:=\inf_{u\in \mathbf{M}_{\beta }^{\infty
}}J_{\beta }^{\infty }(u)=J_{\beta }^{\infty }(w_{\beta })=\frac{p-2}{2p}%
\left( \frac{S_{p}^{p}}{g_{\beta }\left( s_{\beta }\right) }\right)
^{2/\left( p-2\right) },  \label{4-1}
\end{equation}%
where $J_{\beta }^{\infty }$ is the energy functional of Eq. $\left(
E_{\beta }^{\infty }\right) $ in $H^{1}(\mathbb{R}^{3})$ in the form%
\begin{equation*}
J_{\beta }^{\infty }(u)=\frac{1}{2}\int_{\mathbb{R}^{3}}\left( |\nabla
u|^{2}+u^{2}\right) dx-\frac{g_{\beta }\left( s_{\beta }\right) }{p}\int_{%
\mathbb{R}^{3}}\left\vert u\right\vert ^{p}dx
\end{equation*}%
with
\begin{equation*}
\mathbf{M}_{\beta }^{\infty }:=\{u\in H^{1}(\mathbb{R}^{3})\backslash
\{0\}:\left\langle (J_{\beta }^{\infty })^{\prime }\left( u\right)
,u\right\rangle =0\}.
\end{equation*}%
Define
\begin{equation*}
k\left( \lambda \right) :=\left\{
\begin{array}{ll}
\rho _{p}, & \text{ if }0<\lambda <\rho _{p}, \\
\lambda , & \text{ if }\lambda \geq \rho _{p},%
\end{array}%
\right.
\end{equation*}%
where $\rho _{p}:=\frac{\left( p-2\right) \overline{S}^{2}S_{12/5}^{4}}{%
2(4-p)S_{p}^{2p/\left( p-2\right) }}.$ Then $k\left( \lambda \right) \geq
\lambda $ and $k^{-1}\left( \lambda \right) \leq \rho _{p}^{-1}$ for all $%
\lambda >0,$ which implies that
\begin{equation*}
\mathbf{M}_{\lambda ,\beta }\left[ \frac{\left( p-2\right) ^{2}\overline{S}%
^{2}S_{12/5}^{4}}{4p(4-p)k\left( \lambda \right) }\right] \subset \mathbf{M}%
_{\lambda ,\beta }\left[ \frac{p-2}{2p}S_{p}^{2p/\left( p-2\right) }\right]
\end{equation*}%
and%
\begin{equation}
\overline{\mathbf{M}}_{\lambda ,\beta }^{\left( i\right) }:=\mathbf{M}%
_{\lambda ,\beta }^{\left( i\right) }\left[ \frac{\left( p-2\right) ^{2}%
\overline{S}^{2}S_{12/5}^{4}}{4p(4-p)k\left( \lambda \right) }\right]
\subset \mathbf{M}_{\lambda ,\beta }^{\left( i\right) }\left[ \frac{p-2}{2p}%
S_{p}^{2p/\left( p-2\right) }\right]  \label{4-2}
\end{equation}%
for all $\lambda >0$ and $i=1,2.$ Furthermore, we have the following results.

\begin{lemma}
\label{l5}Let $2<p<4$ and $\lambda >0.$ Let $\left( u_{0},v_{0}\right) $ be
a critical point of $J_{\lambda ,\beta }$ on $\mathbf{M}_{\lambda ,\beta
}^{-}.$ Then we have $J_{\lambda ,\beta }\left( u_{0},v_{0}\right) >\frac{p-2%
}{2p}S_{p}^{2p/\left( p-2\right) }$ if either $u_{0}=0$ or $v_{0}=0.$
\end{lemma}

\begin{proof}
Without loss of generality, we may assume that $v_{0}=0.$ Then we have%
\begin{equation*}
J_{\lambda ,\beta }\left( u_{0},0\right) =\frac{1}{2}\left\Vert
u_{0}\right\Vert _{H^{1}}^{2}+\frac{\lambda }{4}\int_{\mathbb{R}^{3}}\phi
_{u_{0}}u_{0}^{2}dx-\frac{1}{p}\int_{\mathbb{R}^{3}}\left\vert
u_{0}\right\vert ^{p}dx
\end{equation*}%
and%
\begin{equation*}
-2\left\Vert u_{0}\right\Vert _{H^{1}}^{2}+\left( 4-p\right) \int_{\mathbb{R}%
^{3}}\left\vert u\right\vert ^{p}dx<0.
\end{equation*}%
Note that
\begin{equation*}
\left\Vert t_{0}\left( u_{0}\right) u_{0}\right\Vert _{H^{1}}^{2}-\int_{%
\mathbb{R}^{3}}\left\vert t_{0}\left( u_{0}\right) u_{0}\right\vert ^{p}dx=0,
\end{equation*}%
where
\begin{equation}
\left( \frac{4-p}{2}\right) ^{1/\left( p-2\right) }<t_{0}\left( u_{0}\right)
:=\left( \frac{\left\Vert u_{0}\right\Vert _{H^{1}}^{2}}{\int_{\mathbb{R}%
^{3}}\left\vert u_{0}\right\vert ^{p}dx}\right) ^{1/\left( p-2\right) }<1.
\label{5-1}
\end{equation}%
By a similar argument in Sun-Wu-Feng \cite[Lemma 2.6]{SWF1}, one has
\begin{equation*}
J_{\lambda ,\beta }\left( u_{0},0\right) =\sup_{0\leq t\leq t_{\lambda
}^{+}}J_{\lambda ,\beta }(tu_{0},0),
\end{equation*}%
where $t_{\lambda }^{+}>\left( \frac{2}{4-p}\right) ^{1/\left( p-2\right)
}t_{0}\left( u_{0}\right) >1$ by (\ref{5-1}). Using this, together with (\ref%
{5-1}) again, one has
\begin{equation*}
J_{\lambda ,\beta }\left( u_{0},0\right) >J_{\lambda ,\beta }(t_{0}\left(
u_{0}\right) u_{0},0).
\end{equation*}%
Thus, by \cite{Wi}, we have
\begin{eqnarray*}
J_{\lambda ,\beta }\left( u_{0},0\right) &>&J_{\lambda ,\beta }(t_{0}\left(
u_{0}\right) u_{0},0) \\
&\geq &\frac{1}{2}\left\Vert t_{0}\left( u_{0}\right) u_{0}\right\Vert
_{H^{1}}^{2}-\frac{1}{p}\int_{\mathbb{R}^{3}}\left\vert t_{0}\left(
u_{0}\right) u_{0}\right\vert ^{p}dx+\frac{\lambda \left[ t_{0}\left(
u_{0}\right) \right] ^{4}}{4}\int_{\mathbb{R}^{3}}\phi _{u_{0}}u_{0}^{2}dx \\
&>&\frac{p-2}{2p}S_{p}^{2p/\left( p-2\right) }.
\end{eqnarray*}%
The proof is complete.
\end{proof}

\begin{lemma}
\label{g4}Let $2<p<4$ and $\lambda >0.$ Let $w_{\beta }\left( x\right) $ be
a unique positive radial solution of Eq. $\left( E_{\beta }^{\infty }\right)
$. Then for each
\begin{equation*}
\beta >\beta _{0}\left( \lambda \right) :=\max \left\{ \frac{p-2}{2},\left[
\frac{\lambda pS_{p}^{2p/\left( p-2\right) }}{\left( p-2\right) \overline{S}%
^{2}S_{12/5}^{4}}\right] ^{\left( p-2\right) /2}\left( \frac{p}{4-p}\right)
^{\left( 4-p\right) /2}-1\right\} ,
\end{equation*}%
there exists two constants $t_{\lambda ,\beta }^{+}$ and $t_{\lambda ,\beta
}^{-}$ satisfying
\begin{equation*}
1<t_{\lambda ,\beta }^{-}<\left( \frac{2}{4-p}\right) ^{\frac{1}{p-2}%
}<t_{\lambda ,\beta }^{+}
\end{equation*}%
such that%
\begin{equation*}
\left( t_{\lambda ,\beta }^{\pm }\sqrt{s_{\beta }}w_{\beta },t_{\lambda
,\beta }^{\pm }\sqrt{1-s_{\beta }}w_{\beta }\right) \in \mathbf{M}_{\lambda
,\beta }^{\pm }\cap H_{r}
\end{equation*}%
and%
\begin{equation*}
J_{\lambda ,\beta }\left( t_{\lambda ,\beta }^{+}\sqrt{s_{\beta }}w_{\beta
},t_{\lambda ,\beta }^{+}\sqrt{1-s_{\beta }}w_{\beta }\right) =\inf_{t\geq
0}J_{\lambda ,\beta }\left( t\sqrt{s_{\beta }}w_{\beta },t\sqrt{1-s_{\beta }}%
w_{\beta }\right) <0.
\end{equation*}%
In particular, $\left( t_{\lambda ,\beta }^{+}\sqrt{s_{\beta }}w_{\beta
},t_{\lambda ,\beta }^{+}\sqrt{1-s_{\beta }}w_{\beta }\right) \in \overline{%
\mathbf{M}}_{\lambda ,\beta }^{\left( 2\right) }\cap H_{r}.$
\end{lemma}

\begin{proof}
Define
\begin{eqnarray*}
\eta \left( t\right)  &=&t^{-2}\left\Vert \left( \sqrt{s_{\beta }}w_{\beta },%
\sqrt{1-s_{\beta }}w_{\beta }\right) \right\Vert _{H}^{2}-t^{p-4}\int_{%
\mathbb{R}^{3}}F_{\beta }\left( \sqrt{s_{\beta }}w_{\beta },\sqrt{1-s_{\beta
}}w_{\beta }\right) dx \\
&=&t^{-2}\left\Vert w_{\beta }\right\Vert _{H^{1}}^{2}-t^{p-4}\int_{\mathbb{R%
}^{3}}g_{\beta }\left( s_{\beta }\right) \left\vert w_{\beta }\right\vert
^{p}dx\text{ for }t>0\text{ and }\beta \geq \frac{p-2}{2}.
\end{eqnarray*}%
Clearly, $tu\in \mathbf{M}_{\lambda ,\beta }$ if and only if
\begin{eqnarray*}
\eta \left( t\right)  &=&-\lambda \int_{\mathbb{R}^{3}}\phi _{\sqrt{s_{\beta
}}w_{\beta },\sqrt{1-s_{\beta }}w_{\beta }}\left( \left( \sqrt{s_{\beta }}%
w_{\beta }\right) ^{2}+\left( \sqrt{1-s_{\beta }}w_{\beta }\right)
^{2}\right) dx \\
&=&-\lambda \int_{\mathbb{R}^{3}}\phi _{w_{\beta }}w_{\beta }^{2}dx.
\end{eqnarray*}%
A straightforward evaluation shows that
\begin{equation*}
\eta \left( 1\right) =0,\ \lim_{t\rightarrow 0^{+}}\eta (t)=\infty \text{
and }\lim_{t\rightarrow \infty }\eta (t)=0.
\end{equation*}%
Since $2<p<4$ and
\begin{equation*}
\eta ^{\prime }\left( t\right) =t^{-3}\left\Vert w_{\beta }\right\Vert
_{H^{1}}^{2}\left[ -2+\left( 4-p\right) t^{p-2}\right] ,
\end{equation*}%
we find that $\eta \left( t\right) $ is decreasing when $0<t<\left( \frac{2}{%
4-p}\right) ^{1/\left( p-2\right) }$ and is increasing when $t>\left( \frac{2%
}{4-p}\right) ^{1/\left( p-2\right) }.$ This implies that
\begin{equation*}
\inf_{t>0}\eta \left( t\right) =\eta \left( \left( \frac{2}{4-p}\right)
^{1/\left( p-2\right) }\right) .
\end{equation*}%
Moreover, for each $\lambda >0$ and $\beta >\beta _{0}\left( \lambda \right)
,$ we further have%
\begin{eqnarray*}
\eta \left( \left( \frac{2}{4-p}\right) ^{1/(p-2)}\right)  &=&\left( \frac{%
4-p}{2}\right) ^{2/(p-2)}\left\Vert w_{\beta }\right\Vert
_{H^{1}}^{2}-\left( \frac{2}{4-p}\right) ^{\left( p-4\right) /(p-2)}\int_{%
\mathbb{R}^{3}}g_{\beta }\left( s_{\beta }\right) \left\vert w_{\beta
}\right\vert ^{p}dx \\
&=&-\left( \frac{p-2}{2}\right) \left( \frac{4-p}{2}\right) ^{\left(
4-p\right) /(p-2)}\left\Vert w_{\beta }\right\Vert _{H^{1}}^{2} \\
&<&-\lambda \overline{S}^{-2}S_{12/5}^{-4}\left\Vert w_{\beta }\right\Vert
_{H^{1}}^{4} \\
&\leq &-\lambda \int_{\mathbb{R}^{3}}\phi _{w_{\beta }}w_{\beta }^{2}dx \\
&=&-\lambda \int_{\mathbb{R}^{3}}\phi _{\sqrt{s_{\beta }}w_{\beta },\sqrt{%
1-s_{\beta }}w_{\beta }}\left( \left( \sqrt{s_{\beta }}w_{\beta }\right)
^{2}+\left( \sqrt{s_{\beta }}w_{\beta }\right) ^{2}\right) dx.
\end{eqnarray*}%
Thus, there exist two positive constants $t_{\lambda ,\beta }^{+}$ and $%
t_{\lambda ,\beta }^{-}$ satisfying
\begin{equation*}
1<t_{\lambda ,\beta }^{-}<\left( \frac{2}{4-p}\right) ^{1/\left( p-2\right)
}<t_{\lambda ,\beta }^{+}
\end{equation*}%
such that
\begin{equation*}
\eta \left( t_{\lambda ,\beta }^{\pm }\right) +\lambda \int_{\mathbb{R}%
^{3}}\phi _{\sqrt{s_{\beta }}w_{\beta },\sqrt{1-s_{\beta }}w_{\beta }}\left(
\left( \sqrt{s_{\beta }}w_{\beta }\right) ^{2}+\left( \sqrt{s_{\beta }}%
w_{\beta }\right) ^{2}\right) dx=0.
\end{equation*}%
That is,
\begin{equation*}
\left( t_{\lambda ,\beta }^{\pm }\sqrt{s_{\beta }}w_{\beta },t_{\lambda
,\beta }^{\pm }\sqrt{1-s_{\beta }}w_{\beta }\right) \in \mathbf{M}_{\lambda
,\beta }\cap H_{r}.
\end{equation*}%
By a calculation on the second order derivatives, we find
\begin{eqnarray*}
h_{\lambda ,\left( t_{\lambda ,\beta }^{-}\sqrt{s_{\beta }}w_{\beta
},t_{\lambda ,\beta }^{-}\sqrt{1-s_{\beta }}w_{\beta }\right) }^{\prime
\prime }\left( 1\right)  &=&-2\left\Vert t_{\lambda ,\beta }^{-}w_{\beta
}\right\Vert _{H^{1}}^{2}+\left( 4-p\right) \int_{\mathbb{R}^{3}}g_{\beta
}\left( s_{\beta }\right) \left\vert t_{\lambda ,\beta }^{-}w_{\beta
}\right\vert ^{p}dx \\
&=&\left( t_{\lambda ,\beta }^{-}\right) ^{5}\eta ^{\prime }\left(
t_{\lambda ,\beta }^{-}\right) <0
\end{eqnarray*}%
and
\begin{eqnarray*}
h_{\lambda ,\left( t_{\lambda ,\beta }^{+}\sqrt{s_{\beta }}w_{\beta
},t_{\lambda ,\beta }^{+}\sqrt{1-s_{\beta }}w_{\beta }\right) }^{\prime
\prime }\left( 1\right)  &=&-2\left\Vert t_{\lambda ,\beta }^{+}w_{\beta
}\right\Vert _{H^{1}}^{2}+\left( 4-p\right) \int_{\mathbb{R}^{3}}g_{\beta
}\left( s_{\beta }\right) \left\vert t_{\lambda ,\beta }^{+}w_{\beta
}\right\vert ^{p}dx \\
&=&\left( t_{\lambda ,\beta }^{+}\right) ^{5}\eta ^{\prime }\left(
t_{\lambda ,\beta }^{+}\right) >0,
\end{eqnarray*}%
leading to
\begin{equation*}
\left( t_{\lambda ,\beta }^{\pm }\sqrt{s_{\beta }}w_{\beta },t_{\lambda
,\beta }^{\pm }\sqrt{1-s_{\beta }}w_{\beta }\right) \in \mathbf{M}_{\lambda
,\beta }^{\pm }\cap H_{r}
\end{equation*}%
and
\begin{eqnarray*}
&&h_{\lambda ,\left( t_{\lambda ,\beta }^{+}\sqrt{s_{\beta }}w_{\beta
},t_{\lambda ,\beta }^{+}\sqrt{1-s_{\beta }}w_{\beta }\right) }^{\prime
}\left( t\right)  \\
&=&t^{3}\left( \eta (t)+\lambda \int_{\mathbb{R}^{3}}\phi _{\sqrt{s_{\beta }}%
w_{\beta },\sqrt{1-s_{\beta }}w_{\beta }}\left( \left( \sqrt{s_{\beta }}%
w_{\beta }\right) ^{2}+\left( \sqrt{s_{\beta }}w_{\beta }\right) ^{2}\right)
dx\right) .
\end{eqnarray*}%
One can see that
\begin{equation*}
h_{\lambda ,\left( \sqrt{s_{\beta }}w_{\beta },\sqrt{1-s_{\beta }}w_{\beta
}\right) }^{\prime }\left( t\right) >0\text{ for all }t\in \left(
0,t_{\lambda ,\beta }^{-}\right) \cup \left( t_{\lambda ,\beta }^{+},\infty
\right)
\end{equation*}%
and
\begin{equation*}
h_{\lambda ,\left( \sqrt{s_{\beta }}w_{\beta },\sqrt{1-s_{\beta }}w_{\beta
}\right) }^{\prime }\left( t\right) <0\text{ for all }t\in (t_{\lambda
,\beta }^{-},t_{\lambda ,\beta }^{+}),
\end{equation*}%
implying that
\begin{equation*}
J_{\lambda ,\beta }\left( t_{\lambda ,\beta }^{-}\sqrt{s_{\beta }}w_{\beta
},t_{\lambda ,\beta }^{-}\sqrt{1-s_{\beta }}w_{\beta }\right) =\sup_{0\leq
t\leq t_{\lambda ,\beta }^{+}}J_{\lambda ,\beta }\left( t\sqrt{s_{\beta }}%
w_{\beta },t\sqrt{1-s_{\beta }}w_{\beta }\right)
\end{equation*}%
and%
\begin{equation*}
J_{\lambda ,\beta }\left( t_{\lambda ,\beta }^{+}\sqrt{s_{\beta }}w_{\beta
},t_{\lambda ,\beta }^{+}\sqrt{1-s_{\beta }}w_{\beta }\right) =\inf_{t\geq
t_{\lambda ,\beta }^{-}}J_{\lambda ,\beta }\left( t\sqrt{s_{\beta }}w_{\beta
},t\sqrt{1-s_{\beta }}w_{\beta }\right) ,
\end{equation*}%
and so
\begin{equation*}
J_{\lambda ,\beta }\left( t_{\lambda ,\beta }^{+}\sqrt{s_{\beta }}w_{\beta
},t_{\lambda ,\beta }^{+}\sqrt{1-s_{\beta }}w_{\beta }\right) <J_{\lambda
,\beta }\left( t_{\lambda ,\beta }^{-}\sqrt{s_{\beta }}w_{\beta },t_{\lambda
,\beta }^{-}\sqrt{1-s_{\beta }}w_{\beta }\right) .
\end{equation*}%
Note that
\begin{eqnarray*}
J_{\lambda ,\beta }\left( t\sqrt{s_{\beta }}w_{\beta },t\sqrt{1-s_{\beta }}%
w_{\beta }\right)  &=&\frac{t^{2}}{2}\left\Vert w_{\beta }\right\Vert
_{H^{1}}^{2}+\frac{\lambda t^{4}}{4}\int_{\mathbb{R}^{3}}\phi _{w_{\beta
}}w_{\beta }^{2}dx-\frac{t^{p}}{p}\int_{\mathbb{R}^{3}}g_{\beta }\left(
s_{\beta }\right) \left\vert w_{\beta }\right\vert ^{p}dx \\
&=&t^{4}\left[ \xi \left( t\right) +\frac{\lambda }{4}\int_{\mathbb{R}%
^{3}}\phi _{w_{\beta }}w_{\beta }^{2}dx\right] ,
\end{eqnarray*}%
where
\begin{equation*}
\xi \left( t\right) :=\frac{t^{-2}}{2}\left\Vert w_{\beta }\right\Vert
_{H^{1}}^{2}-\frac{t^{p-4}}{p}\int_{\mathbb{R}^{3}}g_{\beta }\left( s_{\beta
}\right) \left\vert w_{\beta }\right\vert ^{p}dx.
\end{equation*}%
Clearly, $J_{\lambda ,\beta }\left( t\sqrt{s_{\beta }}w_{\beta },t\sqrt{%
1-s_{\beta }}w_{\beta }\right) =0$ if and only if
\begin{equation*}
\xi \left( t\right) +\frac{\lambda }{4}\int_{\mathbb{R}^{3}}\phi _{w_{\beta
}}w_{\beta }^{2}dx=0.
\end{equation*}%
It is not difficult to verify that
\begin{equation*}
\xi \left( \hat{t}_{a}\right) =0,\ \lim_{t\rightarrow 0^{+}}\xi (t)=\infty \
\text{and}\ \lim_{t\rightarrow \infty }\xi (t)=0,
\end{equation*}%
where $\hat{t}_{0}=\left( \frac{p}{2}\right) ^{1/\left( p-2\right) }.$ By
calculating the derivative of $\xi (t)$, we find that%
\begin{eqnarray*}
\xi ^{\prime }\left( t\right)  &=&-t^{-3}\left\Vert w_{\beta }\right\Vert
_{H^{1}}^{2}+\frac{\left( 4-p\right) }{p}t^{p-5}\int_{\mathbb{R}%
^{3}}g_{\beta }\left( s_{\beta }\right) \left\vert w_{\beta }\right\vert
^{p}dx \\
&=&t^{-3}\left\Vert w_{\beta }\right\Vert _{H^{1}}^{2}\left[ \frac{\left(
4-p\right) t^{p-2}}{p}-1\right] ,
\end{eqnarray*}%
which implies that $\xi \left( t\right) $ is decreasing when $0<t<\left(
\frac{p}{4-p}\right) ^{1/\left( p-2\right) }$ and is increasing when $%
t>\left( \frac{p}{4-p}\right) ^{1/\left( p-2\right) }.$ Then for each $%
\lambda >0$ and $\beta >\beta _{0}\left( \lambda \right) ,$ we have%
\begin{eqnarray*}
\inf_{t>0}\xi \left( t\right)  &=&\xi \left[ \left( \frac{p}{4-p}\right)
^{1/\left( p-2\right) }\right] =-\frac{p-2}{2p}\left( \frac{4-p}{p}\right)
^{\left( 4-p\right) /\left( p-2\right) }\left\Vert w_{\beta }\right\Vert
_{H^{1}}^{2} \\
&<&-\frac{\lambda }{4}\overline{S}^{-2}S_{12/5}^{-4}\left\Vert w_{\beta
}\right\Vert _{H^{1}}^{4}<-\frac{\lambda }{4}\int_{\mathbb{R}^{3}}\phi
_{w_{\beta }}w_{\beta }^{2}dx \\
&=&-\frac{\lambda }{4}\int_{\mathbb{R}^{3}}\phi _{\sqrt{s_{\beta }}w_{\beta
},\sqrt{1-s_{\beta }}w_{\beta }}\left( \left( \sqrt{s_{\beta }}w_{\beta
}\right) ^{2}+\left( \sqrt{s_{\beta }}w_{\beta }\right) ^{2}\right) dx,
\end{eqnarray*}%
which yields that
\begin{equation*}
J_{\lambda ,\beta }\left( t_{\lambda ,\beta }^{+}\sqrt{s_{\beta }}w_{\beta
},t_{\lambda ,\beta }^{+}\sqrt{1-s_{\beta }}w_{\beta }\right) =\inf_{t\geq
0}J_{\lambda ,\beta }\left( t\sqrt{s_{\beta }}w_{\beta },t\sqrt{1-s_{\beta }}%
w_{\beta }\right) <0.
\end{equation*}%
This implies that $\left( t_{\lambda ,\beta }^{+}\sqrt{s_{\beta }}w_{\beta
},t_{\lambda ,\beta }^{+}\sqrt{1-s_{\beta }}w_{\beta }\right) \in \overline{%
\mathbf{M}}_{\lambda ,\beta }^{\left( 2\right) }\cap H_{r}.$ The proof is
complete.
\end{proof}

Note that $\beta \left( \lambda \right) >\beta _{0}\left( \lambda \right) ,$
where we have used the inequality
\begin{equation*}
\frac{\left( 4-p\right) ^{2}}{4}\left( 1+\sqrt{1+\frac{p}{4-p}\left( \frac{2%
}{4-p}\right) ^{\frac{4}{p-2}}}\right) ^{p-2}>1\text{ for }2<p<4.
\end{equation*}%
Then we have the following result.

\begin{lemma}
\label{g6}Let $2<p<4$ and $\lambda >0.$ Let $w_{\beta }\left( x\right) $ be
a unique positive radial solution of Eq. $(E_{\beta }^{\infty })$. Then for
each $\beta >\beta \left( \lambda \right) ,$ there exists two positive
constants $t_{\lambda ,\beta }^{+}$ and $t_{\lambda ,\beta }^{-}$ satisfying
\begin{equation*}
1<t_{\lambda ,\beta }^{-}<\left( \frac{2}{4-p}\right) ^{\frac{1}{p-2}%
}<t_{\lambda ,\beta }^{+}
\end{equation*}%
such that
\begin{equation*}
\left( t_{\lambda ,\beta }^{-}\sqrt{s_{\beta }}w_{\beta },t_{\lambda ,\beta
}^{-}\sqrt{1-s_{\beta }}w_{\beta }\right) \in \overline{\mathbf{M}}_{\lambda
,\beta }^{\left( 1\right) }\cap H_{r}\text{ and }\left( t_{\lambda ,\beta
}^{+}\sqrt{s_{\beta }}w_{\beta },t_{\lambda ,\beta }^{+}\sqrt{1-s_{\beta }}%
w_{\beta }\right) \in \overline{\mathbf{M}}_{\lambda ,\beta }^{\left(
2\right) }\cap H_{r},
\end{equation*}
\end{lemma}

\begin{proof}
By Lemma \ref{g4}, for $\lambda >0$ and $\beta >\beta \left( \lambda \right)
,$ we have%
\begin{equation*}
\left( t_{\lambda ,\beta }^{+}\sqrt{s_{\beta }}w_{\beta },t_{\lambda ,\beta
}^{+}\sqrt{1-s_{\beta }}w_{\beta }\right) \in \overline{\mathbf{M}}_{\lambda
,\beta }^{\left( 2\right) }\cap H_{r}.
\end{equation*}%
Next, we show that for $\lambda >0$ and $\beta >\beta \left( \lambda \right)
,$
\begin{equation*}
\left( t_{\lambda ,\beta }^{-}\sqrt{s_{\beta }}w_{\beta },t_{\lambda ,\beta
}^{-}\sqrt{1-s_{\beta }}w_{\beta }\right) \in \overline{\mathbf{M}}_{\lambda
,\beta }^{\left( 1\right) }\cap H_{r}.
\end{equation*}%
It follows from Lemma \ref{2-1} and (\ref{4-1}) that%
\begin{eqnarray*}
&&J_{\lambda ,\beta }\left( t_{\lambda ,\beta }^{-}\sqrt{s_{\beta }}w_{\beta
},t_{\lambda ,\beta }^{-}\sqrt{1-s_{\beta }}w_{\beta }\right)  \\
&=&\frac{\left( t_{\lambda ,\beta }^{-}\right) ^{2}}{2}\left\Vert w_{\beta
}\right\Vert _{H^{1}}^{2}+\frac{\lambda \left( t_{\lambda ,\beta
}^{-}\right) ^{4}}{4}\int_{\mathbb{R}^{3}}\phi _{w_{\beta }}w_{\beta }^{2}dx-%
\frac{\left( t_{\lambda ,\beta }^{-}\right) ^{p}}{p}\int_{\mathbb{R}%
^{3}}g_{\beta }\left( s_{\beta }\right) w_{\beta }^{p}dx \\
&<&\alpha _{\beta }^{\infty }+\frac{\lambda }{4}\left( \frac{2}{4-p}\right)
^{\frac{4}{p-2}}\overline{S}^{-2}S_{12/5}^{-4}\left\Vert w_{\beta
}\right\Vert _{H^{1}}^{4} \\
&=&\frac{p-2}{p}\left( \frac{S_{p}^{p}}{1+\beta }\right) ^{2/\left(
p-2\right) }+\frac{\lambda }{\overline{S}^{2}S_{12/5}^{4}}\left( \frac{2}{4-p%
}\right) ^{\frac{4}{p-2}}\left( \frac{S_{p}^{p}}{1+\beta }\right) ^{4/\left(
p-2\right) } \\
&<&\frac{\left( p-2\right) ^{2}\overline{S}^{2}S_{12/5}^{4}}{4p(4-p)k\left(
\lambda \right) },
\end{eqnarray*}%
which implies that $\left( t_{\lambda ,\beta }^{-}\sqrt{s_{\beta }}w_{\beta
},t_{\lambda ,\beta }^{-}\sqrt{1-s_{\beta }}w_{\beta }\right) \in \overline{%
\mathbf{M}}_{\lambda ,\beta }^{\left( 1\right) }\cap H_{r}.$ This completes
the proof.
\end{proof}

Define
\begin{eqnarray*}
\alpha _{\lambda ,\beta }^{-} &:&=\inf_{\left( u,v\right) \in \overline{%
\mathbf{M}}_{\lambda ,\beta }^{\left( 1\right) }}J_{\lambda ,\beta }\left(
u,v\right) \text{ for }2<p<4, \\
\alpha _{\lambda ,\beta }^{+} &:&=\inf_{\left( u,v\right) \in \overline{%
\mathbf{M}}_{\lambda ,\beta }^{\left( 2\right) }}J_{\lambda ,\beta }\left(
u,v\right) \text{ for }2<p<4
\end{eqnarray*}%
and%
\begin{equation*}
\alpha _{\lambda ,\beta }^{+,r}:=\inf_{\left( u,v\right) \in \overline{%
\mathbf{M}}_{\lambda ,\beta }^{(2)}\cap H_{r}}J_{\lambda ,\beta }\left(
u,v\right) \text{ for }2<p<3.
\end{equation*}%
Clearly, $\alpha _{\lambda ,\beta }^{-}=\inf_{u\in \overline{\mathbf{M}}%
_{\lambda ,\beta }^{-}}J_{\lambda ,\beta }\left( u,v\right) ,$ $\alpha
_{\lambda ,\beta }^{+}=\inf_{\left( u,v\right) \in \overline{\mathbf{M}}%
_{\lambda ,\beta }^{+}}J_{\lambda ,\beta }\left( u,v\right) $ and $\alpha
_{\lambda ,\beta }^{+,r}=\inf_{\left( u,v\right) \in \overline{\mathbf{M}}%
_{\lambda ,\beta }^{+}\cap H_{r}}J_{\lambda ,\beta }\left( u,v\right) .$ It
follows from Lemmas \ref{m5}, \ref{g5} and \ref{g6} that
\begin{equation*}
\frac{p-2}{4p}C_{\beta }^{-1/\left( p-2\right) }<\alpha _{\lambda ,\beta
}^{-}<\frac{\left( p-2\right) ^{2}\overline{S}^{2}S_{12/5}^{4}}{%
4p(4-p)k\left( \lambda \right) }\text{ for }2<p<4,  \label{25}
\end{equation*}%
and%
\begin{equation}
-\infty <\alpha _{\lambda ,\beta }^{+,r}<0\text{ for }2<p<3.  \label{26}
\end{equation}%
Furthermore, we have the following results.

\begin{theorem}
\label{t4}Let $2<p<4$ and $\lambda >0.$ Then for each $\beta >\beta
_{0}\left( \lambda \right) ,$ we have%
\begin{equation*}
\alpha _{\lambda ,\beta }^{+}=\inf_{\left( u,v\right) \in \mathbf{M}%
_{\lambda ,\beta }^{+}}J_{\lambda ,\beta }\left( u,v\right) =-\infty .
\label{6-1}
\end{equation*}
\end{theorem}

\begin{proof}
Since $w_{\beta }$ is the unique positive radial solution of Eq. $\left(
E_{\beta }^{\infty }\right) $ with $w_{\beta }\left( 0\right) =\max_{x\in
\mathbb{R}^{3}}w_{0}\left( x\right) ,$ we have
\begin{equation}
\frac{\lambda }{4}\overline{S}^{-2}S_{12/5}^{-4}\left\Vert w_{\beta
}\right\Vert _{H^{1}}^{2}<\frac{p-2}{2p}\left( \frac{4-p}{p}\right) ^{\left(
4-p\right) /\left( p-2\right) }  \label{6-3}
\end{equation}%
and
\begin{equation}
\left\Vert w_{\beta }\right\Vert _{H^{1}}^{2}=\int_{\mathbb{R}^{3}}g_{\beta
}\left( s_{\beta }\right) \left\vert w_{\beta }\right\vert ^{p}dx=\left(
\frac{S_{p}^{p}}{g_{\beta }\left( s_{\beta }\right) }\right) ^{2/\left(
p-2\right) }.  \label{6-4}
\end{equation}%
Then by Lemma \ref{g4}, there exists a positive constant $t_{\lambda ,\beta
}^{+}$ satisfying
\begin{equation*}
1<\left( \frac{2}{4-p}\right) ^{\frac{1}{p-2}}<t_{\lambda ,\beta }^{+}
\end{equation*}%
such that
\begin{equation*}
J_{\lambda ,\beta }\left( t_{\lambda ,\beta }^{+}\sqrt{s_{\beta }}w_{\beta
},t_{\lambda ,\beta }^{+}\sqrt{1-s_{\beta }}w_{\beta }\right) =\inf_{t\geq
0}J_{\lambda ,\beta }\left( t\sqrt{s_{\beta }}w_{\beta },t\sqrt{1-s_{\beta }}%
w_{\beta }\right) <0.
\end{equation*}%
For $R>1,$ we define a function $\psi _{R}\in C^{1}(\mathbb{R}^{3},\left[ 0,1%
\right] )$ as
\begin{equation*}
\psi _{R}\left( x\right) =\left\{
\begin{array}{ll}
1 & \left\vert x\right\vert <\frac{R}{2}, \\
0 & \left\vert x\right\vert >R,%
\end{array}%
\right.
\end{equation*}%
and $\left\vert \nabla \psi _{R}\right\vert \leq 1$ in $\mathbb{R}^{3}.$ Let
$u_{R}\left( x\right) =w_{\beta }\left( x\right) \psi _{R}(x).$ Then there
hold
\begin{equation}
\int_{\mathbb{R}^{3}}\left\vert u_{R}\right\vert ^{p}dx\rightarrow \int_{%
\mathbb{R}^{3}}\left\vert w_{\beta }\right\vert ^{p}dx\text{ and }\left\Vert
u_{R}\right\Vert _{H^{1}}\rightarrow \left\Vert w_{\beta }\right\Vert
_{H^{1}}\text{ as }R\rightarrow \infty ,  \label{6-5}
\end{equation}%
and%
\begin{equation*}
\int_{\mathbb{R}^{3}}\phi _{u_{R}}u_{R}^{2}dx=\int_{\mathbb{R}^{3}}\int_{%
\mathbb{R}^{3}}\frac{u_{R}^{2}\left( x\right) u_{R}^{2}\left( y\right) }{%
4\pi \left\vert x-y\right\vert }dxdy\rightarrow \int_{\mathbb{R}^{3}}\phi
_{w_{\beta }}w_{\beta }^{2}dx\text{ as }R\rightarrow \infty .
\end{equation*}%
Since $J_{\lambda ,\beta }\in C^{1}(H,\mathbb{R)}$, by $\left( \ref{6-3}%
\right) $--$\left( \ref{6-5}\right) $ there exists $R_{0}>0$ such that%
\begin{equation}
\frac{\lambda }{4}\overline{S}^{-2}S_{12/5}^{-4}\left\Vert
u_{R_{0}}\right\Vert _{H^{1}}^{2}<\frac{p-2}{2p}\left( \frac{4-p}{p}\right)
^{\left( 4-p\right) /\left( p-2\right) }\left( \frac{\int_{\mathbb{R}%
^{3}}g_{\beta }\left( s_{\beta }\right) \left\vert u_{R_{0}}\right\vert
^{p}dx}{\left\Vert u_{R_{0}}\right\Vert _{H^{1}}^{2}}\right) ^{2/\left(
p-2\right) }  \label{6-7}
\end{equation}%
and%
\begin{equation}
J_{\lambda ,\beta }\left( t_{\lambda ,\beta }^{+}\sqrt{s_{\beta }}%
u_{R_{0}},t_{\lambda ,\beta }^{+}\sqrt{1-s_{\beta }}u_{R_{0}}\right) <0.
\label{6-9}
\end{equation}%
Let%
\begin{equation*}
u_{R_{0},N}^{\left( i\right) }\left( x\right) =w_{\beta }\left(
x+iN^{3}e\right) \psi _{R_{0}}\left( x+iN^{3}e\right)
\end{equation*}%
for $e\in \mathbb{S}^{2}$ and $i=1,2,\ldots ,N$, where $N^{3}>2R_{0}.$ Then
we deduce that
\begin{equation*}
\left\Vert u_{R_{0},N}^{\left( i\right) }\right\Vert _{H^{1}}^{2}=\left\Vert
u_{R_{0}}\right\Vert _{H^{1}}^{2},\text{ }\int_{\mathbb{R}^{3}}\left\vert
u_{R_{0},N}^{\left( i\right) }\right\vert ^{p}dx=\int_{\mathbb{R}%
^{3}}\left\vert u_{R_{0}}\right\vert ^{p}dx
\end{equation*}%
and
\begin{eqnarray*}
\int_{\mathbb{R}^{3}}\phi _{u_{R_{0},N}^{\left( i\right) }}\left[
u_{R_{0},N}^{\left( i\right) }\right] ^{2}dx &=&\int_{\mathbb{R}^{3}}\int_{%
\mathbb{R}^{3}}\frac{\left[ u_{R_{0},N}^{\left( i\right) }\right] ^{2}\left(
x\right) \left[ u_{R_{0},N}^{\left( j\right) }\right] ^{2}\left( y\right) }{%
4\pi \left\vert x-y\right\vert }dxdy \\
&=&\int_{\mathbb{R}^{3}}\int_{\mathbb{R}^{3}}\frac{u_{R_{0}}^{2}\left(
x\right) u_{R_{0}}^{2}\left( y\right) }{4\pi \left\vert x-y\right\vert }dxdy.
\end{eqnarray*}%
for all $N.$ Moreover, by $\left( \ref{6-7}\right) $ and $\left( \ref{6-9}%
\right) ,$ there exists $N_{0}>0$ with $N_{0}^{3}>2R_{0}$ such that for
every $N\geq N_{0},$
\begin{equation*}
\frac{\lambda }{4}\overline{S}^{-2}S_{12/5}^{-4}\left\Vert
u_{R_{0},N}^{\left( i\right) }\right\Vert _{H^{1}}^{2}<\frac{p-2}{2p}\left(
\frac{4-p}{p}\right) ^{\left( 4-p\right) /\left( p-2\right) }\left( \frac{%
\int_{\mathbb{R}^{3}}g_{\beta }\left( s_{\beta }\right) \left\vert
u_{R_{0},N}^{\left( i\right) }\right\vert ^{p}dx}{\left\Vert
u_{R_{0},N}^{\left( i\right) }\right\Vert _{H^{1}}^{2}}\right) ^{2/\left(
p-2\right) }
\end{equation*}%
and
\begin{eqnarray*}
\inf_{t\geq 0}J_{\lambda ,\beta }\left( t\sqrt{s_{\beta }}%
u_{R_{0},N}^{\left( i\right) },t\sqrt{1-s_{\beta }}u_{R_{0},N}^{\left(
i\right) }\right) &\leq &J_{\lambda ,\beta }\left( t_{\lambda ,\beta }^{+}%
\sqrt{s_{\beta }}u_{R_{0},N}^{\left( i\right) },t_{\lambda ,\beta }^{+}\sqrt{%
1-s_{\beta }}u_{R_{0},N}^{\left( i\right) }\right) \\
&=&J_{\lambda ,\beta }\left( t_{\lambda ,\beta }^{+}\sqrt{s_{\beta }}%
u_{R_{0}},t_{\lambda ,\beta }^{+}\sqrt{1-s_{\beta }}u_{R_{0}}\right) \\
&<&0,
\end{eqnarray*}%
for all $e\in \mathbb{S}^{2}$ and $i=1,2,\ldots ,N.$ Let
\begin{equation*}
w_{R_{0},N}\left( x\right) =\sum_{i=1}^{N}u_{R_{0},N}^{\left( i\right) }.
\end{equation*}%
Observe that $w_{R_{0},N}$ is a sum of translation of $u_{R_{0}}.$ When $%
N^{3}\geq N_{0}^{3}>2R_{0}$, the summands have disjoint support. In such a
case we have%
\begin{equation}
\left\Vert w_{R_{0},N}\right\Vert _{H^{1}}^{2}=N\Vert u_{R_{0}}\Vert
_{H^{1}}^{2},  \label{14}
\end{equation}%
\begin{equation}
\int_{\mathbb{R}^{3}}\left\vert w_{R_{0},N}\right\vert
^{p}dx=\sum_{i=1}^{N}\int_{\mathbb{R}^{3}}\left\vert u_{R_{0},N}^{\left(
i\right) }\right\vert ^{p}dx=N\int_{\mathbb{R}^{3}}\left\vert
u_{R_{0}}\right\vert ^{p}dx,  \label{15}
\end{equation}%
and%
\begin{eqnarray}
&&\int_{\mathbb{R}^{3}}\phi _{\sqrt{s_{\beta }}w_{R_{0},N},\sqrt{1-s_{\beta }%
}w_{R_{0},N}}\left( \left( \sqrt{s_{\beta }}w_{R_{0},N}\right) ^{2}+\left(
\sqrt{1-s_{\beta }}w_{R_{0},N}\right) ^{2}\right) dx  \notag \\
&=&\int_{\mathbb{R}^{3}}\int_{\mathbb{R}^{3}}\frac{w_{R_{0},N}^{2}\left(
x\right) w_{R_{0},N}^{2}\left( y\right) }{4\pi \left\vert x-y\right\vert }%
dxdy  \notag \\
&=&\sum_{i=1}^{N}\int_{\mathbb{R}^{3}}\int_{\mathbb{R}^{3}}\frac{\left[
u_{R_{0},N}^{\left( i\right) }\right] ^{2}\left( x\right) \left[
u_{R_{0},N}^{\left( i\right) }\right] ^{2}\left( y\right) }{4\pi \left\vert
x-y\right\vert }dxdy  \notag \\
&&+\sum_{i\neq j}^{N}\int_{\mathbb{R}^{3}}\int_{\mathbb{R}^{3}}\frac{\left[
u_{R_{0},N}^{\left( i\right) }\right] ^{2}\left( x\right) \left[
u_{R_{0},N}^{\left( j\right) }\right] ^{2}\left( y\right) }{4\pi \left\vert
x-y\right\vert }dxdy.  \label{16}
\end{eqnarray}%
A straightforward calculation shows that
\begin{equation*}
\sum_{i\neq j}^{N}\int_{\mathbb{R}^{3}}\int_{\mathbb{R}^{3}}\frac{\left[
u_{R_{0},N}^{\left( i\right) }\right] ^{2}\left( x\right) \left[
u_{R_{0},N}^{\left( j\right) }\right] ^{2}\left( y\right) }{4\pi \left\vert
x-y\right\vert }dxdy\leq \frac{(N^{2}-N)}{N^{3}-2R_{0}}\left( \int_{\mathbb{R%
}^{3}}w_{\beta }^{2}\left( x\right) dx\right) ^{2},
\end{equation*}%
which implies that%
\begin{equation}
\sum_{i\neq j}^{N}\int_{\mathbb{R}^{3}}\int_{\mathbb{R}^{3}}\frac{\left[
u_{R_{0},N}^{\left( i\right) }\right] ^{2}\left( x\right) \left[
u_{R_{0},N}^{\left( j\right) }\right] ^{2}\left( y\right) }{4\pi \left\vert
x-y\right\vert }dxdy\rightarrow 0\text{ as }N\rightarrow \infty .  \label{17}
\end{equation}%
Next, we define
\begin{equation*}
\eta _{N}\left( t\right) =t^{-2}\left\Vert \left( \sqrt{s_{\beta }}%
w_{R_{0},N},\sqrt{1-s_{\beta }}w_{R_{0},N}\right) \right\Vert
_{H}^{2}-t^{p-4}\int_{\mathbb{R}^{3}}F_{\beta }\left( \sqrt{s_{\beta }}%
w_{R_{0},N},\sqrt{1-s_{\beta }}w_{R_{0},N}\right) dx
\end{equation*}%
and%
\begin{equation*}
\eta _{R_{0}}(t)=t^{-2}\left\Vert u_{R_{0}}\right\Vert
_{H^{1}}^{2}-t^{p-4}\int_{\mathbb{R}^{3}}g_{\beta }\left( s_{\beta }\right)
\left\vert u_{R_{0}}\right\vert ^{p}dx
\end{equation*}%
for $t>0.$ Then by $\left( \ref{14}\right) $ and $\left( \ref{15}\right) $,
we get
\begin{eqnarray}
\eta _{N}\left( t\right) &=&t^{-2}\left\Vert w_{R_{0},N}\right\Vert
_{H}^{2}-t^{p-4}\int_{\mathbb{R}^{3}}g_{\beta }\left( s_{\beta }\right)
\left\vert w_{R_{0},N}\right\vert ^{p}dx  \notag \\
&=&t^{-2}N\left\Vert u_{R_{0}}\right\Vert _{H^{1}}^{2}-t^{p-4}N\int_{\mathbb{%
R}^{3}}g_{\beta }\left( s_{\beta }\right) \left\vert u_{R_{0}}\right\vert
^{p}dx  \notag \\
&=&N\eta _{R_{0}}(t)\text{ for all }t>0.  \label{6-8}
\end{eqnarray}%
So one can see that $\left( t\sqrt{s_{\beta }}w_{R_{0},N},t\sqrt{1-s_{\beta }%
}w_{R_{0},N}\right) \in \mathbf{M}_{\lambda ,\beta }$ if and only if
\begin{equation*}
\eta _{N}\left( t\right) =-\lambda \int_{\mathbb{R}^{3}}\phi _{\sqrt{%
s_{\beta }}w_{R_{0},N},\sqrt{1-s_{\beta }}w_{R_{0},N}}\left( \left( \sqrt{%
s_{\beta }}w_{R_{0},N}\right) ^{2}+\left( \sqrt{1-s_{\beta }}%
w_{R_{0},N}\right) ^{2}\right) dx.
\end{equation*}%
We observe that
\begin{equation*}
\eta _{R_{0}}\left( T_{\beta }\left( u_{R_{0}}\right) \right) =0,\
\lim_{t\rightarrow 0^{+}}\eta _{R_{0}}(t)=\infty \text{ and }%
\lim_{t\rightarrow \infty }\eta _{R_{0}}(t)=0,
\end{equation*}%
where%
\begin{equation*}
T_{\beta }\left( u_{R_{0}}\right) :=\left( \frac{\left\Vert
u_{R_{0}}\right\Vert _{H^{1}}^{2}}{\int_{\mathbb{R}^{3}}g_{\beta }\left(
s_{\beta }\right) \left\vert u_{R_{0}}\right\vert ^{p}dx}\right) ^{1/\left(
p-2\right) }.
\end{equation*}%
Moreover, the first derivative of $\eta _{R_{0}}(t)$ is the following%
\begin{equation*}
\eta _{R_{0}}^{\prime }\left( t\right) =-2t^{-3}\left\Vert
u_{R_{0}}\right\Vert _{H^{1}}^{2}+\left( 4-p\right) t^{p-5}\int_{\mathbb{R}%
^{3}}g_{\beta }\left( s_{\beta }\right) \left\vert u_{R_{0}}\right\vert
^{p}dx.
\end{equation*}%
Then we obtain that $\eta _{R_{0}}$ is decreasing on $0<t<\left( \frac{%
2\left\Vert u_{R_{0}}\right\Vert _{H^{1}}^{2}}{\left( 4-p\right) \int_{%
\mathbb{R}^{3}}g_{\beta }\left( s_{\beta }\right) \left\vert
u_{R_{0}}\right\vert ^{p}dx}\right) ^{1/\left( p-2\right) }$ and is
increasing on $t>\left( \frac{2\left\Vert u_{R_{0}}\right\Vert _{H^{1}}^{2}}{%
\left( 4-p\right) \int_{\mathbb{R}^{3}}g_{\beta }\left( s_{\beta }\right)
\left\vert u_{R_{0}}\right\vert ^{p}dx}\right) ^{1/\left( p-2\right) }.$
Moreover, by $\left( \ref{6-7}\right) $ one has
\begin{eqnarray}
\inf_{t>0}\eta _{R_{0}}\left( t\right) &=&\eta _{R_{0}}\left( \left( \frac{%
2\left\Vert u_{R_{0}}\right\Vert _{H^{1}}^{2}}{\left( 4-p\right) \int_{%
\mathbb{R}^{3}}g_{\beta }\left( s_{\beta }\right) \left\vert
u_{R_{0}}\right\vert ^{p}dx}\right) ^{1/\left( p-2\right) }\right)  \notag \\
&=&-\frac{2\left( p-2\right) }{4-p}\left( \frac{\left( 4-p\right) \int_{%
\mathbb{R}^{3}}g_{\beta }\left( s_{\beta }\right) \left\vert
u_{R_{0}}\right\vert ^{p}dx}{2\left\Vert u_{R_{0}}\right\Vert _{H^{1}}^{2}}%
\right) ^{2/\left( p-2\right) }\left\Vert u_{R_{0}}\right\Vert _{H^{1}}^{2}
\notag \\
&<&-\lambda \overline{S}^{-2}S_{12/5}^{-4}\left\Vert u_{R_{0}}\right\Vert
_{H^{1}}^{4}  \notag \\
&=&-\lambda \int_{\mathbb{R}^{3}}\phi _{u_{R_{0}}}u_{R_{0}}^{2}dx.
\label{6-10}
\end{eqnarray}%
Then it follows from $\left( \ref{6-8}\right) $ and $\left( \ref{6-10}%
\right) $ that
\begin{eqnarray*}
\inf_{t>0}\eta _{N}\left( t\right) &\leq &\eta _{N}\left( \left( \frac{%
2\left\Vert u_{R_{0}}\right\Vert _{H^{1}}^{2}}{\left( 4-p\right) \int_{%
\mathbb{R}^{3}}g_{\beta }\left( s_{\beta }\right) \left\vert
u_{R_{0}}\right\vert ^{p}dx}\right) ^{1/\left( p-2\right) }\right) \\
&=&N\eta _{R_{0}}\left( \left( \frac{2\left\Vert u_{R_{0}}\right\Vert
_{H^{1}}^{2}}{\left( 4-p\right) \int_{\mathbb{R}^{3}}g_{\beta }\left(
s_{\beta }\right) \left\vert u_{R_{0}}\right\vert ^{p}dx}\right) ^{1/\left(
p-2\right) }\right) \\
&<&-\lambda N\int_{\mathbb{R}^{3}}\phi _{u_{R_{0}}}u_{R_{0}}^{2}dx \\
&=&-\lambda N\int_{\mathbb{R}^{3}}\int_{\mathbb{R}^{3}}\frac{%
u_{R_{0}}^{2}\left( x\right) u_{R_{0}}^{2}\left( y\right) }{4\pi \left\vert
x-y\right\vert }dxdy,
\end{eqnarray*}%
and together with $\left( \ref{17}\right) ,$ we further have
\begin{eqnarray*}
&&\inf_{t>0}\eta _{N}\left( t\right) \\
&<&-\lambda N\int_{\mathbb{R}^{3}}\int_{\mathbb{R}^{3}}\frac{%
u_{R_{0}}^{2}\left( x\right) u_{R_{0}}^{2}\left( y\right) }{4\pi \left\vert
x-y\right\vert }dxdy-\lambda \sum_{i\neq j}^{N}\int_{\mathbb{R}^{3}}\int_{%
\mathbb{R}^{3}}\frac{u_{R_{0},N}^{\left( i\right) }\left( x\right)
u_{R_{0},N}^{\left( j\right) }\left( y\right) }{4\pi \left\vert
x-y\right\vert }dxdy \\
&=&-\lambda \int_{\mathbb{R}^{3}}\phi _{\sqrt{s_{\beta }}w_{R_{0},N},\sqrt{%
1-s_{\beta }}w_{R_{0},N}}\left( \left( \sqrt{s_{\beta }}w_{R_{0},N}\right)
^{2}+\left( \sqrt{1-s_{\beta }}w_{R_{0},N}\right) ^{2}\right) dx
\end{eqnarray*}%
for $N$ sufficiently large. Thus, for $N$ sufficiently large, there exist
two positive constants $t_{\lambda ,N}^{\left( 1\right) }$ and $t_{\lambda
,N}^{\left( 2\right) }$ satisfying
\begin{equation*}
1<t_{\lambda ,N}^{\left( 1\right) }<\left( \frac{2\left\Vert
u_{R_{0}}\right\Vert _{H^{1}}^{2}}{\left( 4-p\right) \int_{\mathbb{R}%
^{3}}g_{\beta }\left( s_{\beta }\right) \left\vert u_{R_{0}}\right\vert
^{p}dx}\right) ^{1/\left( p-2\right) }<t_{\lambda ,N}^{\left( 2\right) }
\end{equation*}%
such that
\begin{equation*}
\eta _{N}\left( t_{\lambda ,N}^{\left( i\right) }\right) +\lambda \int_{%
\mathbb{R}^{3}}\phi _{\sqrt{s_{\beta }}w_{R_{0},N},\sqrt{1-s_{\beta }}%
w_{R_{0},N}}\left( \left( \sqrt{s_{\beta }}w_{R_{0},N}\right) ^{2}+\left(
\sqrt{1-s_{\beta }}w_{R_{0},N}\right) ^{2}\right) dx=0
\end{equation*}%
for $i=1,2$. That is, $\left( t_{\lambda ,N}^{\left( i\right) }\sqrt{%
s_{\beta }}w_{R_{0},N},t_{\lambda ,N}^{\left( i\right) }\sqrt{1-s_{\beta }}%
w_{R_{0},N}\right) \in \mathbf{M}_{\lambda ,\beta }$ for $i=1,2.$ A direct
calculation on the second order derivatives gives
\begin{eqnarray*}
h_{\lambda ,\left( t_{\lambda ,N}^{\left( 1\right) }\sqrt{s_{\beta }}%
w_{R,N},t_{\lambda ,N}^{\left( 1\right) }\sqrt{1-s_{\beta }}w_{R,N}\right)
}^{\prime \prime }\left( 1\right) &=&-2\left\Vert t_{\lambda ,N}^{\left(
1\right) }w_{R,N}\right\Vert _{H^{1}}^{2}+\left( 4-p\right) \int_{\mathbb{R}%
^{3}}g_{\beta }\left( s_{\beta }\right) \left\vert t_{\lambda ,N}^{\left(
1\right) }w_{R,N}\right\vert ^{p}dx \\
&=&\left( t_{\lambda ,N}^{\left( 1\right) }\right) ^{5}\eta _{N}^{\prime
}\left( t_{\lambda ,N}^{\left( 1\right) }\right) <0
\end{eqnarray*}%
and
\begin{eqnarray*}
h_{\lambda ,\left( t_{\lambda ,N}^{\left( 2\right) }\sqrt{s_{\beta }}%
w_{R,N},t_{\lambda ,N}^{\left( 2\right) }\sqrt{1-s_{\beta }}w_{R,N}\right)
}^{\prime \prime }\left( 1\right) &=&-2\left\Vert t_{\lambda ,N}^{\left(
2\right) }w_{R,N}\right\Vert _{H^{1}}^{2}+\left( 4-p\right) \int_{\mathbb{R}%
^{3}}g_{\beta }\left( s_{\beta }\right) \left\vert t_{\lambda ,N}^{\left(
2\right) }w_{R,N}\right\vert ^{p}dx \\
&=&\left( t_{\lambda ,N}^{\left( 2\right) }\right) ^{5}\eta _{N}^{\prime
}\left( t_{\lambda ,N}^{\left( 2\right) }\right) >0.
\end{eqnarray*}%
These enable us to conclude that
\begin{equation*}
\left( t_{\lambda ,N}^{\left( 1\right) }\sqrt{s_{\beta }}w_{R,N},t_{\lambda
,N}^{\left( 1\right) }\sqrt{1-s_{\beta }}w_{R,N}\right) \in \mathbf{M}%
_{\lambda ,\beta }^{-}
\end{equation*}%
and%
\begin{equation*}
\left( t_{\lambda ,N}^{\left( 2\right) }\sqrt{s_{\beta }}w_{R,N},t_{\lambda
,N}^{\left( 2\right) }\sqrt{1-s_{\beta }}w_{R,N}\right) \in \mathbf{M}%
_{\lambda ,\beta }^{+}.
\end{equation*}%
Moreover, it follows from $\left( \ref{14}\right) -\left( \ref{17}\right) $
that%
\begin{eqnarray*}
J_{\lambda ,\beta }\left( t_{\lambda ,N}^{\left( 2\right) }\sqrt{s_{\beta }}%
w_{R,N},t_{\lambda ,N}^{\left( 2\right) }\sqrt{1-s_{\beta }}w_{R,N}\right)
&=&\inf_{t>0}J_{\lambda ,\beta }\left( t\sqrt{s_{\beta }}w_{R,N},t\sqrt{%
1-s_{\beta }}w_{R,N}\right) \\
&\leq &J_{\lambda ,\beta }\left( t_{\lambda ,\beta }^{+}\sqrt{s_{\beta }}%
w_{R,N},t_{\lambda ,\beta }^{+}\sqrt{1-s_{\beta }}w_{R,N}\right) \\
&\leq &NJ_{\lambda ,\beta }\left( t_{\lambda ,\beta }^{+}\sqrt{s_{\beta }}%
u_{R_{0}},t_{\lambda ,\beta }^{+}\sqrt{1-s_{\beta }}u_{R_{0}}\right) +C_{0}\
\text{for some }C_{0}>0
\end{eqnarray*}%
and
\begin{equation*}
J_{\lambda ,\beta }\left( t_{\lambda ,N}^{\left( 2\right) }\sqrt{s_{\beta }}%
w_{R,N},t_{\lambda ,N}^{\left( 2\right) }\sqrt{1-s_{\beta }}w_{R,N}\right)
\rightarrow -\infty \text{ as }N\rightarrow \infty ,
\end{equation*}%
which implies that $\alpha _{\lambda ,\beta }^{+}=\inf_{\left( u,v\right)
\in \mathbf{M}_{\lambda ,\beta }^{+}}J_{\lambda ,\beta }\left( u,v\right)
=-\infty .$ This completes the proof.
\end{proof}

\begin{theorem}
\label{T4-2}Let $2<p<3$ and $\lambda >0.$ Then for each $\beta >\beta
_{0}\left( \lambda \right) ,$ System $(E_{\lambda ,\beta })$ has a vectorial
positive radial solution $\left( u_{\lambda ,\beta }^{\left( 2\right)
},v_{\lambda ,\beta }^{\left( 2\right) }\right) \in \overline{\mathbf{M}}%
_{\lambda ,\beta }^{\left( 2\right) }\cap H_{r}$ with $J_{\lambda ,\beta
}\left( u_{\lambda ,\beta }^{\left( 2\right) },v_{\lambda ,\beta }^{\left(
2\right) }\right) =\alpha _{\lambda ,\beta }^{+,r}.$
\end{theorem}

\begin{proof}
It follows from Lemma \ref{m5} and $\left( \ref{26}\right) $ that $%
J_{\lambda ,\beta }$ is coercive and bounded below on $H_{r}$ and
\begin{equation*}
-\infty <\alpha _{\lambda ,\beta }:=\inf_{(u,v)\in H_{r}}J_{\lambda ,\beta
}(u,v)<0.
\end{equation*}%
Then by the Ekeland variational principle \cite{E} and Palais criticality
principle \cite{P}, there exists a sequence $\{\left( u_{n},v_{n}\right)
\}\subset H_{r}$ such that
\begin{equation*}
J_{\lambda ,\beta }(u_{n},v_{n})=\alpha _{\lambda ,\beta }+o(1)\text{ and }%
J_{\lambda ,\beta }^{\prime }(u_{n},v_{n})=o(1)\text{ in }H^{-1}.
\end{equation*}%
Again, adopting the argument used in \cite[Theorem 4.3]{R1}, there exist a
subsequence $\{\left( u_{n},v_{n}\right) \}\subset H_{r}$ and $\left(
u_{\lambda ,\beta }^{\left( 2\right) },v_{\lambda ,\beta }^{\left( 2\right)
}\right) \in \mathbf{M}_{\lambda ,\beta }\cap H_{r}$ such that $\left(
u_{n},v_{n}\right) \rightarrow \left( u_{\lambda ,\beta }^{\left( 2\right)
},v_{\lambda ,\beta }^{\left( 2\right) }\right) $ strongly in $H_{r}$ and $%
\left( u_{\lambda ,\beta }^{\left( 2\right) },v_{\lambda ,\beta }^{\left(
2\right) }\right) $ is a solution of System $(E_{\lambda ,\beta })$
satisfying%
\begin{equation*}
J_{\lambda ,\beta }\left( u_{\lambda ,\beta }^{\left( 2\right) },v_{\lambda
,\beta }^{\left( 2\right) }\right) =\alpha _{\lambda ,\beta }<0.
\end{equation*}%
Moreover, by Lemma \ref{g5} it follows that $\left( u_{\lambda ,\beta
}^{\left( 2\right) },v_{\lambda ,\beta }^{\left( 2\right) }\right) \in
\overline{\mathbf{M}}_{\lambda ,\beta }^{\left( 2\right) }\cap H_{r}$ and
further%
\begin{equation*}
\alpha _{\lambda ,\beta }^{+,r}\leq J_{\lambda ,\beta }\left( u_{\lambda
,\beta }^{\left( 2\right) },v_{\lambda ,\beta }^{\left( 2\right) }\right)
=\alpha _{\lambda ,\beta }\leq \alpha _{\lambda ,\beta }^{+,r},
\end{equation*}%
which implies that%
\begin{equation*}
J_{\lambda ,\beta }\left( u_{\lambda ,\beta }^{\left( 2\right) },v_{\lambda
,\beta }^{\left( 2\right) }\right) =\alpha _{\lambda ,\beta }=\alpha
_{\lambda ,\beta }^{+,r},
\end{equation*}%
then so is $\left( \left\vert u_{\lambda ,\beta }^{\left( 2\right)
}\right\vert ,\left\vert v_{\lambda ,\beta }^{\left( 2\right) }\right\vert
\right) .$ According to Lemma \ref{g7}, we may assume that $\left(
u_{\lambda ,\beta }^{\left( 2\right) },v_{\lambda ,\beta }^{\left( 2\right)
}\right) $ is a nonnegative nontrivial critical point of $J_{\lambda ,\beta
} $. Furthermore, since $\alpha _{\lambda ,\beta }<0,$ it follows from
Theorem \ref{t5} that $u_{\lambda ,\beta }^{\left( 2\right) }\neq 0$ and $%
v_{\lambda ,\beta }^{\left( 2\right) }\neq 0.$ This completes the proof.
\end{proof}

\begin{theorem}
\label{T4-1}Let $2<p<4$ and $\lambda >0.$ Then for each $\beta >\beta \left(
\lambda \right) ,$ System $(E_{\lambda ,\beta })$ has a vectorial positive
solution $\left( u_{\lambda ,\beta }^{\left( 1\right) },v_{\lambda ,\beta
}^{\left( 1\right) }\right) \in \overline{\mathbf{M}}_{\lambda ,\beta
}^{\left( 1\right) }$ with $J_{\lambda ,\beta }\left( u_{\lambda ,\beta
}^{\left( 1\right) },v_{\lambda ,\beta }^{\left( 1\right) }\right) =\alpha
_{\lambda ,\beta }^{-}.$
\end{theorem}

\begin{proof}
By Lemmas \ref{g5}--\ref{g7} and the Ekeland variational principle, there
exists a minimizing sequence $\left\{ \left( u_{n},v_{n}\right) \right\}
\subset \overline{\mathbf{M}}_{\lambda ,\beta }^{\left( 1\right) }$ such
that
\begin{equation*}
J_{\lambda ,\beta }\left( u_{n},v_{n}\right) =\alpha _{\lambda ,\beta
}^{-}+o\left( 1\right) \text{ and }J_{\lambda ,\beta }^{\prime }\left(
u_{n},v_{n}\right) =o\left( 1\right) \text{ in }H^{-1}.  \label{18-0}
\end{equation*}%
Since $\left\{ \left( u_{n},v_{n}\right) \right\} $ is bounded, there exists
a convergent subsequence of $\left\{ \left( u_{n},v_{n}\right) \right\} $
(denoted as $\left\{ \left( u_{n},v_{n}\right) \right\} $ for notation
convenience) such that as $n\rightarrow \infty $,%
\begin{equation*}
\begin{array}{l}
\left( u_{n},v_{n}\right) \rightharpoonup \left( u_{0},v_{0}\right) \text{
weakly in }H, \\
\left( u_{n},v_{n}\right) \rightarrow \left( u_{0},v_{0}\right) \text{
strongly in }L_{loc}^{p}\left( \mathbb{R}^{3}\right) \times
L_{loc}^{p}\left( \mathbb{R}^{3}\right) , \\
\left( u_{n},v_{n}\right) \rightarrow \left( u_{0},v_{0}\right) \text{ a.e.
in }\mathbb{R}^{3}.%
\end{array}%
\end{equation*}

Now we claim that there exist a subsequence $\left\{ \left(
u_{n},v_{n}\right) \right\} _{n=1}^{\infty }$ and a sequence $%
\{x_{n}\}_{n=1}^{\infty }\subset \mathbb{R}^{3}$ such that
\begin{equation}
\int_{B^{N}\left( x_{n},R\right) }\left\vert \left( u_{n},v_{n}\right)
\right\vert ^{2}dx\geq d_{0}>0\text{ for all }n\in \mathbb{N},  \label{12-5}
\end{equation}%
where $d_{0}$ and $R$ are positive constants, independent of $n.$ Suppose on
the contrary. Then for all $R>0$, there holds%
\begin{equation*}
\sup_{x\in \mathbb{R}^{N}}\int_{B^{N}\left( x_{n},R\right) }\left\vert
\left( u_{n},v_{n}\right) \right\vert ^{2}dx\rightarrow 0\text{ as }%
n\rightarrow \infty .
\end{equation*}%
Applying the argument of \cite[Lemma I.1]{Li1} (see also \cite{Wi}) gives
\begin{equation*}
\int_{\mathbb{R}^{N}}(\left\vert u_{n}\right\vert ^{r}+\left\vert
v_{n}\right\vert ^{r})dx\rightarrow 0\text{ as }n\rightarrow \infty ,
\end{equation*}%
for all $2<r<2^{\ast }.$ Then we have $\int_{\mathbb{R}^{N}}F_{\beta }\left(
u_{n},v_{n}\right) dx\rightarrow 0$ and $\int_{\mathbb{R}^{3}}\phi
_{u_{n},v_{n}}\left( u_{n}^{2}+v_{n}^{2}\right) dx\rightarrow 0$ as $%
n\rightarrow \infty ,$ which implies that%
\begin{eqnarray*}
\alpha _{\lambda ,\beta }^{-}+o\left( 1\right) &=&J_{\lambda ,\beta }\left(
u_{\lambda ,\beta }^{\left( 1\right) },v_{\lambda ,\beta }^{\left( 1\right)
}\right) \\
&=&-\frac{1}{4}\int_{\mathbb{R}^{3}}\phi _{u_{n},v_{n}}\left(
u_{n}^{2}+v_{n}^{2}\right) dx+\frac{p-2}{2p}\int_{\mathbb{R}^{N}}F_{\beta
}\left( u_{n},v_{n}\right) dx \\
&=&o\left( 1\right) ,
\end{eqnarray*}%
which contradicts with $\alpha _{\lambda ,\beta }^{-}>0.$ So, $\left( \ref%
{12-5}\right) $ is claimed. Let $\left( \overline{u}_{n}\left( x\right) ,%
\overline{v}_{n}\left( x\right) \right) =\left( u_{n}\left( x-x_{n}\right)
,v_{n}\left( x-x_{n}\right) \right) .$ Clearly, $\left\{ \left( \overline{u}%
_{n},\overline{v}_{n}\right) \right\} \subset \overline{\mathbf{M}}_{\lambda
,\beta }^{\left( 1\right) }$ such that
\begin{equation}
J_{\lambda ,\beta }\left( \overline{u}_{n},\overline{v}_{n}\right) =\alpha
_{\lambda ,\beta }^{-}+o\left( 1\right) \text{ and }J_{\lambda ,\beta
}^{\prime }\left( \overline{u}_{n},\overline{v}_{n}\right) =o\left( 1\right)
\text{ in }H^{-1}.  \label{12-7}
\end{equation}%
Since $\left\{ \left( \overline{u}_{n},\overline{v}_{n}\right) \right\} $
also is bounded, there exist a convergent subsequence of $\left\{ \left(
\overline{u}_{n},\overline{v}_{n}\right) \right\} $ and $\left( u_{\lambda
,\beta }^{\left( 1\right) },v_{\lambda ,\beta }^{\left( 1\right) }\right)
\in H$ such that as $n\rightarrow \infty $,
\begin{equation}
\begin{array}{l}
\left( \overline{u}_{n},\overline{v}_{n}\right) \rightharpoonup \left(
u_{\lambda ,\beta }^{\left( 1\right) },v_{\lambda ,\beta }^{\left( 1\right)
}\right) \text{ weakly in }H, \\
\left( \overline{u}_{n},\overline{v}_{n}\right) \rightarrow \left(
u_{\lambda ,\beta }^{\left( 1\right) },v_{\lambda ,\beta }^{\left( 1\right)
}\right) \text{ strongly in }L_{loc}^{p}\left( \mathbb{R}^{3}\right) \times
L_{loc}^{p}\left( \mathbb{R}^{3}\right) , \\
\left( \overline{u}_{n},\overline{v}_{n}\right) \rightarrow \left(
u_{\lambda ,\beta }^{\left( 1\right) },v_{\lambda ,\beta }^{\left( 1\right)
}\right) \text{ a.e. in }\mathbb{R}^{3}.%
\end{array}
\label{15-1}
\end{equation}%
Moreover, by $\left( \ref{12-5}\right) $ and $(\ref{12-7})-(\ref{15-1})$, we
have%
\begin{equation*}
\int_{B^{N}\left( R\right) }\left\vert \left( u_{\lambda ,\beta }^{\left(
1\right) },v_{\lambda ,\beta }^{\left( 1\right) }\right) \right\vert
^{2}dx\geq d_{0}>0\text{ and }\left( u_{\lambda ,\beta }^{\left( 1\right)
},v_{\lambda ,\beta }^{\left( 1\right) }\right) \in \mathbf{M}_{\lambda
,\beta }.
\end{equation*}

Next, we show that%
\begin{equation*}
\left( \overline{u}_{n},\overline{v}_{n}\right) \rightarrow \left(
u_{\lambda ,\beta }^{\left( 1\right) },v_{\lambda ,\beta }^{\left( 1\right)
}\right) \text{ in }H.
\end{equation*}%
To this end, we suppose the contrary. Then it has
\begin{equation}
\left\Vert \left( u_{\lambda ,\beta }^{\left( 1\right) },v_{\lambda ,\beta
}^{\left( 1\right) }\right) \right\Vert _{H}^{2}<\liminf_{n\rightarrow
\infty }\left\Vert \left( \overline{u}_{n},\overline{v}_{n}\right)
\right\Vert _{H}^{2},  \label{15-4}
\end{equation}%
which implies that $\left\Vert \left( u_{\lambda ,\beta }^{\left( 1\right)
},v_{\lambda ,\beta }^{\left( 1\right) }\right) \right\Vert _{H}<\left(
\frac{\left( p-2\right) \overline{S}^{2}S_{12/5}^{4}}{\lambda (4-p)}\right)
^{1/2},$ since $\left\{ \left( \overline{u}_{n},\overline{v}_{n}\right)
\right\} \subset \overline{\mathbf{M}}_{\lambda ,\beta }^{\left( 1\right) }.$
From $\left( \ref{2-6-1}\right) ,$ the Sobolev inequality and Lemma \ref%
{L2-3}, it follows that
\begin{eqnarray*}
h_{\lambda ,\left( u_{\lambda ,\beta }^{\left( 1\right) },v_{\lambda ,\beta
}^{\left( 1\right) }\right) }^{\prime \prime }\left( 1\right) &=&-\left(
p-2\right) \left\Vert \left( u_{\lambda ,\beta }^{\left( 1\right)
},v_{\lambda ,\beta }^{\left( 1\right) }\right) \right\Vert _{H}^{2}+\lambda
\left( 4-p\right) \int_{\mathbb{R}^{3}}\phi _{u_{\lambda ,\beta }^{\left(
1\right) },v_{\lambda ,\beta }^{\left( 1\right) }}\left( \left[ u_{\lambda
,\beta }^{\left( 1\right) }\right] ^{2}+\left[ v_{\lambda ,\beta }^{\left(
1\right) }\right] ^{2}\right) dx \\
&\leq &\left\Vert \left( u_{\lambda ,\beta }^{\left( 1\right) },v_{\lambda
,\beta }^{\left( 1\right) }\right) \right\Vert _{H}^{2}\left[ \frac{\lambda
(4-p)}{\overline{S}^{2}S_{12/5}^{4}}\left\Vert \left( u_{\lambda ,\beta
}^{\left( 1\right) },v_{\lambda ,\beta }^{\left( 1\right) }\right)
\right\Vert _{H}^{2}-\left( p-2\right) \right] \\
&<&\left\Vert \left( u_{\lambda ,\beta }^{\left( 1\right) },v_{\lambda
,\beta }^{\left( 1\right) }\right) \right\Vert _{H}^{2}\left( \frac{\lambda
(4-p)}{\overline{S}^{2}S_{12/5}^{4}}\frac{\left( p-2\right) \overline{S}%
^{2}S_{12/5}^{4}}{\lambda (4-p)}-\left( p-2\right) \right) \\
&=&0.
\end{eqnarray*}%
This indicate that%
\begin{equation}
\left( u_{\lambda ,\beta }^{\left( 1\right) },v_{\lambda ,\beta }^{\left(
1\right) }\right) \in \mathbf{M}_{\lambda ,\beta }^{-}\text{ and }J_{\lambda
,\beta }\left( u_{\lambda ,\beta }^{\left( 1\right) },v_{\lambda ,\beta
}^{\left( 1\right) }\right) \geq \alpha _{\lambda ,\beta }^{-}.  \label{15-3}
\end{equation}%
Let $\left( w_{n},z_{n}\right) =\left( \overline{u}_{n}-u_{\lambda ,\beta
}^{\left( 1\right) },\overline{v}_{n}-v_{\lambda ,\beta }^{\left( 1\right)
}\right) .$ Then by $\left( \ref{15-1}\right) $ and $\left( \ref{15-4}%
\right) ,$ there exists $c_{0}>0$ such that%
\begin{equation*}
c_{0}\leq \left\Vert \left( w_{n},z_{n}\right) \right\Vert
_{H}^{2}=\left\Vert \left( \overline{u}_{n},\overline{v}_{n}\right)
\right\Vert _{H}^{2}-\left\Vert \left( u_{\lambda ,\beta }^{\left( 1\right)
},v_{\lambda ,\beta }^{\left( 1\right) }\right) \right\Vert _{H}^{2}+o(1),
\end{equation*}%
which implies that%
\begin{equation}
\left\Vert \left( w_{n},z_{n}\right) \right\Vert _{H}^{2}<\left( \frac{%
\left( p-2\right) \overline{S}^{2}S_{12/5}^{4}}{\lambda (4-p)}\right) ^{1/2}%
\text{ for }n\text{ sufficiently large.}  \label{15-5}
\end{equation}%
On the other hand, it follows from the Brezis-Lieb Lemma \cite{BLi} that%
\begin{equation*}
\int_{\mathbb{R}^{N}}F_{\beta }\left( \overline{u}_{n},\overline{v}%
_{n}\right) dx=\int_{\mathbb{R}^{N}}F_{\beta }\left( w_{n},z_{n}\right)
dx+\int_{\mathbb{R}^{N}}F_{\beta }\left( u_{\lambda ,\beta }^{\left(
1\right) },v_{\lambda ,\beta }^{\left( 1\right) }\right) dx+o(1)
\end{equation*}%
and%
\begin{equation*}
\int_{\mathbb{R}^{3}}\phi _{\overline{u}_{n},\overline{v}_{n}}\left(
\overline{u}_{n}^{2}+\overline{v}_{n}^{2}\right) dx=\int_{\mathbb{R}%
^{3}}\phi _{w_{n},z_{n}}\left( w_{n}^{2}+z_{n}^{2}\right) dx+\int_{\mathbb{R}%
^{3}}\phi _{u_{\lambda ,\beta }^{\left( 1\right) },v_{\lambda ,\beta
}^{\left( 1\right) }}\left( \left[ u_{\lambda ,\beta }^{\left( 1\right) }%
\right] ^{2}+\left[ v_{\lambda ,\beta }^{\left( 1\right) }\right]
^{2}\right) dx+o(1),
\end{equation*}%
which implies that%
\begin{equation}
\left\Vert \left( w_{n},z_{n}\right) \right\Vert _{H}^{2}+\int_{\mathbb{R}%
^{3}}\phi _{w_{n},z_{n}}\left( w_{n}^{2}+z_{n}^{2}\right) dx-\int_{\mathbb{R}%
^{N}}F_{\beta }\left( w_{n},z_{n}\right) dx=o\left( 1\right)  \label{15-6}
\end{equation}%
and%
\begin{equation}
J_{\lambda ,\beta }\left( \overline{u}_{n},\overline{v}_{n}\right)
=J_{\lambda ,\beta }\left( w_{n},z_{n}\right) +J_{\lambda ,\beta }\left(
u_{\lambda ,\beta }^{\left( 1\right) },v_{\lambda ,\beta }^{\left( 1\right)
}\right) +o(1).  \label{15-7}
\end{equation}%
Moreover, by $\left( \ref{15-5}\right) $ and $\left( \ref{15-6}\right) ,$
there exists $s_{n}=1+o\left( 1\right) $ such that%
\begin{equation*}
\left\Vert \left( s_{n}w_{n},s_{n}z_{n}\right) \right\Vert _{H}^{2}+\int_{%
\mathbb{R}^{3}}\phi _{s_{n}w_{n},s_{n}z_{n}}\left(
s_{n}^{2}w_{n}^{2}+s_{n}^{2}z_{n}^{2}\right) dx-\int_{\mathbb{R}%
^{N}}F_{\beta }\left( s_{n}w_{n},s_{n}z_{n}\right) dx=0
\end{equation*}%
and%
\begin{equation*}
\left\Vert \left( s_{n}w_{n},s_{n}z_{n}\right) \right\Vert _{H}^{2}<\left(
\frac{\left( p-2\right) \overline{S}^{2}S_{12/5}^{4}}{\lambda (4-p)}\right)
^{1/2}\text{ for }n\text{ sufficiently large.}
\end{equation*}%
Thus, we have%
\begin{equation*}
h_{\lambda ,\left( s_{n}w_{n},s_{n}z_{n}\right) }^{\prime \prime }\left(
1\right) =-\left( p-2\right) \left\Vert \left( s_{n}w_{n},s_{n}z_{n}\right)
\right\Vert _{H}^{2}+\lambda \left( 4-p\right) \int_{\mathbb{R}^{3}}\phi
_{s_{n}w_{n},s_{n}z_{n}}\left( s_{n}^{2}w_{n}^{2}+s_{n}^{2}z_{n}^{2}\right)
dx<0,
\end{equation*}%
which implies that%
\begin{equation}
J_{\lambda ,\beta }\left( s_{n}w_{n},s_{n}z_{n}\right) \geq \frac{1}{2}%
\alpha _{\lambda ,\beta }^{-}\text{ for }n\text{ sufficiently large.}
\label{15-2}
\end{equation}%
Hence, by (\ref{15-3}), (\ref{15-7}) and (\ref{15-2}) one has%
\begin{equation*}
\alpha _{\lambda ,\beta }^{-}+o\left( 1\right) =J_{\lambda ,\beta }\left(
\overline{u}_{n},\overline{v}_{n}\right) \geq \frac{3}{2}\alpha _{\lambda
,\beta }^{-}\text{ for }n\text{ sufficiently large.}
\end{equation*}%
This is a contradiction. Therefore, we conclude that $\left( \overline{u}%
_{n},\overline{v}_{n}\right) \rightarrow \left( u_{\lambda ,\beta }^{\left(
1\right) },v_{\lambda ,\beta }^{\left( 1\right) }\right) $ strongly in $H$
and $J_{\lambda ,\beta }\left( u_{\lambda ,\beta }^{\left( 1\right)
},v_{\lambda ,\beta }^{\left( 1\right) }\right) =\alpha _{\lambda ,\beta
}^{-},$ then so is $\left( \left\vert u_{\lambda ,\beta }^{\left( 1\right)
}\right\vert ,\left\vert v_{\lambda ,\beta }^{\left( 1\right) }\right\vert
\right) .$ According to Lemma \ref{g7}, we may assume that $\left(
u_{\lambda ,\beta }^{\left( 1\right) },v_{\lambda ,\beta }^{\left( 1\right)
}\right) $ is a nonnegative nontrivial critical point of $J_{\lambda ,\beta
} $. Moreover, since $\alpha _{\lambda ,\beta }^{-}\leq \frac{p-2}{2p}%
S_{p}^{2p/\left( p-2\right) }$ by (\ref{4-2}), it follows from Lemma \ref{l5}
that $u_{\lambda ,\beta }^{\left( 1\right) }\neq 0$ and $u_{\lambda ,\beta
}^{\left( 1\right) }\neq 0.$ The proof is complete.
\end{proof}

\textbf{We are now ready to prove Theorem \ref{t2}: }The proof directly
follows from Theorems \ref{T4-2} and \ref{T4-1}.

\section{Proofs of Theorems \protect\ref{t3} and \textbf{\protect\ref{t3-2}}}

Define
\begin{equation*}
\mathbb{A}_{\lambda ,\beta }:=\left\{ \left( u,v\right) \in H\setminus
\left\{ \left( 0,0\right) \right\} :\left( u,v\right) \text{ is a solution
of System }(E_{\lambda ,\beta })\text{ with }J_{\lambda ,\beta }\left(
u,v\right) <\frac{p-2}{2p}S_{p}^{2p/\left( p-2\right) }\right\} .
\end{equation*}%
Clearly, $\mathbb{A}_{\lambda ,\beta }\subset \mathbf{M}_{\lambda ,\beta }%
\left[ \frac{p-2}{2p}S_{p}^{2p/\left( p-2\right) }\right] .$ Furthermore, we
have the following result.

\begin{proposition}
\label{t6}Let $3\leq p<4$. Then for every $0<\lambda <\lambda _{0}$ and $%
\beta >0,$ we have $\mathbb{A}_{\lambda ,\beta }\subset \mathbf{M}_{\lambda
,\beta }^{-},$ where%
\begin{equation*}
\lambda _{0}:=\frac{6p\sqrt{3p}\left( p-2\right) \pi }{8\sqrt[3]{2}\left(
4-p\right) \left( 6-p\right) ^{3/2}S_{p}^{2p/\left( p-2\right) }}.
\end{equation*}
\end{proposition}

\begin{proof}
Let $\left( u_{0},v_{0}\right) \in \mathbb{A}_{\lambda ,\beta }.$ Then there
holds%
\begin{equation}
\left\Vert \left( u_{0},v_{0}\right) \right\Vert _{H}^{2}+\lambda \int_{%
\mathbb{R}^{3}}\phi _{u_{0},v_{0}}\left( u_{0}^{2}+v_{0}^{2}\right) dx-\int_{%
\mathbb{R}^{3}}F_{\beta }\left( u_{0},v_{0}\right) dx=0.  \label{2-6}
\end{equation}%
Following the argument of \cite[Lemma 3.1]{DM1}, we have the following
Pohozaev type identity%
\begin{equation}
\frac{1}{2}\int_{\mathbb{R}^{3}}(\left\vert \nabla u_{0}\right\vert
^{2}+\left\vert \nabla v_{0}\right\vert ^{2})dx+\frac{3}{2}\int_{\mathbb{R}%
^{3}}(u_{0}^{2}+v_{0}^{2})dx+\frac{5\lambda }{4}\int_{\mathbb{R}^{3}}\phi
_{u_{0},v_{0}}\left( u_{0}^{2}+v_{0}^{2}\right) dx=\frac{3}{p}\int_{\mathbb{R%
}^{3}}F_{\beta }\left( u_{0},v_{0}\right) dx.  \label{2-7}
\end{equation}%
Set%
\begin{equation}
\theta :=J_{\lambda ,\beta }\left( u_{0},v_{0}\right) =\frac{1}{2}\left\Vert
\left( u_{0},v_{0}\right) \right\Vert _{H}^{2}+\frac{\lambda }{4}\int_{%
\mathbb{R}^{3}}\phi _{u_{0},v_{0}}\left( u_{0}^{2}+v_{0}^{2}\right) dx-\frac{%
1}{p}\int_{\mathbb{R}^{3}}F_{\beta }\left( u_{0},v_{0}\right) dx.
\label{2-13}
\end{equation}%
Then it follows from $\left( \ref{2-6}\right) $--$\left( \ref{2-13}\right) $
that%
\begin{eqnarray}
\theta &=&\frac{p-2}{6-p}\int_{\mathbb{R}^{3}}(u_{0}^{2}+v_{0}^{2})dx+\frac{%
\lambda (p-3)}{6-p}\int_{\mathbb{R}^{3}}\phi _{u_{0},v_{0}}\left(
u_{0}^{2}+v_{0}^{2}\right) dx  \notag \\
&\geq &\frac{p-2}{6-p}\int_{\mathbb{R}^{3}}(u_{0}^{2}+v_{0}^{2})dx>0\text{
for }3\leq p<4.  \label{2-9}
\end{eqnarray}%
Moreover, by the Hardy-Littlewood-Sobolev and Gagliardo-Nirenberg
inequalities and $\left( \ref{2-9}\right) $, we have%
\begin{eqnarray}
\int_{\mathbb{R}^{3}}\phi _{u_{0},v_{0}}\left( u_{0}^{2}+v_{0}^{2}\right) dx
&\leq &\frac{8\sqrt[3]{2}}{3\sqrt[3]{\pi }}\left( \int_{\mathbb{R}%
^{3}}\left( u_{0}^{2}+v_{0}^{2}\right) ^{6/5}dx\right) ^{5/3}  \notag \\
&\leq &\frac{8\sqrt[3]{2}}{3\sqrt[3]{\pi }}\left( \int_{\mathbb{R}%
^{3}}\left( u_{0}^{2}+v_{0}^{2}\right) dx\right) ^{3/2}\left( \int_{\mathbb{R%
}^{3}}\left( u_{0}^{2}+v_{0}^{2}\right) ^{3}dx\right) ^{1/6}  \notag \\
&\leq &\frac{8\sqrt[3]{4}}{3\sqrt[3]{\pi }}\left( \frac{\theta (6-p)}{p-2}%
\right) ^{3/2}\left( \int_{\mathbb{R}^{3}}(u_{0}^{6}+v_{0}^{6})dx\right)
^{1/6}  \notag \\
&\leq &\frac{8\sqrt[3]{4}S}{3\sqrt[3]{\pi }}\left( \frac{\theta (6-p)}{p-2}%
\right) ^{3/2}\left[ \left( \int_{\mathbb{R}^{3}}|\nabla u_{0}|^{2}dx\right)
^{3}+\left( \int_{\mathbb{R}^{3}}|\nabla v_{0}|^{2}dx\right) ^{3}\right]
^{1/6}  \notag \\
&\leq &\frac{2^{11/3}}{3\sqrt[3]{\pi }}\frac{1}{\sqrt{3}}\frac{\sqrt[3]{4}}{%
\pi ^{\frac{2}{3}}}\left( \frac{\theta (6-p)}{p-2}\right) ^{3/2}\left( \int_{%
\mathbb{R}^{3}}(|\nabla u_{0}|^{2}+|\nabla v_{0}|^{2})dx\right) ^{1/2}
\notag \\
&=&\frac{16\sqrt[3]{2}}{3\sqrt{3}\pi }\left( \frac{\theta (6-p)}{p-2}\right)
^{3/2}\left( \int_{\mathbb{R}^{3}}(|\nabla u_{0}|^{2}+|\nabla
v_{0}|^{2})dx\right) ^{1/2}.  \label{2-8}
\end{eqnarray}%
We now define%
\begin{equation*}
\begin{array}{ll}
z_{1}=\int_{\mathbb{R}^{3}}(|\nabla u_{0}|^{2}+|\nabla v_{0}|^{2})dx, &
z_{2}=\int_{\mathbb{R}^{3}}\left( u_{0}^{2}+v_{0}^{2}\right) dx, \\
z_{3}=\int_{\mathbb{R}^{3}}\phi _{u_{0},v_{0}}\left(
u_{0}^{2}+v_{0}^{2}\right) dx, & z_{4}=\int_{\mathbb{R}^{3}}F_{\beta }\left(
u_{0},v_{0}\right) dx.%
\end{array}%
\end{equation*}%
Then from $\left( \ref{2-6}\right) -\left( \ref{2-13}\right) $ it follows
that%
\begin{equation}
\left\{
\begin{array}{l}
\frac{1}{2}z_{1}+\frac{1}{2}z_{2}+\frac{\lambda }{4}z_{3}-\frac{1}{p}%
z_{4}=\theta , \\
z_{1}+z_{2}+\lambda z_{3}-z_{4}=0, \\
\frac{1}{2}z_{1}+\frac{3}{2}z_{2}+\frac{5\lambda }{4}z_{3}-\frac{3}{p}%
z_{4}=0, \\
z_{i}>0\text{ for }i=1,2,3,4.%
\end{array}%
\right.  \label{2-10}
\end{equation}%
Moreover, by $\left( \ref{2-8}\right) $ and System $\left( \ref{2-10}\right)
$, we have%
\begin{equation*}
\theta =\frac{p-2}{6-p}z_{2}+\frac{\lambda \left( p-3\right) }{6-p}z_{3}\geq
\frac{p-2}{6-p}z_{2}>0
\end{equation*}%
and
\begin{equation}
z_{3}^{2}\leq \left( \frac{16\sqrt[3]{2}}{3\sqrt{3}\pi }\right) ^{2}\left(
\frac{6-p}{p-2}\theta \right) ^{3}z_{1}.  \label{2-12}
\end{equation}%
Next, we show that there exists a constant%
\begin{equation*}
\lambda _{0}:=\frac{6p\sqrt{3p}\left( p-2\right) \pi }{8\sqrt[3]{2}\left(
4-p\right) \left( 6-p\right) ^{3/2}S_{p}^{2p/\left( p-2\right) }}>0
\end{equation*}%
such that%
\begin{equation}
-\left( p-2\right) \left( z_{1}+z_{2}\right) +\lambda \left( 4-p\right)
z_{3}<0\text{ for all }\lambda \in \left( 0,\lambda _{0}\right) .
\label{2-19}
\end{equation}%
Since the general solution of System $\left( \ref{2-10}\right) $ is
\begin{equation}
\left[
\begin{array}{c}
z_{1} \\
z_{2} \\
z_{3} \\
z_{4}%
\end{array}%
\right] =\frac{\theta }{p-2}\left[
\begin{array}{c}
3(p-2) \\
6-p \\
0 \\
2p%
\end{array}%
\right] +t\left[
\begin{array}{c}
p-2 \\
-2(p-3) \\
\frac{2}{\lambda }(p-2) \\
p%
\end{array}%
\right] ,  \label{2-3}
\end{equation}%
where $s,t,w\in \mathbb{R}.$ From $\left( \ref{2-3}\right) $, we know that $%
z_{i}>0$ ($i=1,2,3,4)$ provided that the parameter $t$ satisfies
\begin{equation}
2(p-3)t<\frac{6-p}{p-2}\theta \text{ with }t>0.  \label{2-11}
\end{equation}%
Substituting $\left( \ref{2-3}\right) $ into $\left( \ref{2-12}\right) $, we
have%
\begin{eqnarray}
&&\left( \frac{2t(p-2)}{\lambda }\right) ^{2}-t\left( 4-p\right) \left(
\frac{16\sqrt[3]{2}}{3\sqrt{3}\pi }\right) ^{2}\left( \frac{\theta (6-p)}{p-2%
}\right) ^{3}  \notag \\
&\leq &\left( \frac{16\sqrt[3]{2}}{3\sqrt{3}\pi }\right) ^{2}\left( \frac{%
\theta (6-p)}{p-2}\right) ^{3}\left[ 3\theta +2\left( p-3\right) t\right] .
\label{2-14}
\end{eqnarray}%
Using the fact that $t>0$, it follows from (\ref{2-11}) and (\ref{2-14}) that%
\begin{eqnarray*}
&&\left[ \frac{4t^{2}(p-2)^{2}}{\lambda ^{2}}-t\theta ^{3}\left( 4-p\right)
\left( \frac{16\sqrt[3]{2}}{3\sqrt{3}\pi }\right) ^{2}\left( \frac{6-p}{p-2}%
\right) ^{3}\right] \\
&<&\theta ^{4}\left( \frac{16\sqrt[3]{2}}{3\sqrt{3}\pi }\right) ^{2}\left(
\frac{6-p}{p-2}\right) ^{3}\left( 3+\frac{6-p}{p-2}\right)
\end{eqnarray*}%
or%
\begin{equation*}
\frac{4t^{2}(p-2)^{2}}{\lambda ^{2}}-At\theta ^{3}\left( 4-p\right) -\frac{%
2pA\theta ^{4}}{p-2}<0,
\end{equation*}%
where $A:=\left( \frac{16\sqrt[3]{2}}{3\sqrt{3}\pi }\right) ^{2}\left( \frac{%
6-p}{p-2}\right) ^{3}$. This implies that the parameter $t$ satisfies
\begin{equation}
0<t<\frac{\lambda ^{2}\left( A\left( 4-p\right) \theta ^{3}+\sqrt{%
A^{2}\left( 4-p\right) ^{2}\theta ^{6}+\frac{32p(p-2)A\theta ^{4}}{\lambda
^{2}}}\right) }{8(p-2)^{2}}.  \label{2-15}
\end{equation}%
Using $\left( \ref{2-3}\right) $ again, we have
\begin{equation}
-\left( p-2\right) \left( z_{1}+z_{2}\right) +\lambda \left( 4-p\right)
z_{3}=-2p\theta +t(p-2)(4-p).  \label{2-18}
\end{equation}%
Then, it follows from (\ref{2-15}) and (\ref{2-18}) that%
\begin{eqnarray}
&&\frac{-\left( p-2\right) \left( z_{1}+z_{2}\right) +\lambda \left(
4-p\right) z_{3}}{\theta }  \notag \\
&\leq &-2p+(p-2)(4-p)\frac{\lambda ^{2}\left( A\left( 4-p\right) \theta ^{3}+%
\sqrt{A^{2}\left( 4-p\right) ^{2}\theta ^{6}+\frac{32p(p-2)A\theta ^{4}}{%
\lambda ^{2}}}\right) }{8(p-2)^{2}\theta }  \notag \\
&=&-2p+\frac{\lambda ^{2}(4-p)\left( A\left( 4-p\right) \theta ^{2}+\sqrt{%
A^{2}\left( 4-p\right) ^{2}\theta ^{4}+\frac{32p(p-2)A\theta ^{2}}{\lambda
^{2}}}\right) }{8\left( p-2\right) }.  \label{2-16}
\end{eqnarray}%
In addition, a direct calculation shows that%
\begin{equation}
A\left( 4-p\right) \lambda ^{2}\theta ^{2}+\lambda ^{2}\sqrt{A^{2}\left(
4-p\right) ^{2}\theta ^{4}+\frac{32p(p-2)A\theta ^{2}}{\lambda ^{2}}}<\frac{%
16p\left( p-2\right) }{4-p}  \label{2-20}
\end{equation}%
for all $0<\theta <\frac{p-2}{2p}S_{p}^{2p/\left( p-2\right) }$ and $%
0<\lambda <\frac{4p}{\left( 4-p\right) \left( p-2\right) S_{p}^{2p/\left(
p-2\right) }}\left( \frac{p\left( p-2\right) }{A}\right) ^{1/2}.$ Hence, it
follows from $\left( \ref{2-16}\right) $ and $\left( \ref{2-20}\right) $
that for each $\lambda \in (0,\lambda _{0})$,
\begin{equation*}
-\left( p-2\right) \left( z_{1}+z_{2}\right) +\lambda \left( 4-p\right)
z_{3}<0,
\end{equation*}%
where $\lambda _{0}$ is as in (\ref{2-20}). Namely, (\ref{2-19}) is proved.
This shows that%
\begin{equation*}
h_{\lambda ,\left( u_{0},v_{0}\right) }^{\prime \prime }\left( 1\right)
=-\left( p-2\right) \left\Vert \left( u_{0},v_{0}\right) \right\Vert
_{H}^{2}+\lambda \left( 4-p\right) \int_{\mathbb{R}^{3}}\phi
_{u_{0},v_{0}}\left( u_{0}^{2}+v_{0}^{2}\right) dx<0,
\end{equation*}%
leading to $\left( u_{0},v_{0}\right) \in \mathbf{M}_{\lambda ,\beta }^{-}.$
Therefore, we have $\mathbb{A}_{\lambda ,\beta }\subset \mathbf{M}_{\lambda
,\beta }^{-}.$ This completes the proof.
\end{proof}

\textbf{We are now ready to prove Theorem \ref{t3}:} By Theorem \ref{T4-1},
System $(E_{\lambda ,\beta })$ has a vectorial solution $\left( u_{\lambda
,\beta }^{\left( 1\right) },v_{\lambda ,\beta }^{\left( 1\right) }\right)
\in \overline{\mathbf{M}}_{\lambda ,\beta }^{\left( 1\right) },$ which
satisfies%
\begin{equation*}
J_{\lambda ,\beta }\left( u_{\lambda ,\beta }^{\left( 1\right) },v_{\lambda
,\beta }^{\left( 1\right) }\right) =\alpha _{\lambda ,\beta }^{-}<\frac{%
\left( p-2\right) ^{2}\overline{S}^{2}S_{12/5}^{4}}{4p(4-p)k\left( \lambda
\right) }
\end{equation*}%
and%
\begin{equation*}
J_{\lambda ,\beta }\left( u_{\lambda ,\beta }^{\left( 1\right) },v_{\lambda
,\beta }^{\left( 1\right) }\right) =\alpha _{\lambda ,\beta }^{-}=\inf_{u\in
\mathbf{M}_{\lambda ,\beta }^{-}}J_{\lambda ,\beta }(u,v).
\end{equation*}%
Since $\frac{\left( p-2\right) ^{2}\overline{S}^{2}S_{12/5}^{4}}{%
4p(4-p)k\left( \lambda \right) }\leq \frac{p-2}{2p}S_{p}^{2p/\left(
p-2\right) }$, it follows from Proposition \ref{t6} that
\begin{equation*}
J_{\lambda ,\beta }\left( u_{\lambda ,\beta }^{\left( 1\right) },v_{\lambda
,\beta }^{\left( 1\right) }\right) =\alpha _{\lambda ,\beta }^{-}=\inf_{u\in
\mathbb{A}_{\lambda ,\beta }}J_{\lambda ,\beta }(u,v),
\end{equation*}%
which implies that $\left( u_{\lambda ,\beta }^{\left( 1\right) },v_{\lambda
,\beta }^{\left( 1\right) }\right) $ is a vectorial ground state solution of
System $(E_{\lambda ,\beta }).$ This completes the proof.

\begin{proposition}
\label{g3}Let $\frac{1+\sqrt{73}}{3}\leq p<6,\lambda >0$ and $\beta >0.$ Let
$\left( u_{0},v_{0}\right) $ be a nontrivial solution of System $(E_{\lambda
,\beta }).$ Then $\left( u_{0},v_{0}\right) \in \mathbf{M}_{\lambda ,\beta
}^{-}.$
\end{proposition}

\begin{proof}
Since $\left( u_{0},v_{0}\right) $ is a nontrivial solution of System $%
(E_{\lambda ,\beta })$, we have%
\begin{equation}
\left\Vert \left( u_{0},v_{0}\right) \right\Vert _{H}^{2}+\lambda \int_{%
\mathbb{R}^{3}}\phi _{u_{0},v_{0}}\left( u_{0}^{2}+v_{0}^{2}\right) dx-\int_{%
\mathbb{R}^{3}}F_{\beta }\left( u_{0},v_{0}\right) dx=0  \label{6-11}
\end{equation}%
and%
\begin{eqnarray}
\frac{1}{2}\int_{\mathbb{R}^{3}}(\left\vert \nabla u_{0}\right\vert
^{2}+\left\vert \nabla v_{0}\right\vert ^{2})dx+\frac{3}{2}\int_{\mathbb{R}%
^{3}}(u_{0}^{2}+v_{0}^{2})dx+\frac{5\lambda }{4}\int_{\mathbb{R}^{3}}\phi
_{u_{0},v_{0}}\left( u_{0}^{2}+v_{0}^{2}\right) dx =\frac{3}{p}\int_{\mathbb{%
R}^{3}}F_{\beta }\left( u_{0},v_{0}\right) dx.  \label{6-12}
\end{eqnarray}%
Combining $(\ref{6-11})$ with $(\ref{6-12})$, one has
\begin{equation*}
\int_{\mathbb{R}^{3}}(\left\vert \nabla u_{0}\right\vert ^{2}+\left\vert
\nabla v_{0}\right\vert ^{2})dx=\frac{3(p-2)}{6-p}\int_{\mathbb{R}%
^{3}}(u_{0}^{2}+v_{0}^{2})dx+\frac{\lambda (5p-12)}{2\left( 6-p\right) }%
\int_{\mathbb{R}^{3}}\phi _{u_{0},v_{0}}\left( u_{0}^{2}+v_{0}^{2}\right) dx.
\end{equation*}%
Using this, together with $\left( \ref{2-6-1}\right) ,$ gives
\begin{eqnarray*}
h_{\lambda ,\left( u_{0},v_{0}\right) }^{\prime \prime }\left( 1\right)
&=&-\left( p-2\right) \left\Vert \left( u_{0},v_{0}\right) \right\Vert
_{H}^{2}+\lambda \left( 4-p\right) \int_{\mathbb{R}^{3}}\phi
_{u_{0},v_{0}}\left( u_{0}^{2}+v_{0}^{2}\right) dx \\
&=&-\frac{2p(p-2)}{6-p}\int_{\mathbb{R}^{3}}(u_{0}^{2}+v_{0}^{2})dx-\frac{%
\lambda (3p^{2}-2p-24)}{2\left( 6-p\right) }\int_{\mathbb{R}^{3}}\phi
_{u_{0},v_{0}}\left( u_{0}^{2}+v_{0}^{2}\right) dx \\
&<&0,
\end{eqnarray*}%
where we have also used the fact of $3p^{2}-2p-24\geq 0$ if $\frac{1+\sqrt{73%
}}{3}\leq p<6.$ Therefore, there holds $\left( u_{0},v_{0}\right) \in
\mathbf{M}_{\lambda ,\beta }^{-}.$ This completes the proof.
\end{proof}

\textbf{We are now ready to prove Theorem \ref{t3-2}:} For $\lambda >0$ and $%
\beta >\beta \left( \lambda \right) .$ By Theorem \ref{T4-1}, System $%
(E_{\lambda ,\beta })$ has a vectorial solution $\left( u_{\lambda ,\beta
}^{\left( 1\right) },v_{\lambda ,\beta }^{\left( 1\right) }\right) \in
\overline{\mathbf{M}}_{\lambda ,\beta }^{\left( 1\right) }$ satisfying%
\begin{equation*}
J_{\lambda ,\beta }\left( u_{\lambda ,\beta }^{\left( 1\right) },v_{\lambda
,\beta }^{\left( 1\right) }\right) =\alpha _{\lambda ,\beta }^{-}=\inf_{u\in
\mathbf{M}_{\lambda ,\beta }^{-}}J_{\lambda ,\beta }\left( u,v\right) ,
\end{equation*}%
and according to Proposition \ref{g3}, we conclude that $\left( u_{\lambda
,\beta }^{\left( 1\right) },v_{\lambda ,\beta }^{\left( 1\right) }\right) $
is a vectorial ground state solution of System $(E_{\lambda ,\beta }).$ This
completes the proof.

\section{Appendix}

\begin{theorem}
\label{AT-1}Let $2<p<3$ and $\beta \geq 0$. Then the following statements
are true.\newline
$\left( i\right) $ $0<\Lambda \left( \beta \right) <\infty ;$\newline
$\left( ii\right) $ $\Lambda \left( \beta \right) $ is achieved, i.e. there
exists $\left( u_{0},v_{0}\right) \in H_{r}\setminus \left\{ \left(
0,0\right) \right\} $ such that
\begin{equation*}
\Lambda \left( \beta \right) =\frac{\frac{1}{p}\int_{\mathbb{R}^{3}}F_{\beta
}\left( u_{0},v_{0}\right) dx-\frac{1}{2}\left\Vert \left(
u_{0},v_{0}\right) \right\Vert _{H}^{2}}{\int_{\mathbb{R}^{3}}\phi
_{u_{0},v_{0}}\left( u_{0}^{2}+v_{0}^{2}\right) dx}>0.
\end{equation*}
\end{theorem}

\begin{proof}
$\left( i\right) $ Since $2<p<3,$ by Fatou's lemma, for $\left( u,v\right)
\in H_{r}\setminus \left\{ \left( 0,0\right) \right\} $ with $\int_{\mathbb{R%
}^{3}}F_{\beta }\left( u,v\right) dx>0,$ we have%
\begin{equation*}
\lim_{t\rightarrow \infty }\frac{1}{t^{p}}\left[ \frac{1}{2}\left\Vert
\left( tu,tv\right) \right\Vert _{H}^{2}-\frac{1}{p}\int_{\mathbb{R}%
^{3}}F_{\beta }\left( tu,tv\right) dx\right] =-\frac{1}{p}\int_{\mathbb{R}%
^{3}}F_{\beta }\left( u,v\right) dx<0,
\end{equation*}%
which implies that there exists $\left( e_{1},e_{2}\right) \in H_{r}$ such
that
\begin{equation*}
\frac{1}{2}\left\Vert \left( e_{1},e_{2}\right) \right\Vert _{H}^{2}-\frac{1%
}{p}\int_{\mathbb{R}^{3}}F_{\beta }\left( e_{1},e_{2}\right) dx<0.
\end{equation*}%
Then, for each $\left( u,v\right) \in H_{r}\setminus \left\{ \left(
0,0\right) \right\} $ with $\frac{1}{2}\left\Vert \left( u,v\right)
\right\Vert _{H}^{2}-\frac{1}{p}\int_{\mathbb{R}^{3}}F_{\beta }\left(
u,v\right) dx<0,$ there exists $c_{0}>0$ such that%
\begin{equation*}
\frac{1}{2}\left\Vert \left( u,v\right) \right\Vert _{H}^{2}+\frac{c_{0}}{4}%
\int_{\mathbb{R}^{3}}\phi _{u,v}\left( u^{2}+v^{2}\right) dx-\frac{1}{p}%
\int_{\mathbb{R}^{3}}F_{\beta }\left( u,v\right) dx<0
\end{equation*}%
or%
\begin{equation*}
\frac{c_{0}}{4}<\frac{\frac{1}{p}\int_{\mathbb{R}^{3}}F_{\beta }\left(
u,v\right) dx-\frac{1}{2}\left\Vert \left( u,v\right) \right\Vert _{H}^{2}}{%
\int_{\mathbb{R}^{3}}\phi _{u,v}\left( u^{2}+v^{2}\right) dx}.
\end{equation*}%
This indicates that there exists $\hat{c}_{0}>0$ such that $\Lambda \left(
\beta \right) \geq \hat{c}_{0}>0.$

Next, we show that $0<\Lambda \left( \beta \right) <\infty .$ By Young's
inequality, we have
\begin{equation}
\frac{1+\beta }{p}\left\vert w\right\vert ^{p}\leq \frac{1}{2}%
w^{2}+C_{p,\beta }\left\vert w\right\vert ^{3},  \label{A-1}
\end{equation}%
where
\begin{equation*}
C_{p,\beta }=\left( p-2\right) \left[ 2\left( 3-p\right) \right] ^{\frac{3-p%
}{p-2}}\left( \frac{1+\beta }{p}\right) ^{\frac{1}{p-2}}>0.
\end{equation*}%
Moreover, similar to (\ref{2-0}) and (\ref{2-00}), we have
\begin{equation}
C_{p,\beta }\int_{\mathbb{R}^{3}}(\left\vert u\right\vert
^{3}+v^{2}\left\vert u\right\vert )dx\leq \frac{1}{2}\int_{\mathbb{R}%
^{3}}\left\vert \nabla u\right\vert ^{2}dx+\frac{C_{p,\beta }^{2}}{2}\int_{%
\mathbb{R}^{3}}\phi _{u,v}\left( u^{2}+v^{2}\right) dx  \label{A-2}
\end{equation}%
and%
\begin{equation}
C_{p,\beta }\int_{\mathbb{R}^{3}}(u^{2}\left\vert v\right\vert +\left\vert
v\right\vert ^{3})dx\leq \frac{1}{2}\int_{\mathbb{R}^{3}}\left\vert \nabla
v\right\vert ^{2}dx+\frac{C_{p,\beta }^{2}}{2}\int_{\mathbb{R}^{3}}\phi
_{u,v}\left( u^{2}+v^{2}\right) dx  \label{A-3}
\end{equation}%
for all $\left( u,v\right) \in H_{r}\setminus \left\{ \left( 0,0\right)
\right\} .$ Then it follows from (\ref{A-1})--(\ref{A-3}) that%
\begin{eqnarray*}
&&\frac{\frac{1}{p}\int_{\mathbb{R}^{3}}F_{\beta }\left( u,v\right) dx-\frac{%
1}{2}\left\Vert \left( u,v\right) \right\Vert _{H}^{2}}{\int_{\mathbb{R}%
^{3}}\phi _{u,v}(u^{2}+v^{2})dx} \\
&\leq &2C_{p,\beta }^{2}\times \frac{\frac{1+\beta }{p}\int_{\mathbb{R}%
^{3}}(\left\vert u\right\vert ^{p}+\left\vert v\right\vert ^{p})dx-\frac{1}{2%
}\left\Vert \left( u,v\right) \right\Vert _{H}^{2}}{2C_{p,\beta }\int_{%
\mathbb{R}^{3}}(\left\vert u\right\vert ^{3}+\left\vert v\right\vert
^{3})dx+2C_{p,\beta }\int_{\mathbb{R}^{3}}(u^{2}\left\vert v\right\vert
+v^{2}\left\vert u\right\vert )dx-\int_{\mathbb{R}^{3}}(\left\vert \nabla
u\right\vert ^{2}+\left\vert \nabla v\right\vert ^{2})dx} \\
&\leq &2C_{p,\beta }^{2}\times \frac{C_{p,\beta }\int_{\mathbb{R}%
^{3}}(\left\vert u\right\vert ^{3}+\left\vert v\right\vert ^{3})dx-\frac{1}{2%
}\int_{\mathbb{R}^{3}}(\left\vert \nabla u\right\vert ^{2}+\left\vert \nabla
v\right\vert ^{2})dx}{2C_{p,\beta }\int_{\mathbb{R}^{3}}(\left\vert
u\right\vert ^{3}+\left\vert v\right\vert ^{3})dx-\int_{\mathbb{R}%
^{3}}(\left\vert \nabla u\right\vert ^{2}+\left\vert \nabla v\right\vert
^{2})dx} \\
&=&C_{p,\beta }^{2},
\end{eqnarray*}%
which shows that%
\begin{equation*}
0<\Lambda \left( \beta \right) :=\sup_{\left( u,v\right) \in H_{r}\setminus
\left\{ \left( 0,0\right) \right\} }\frac{\frac{1}{p}\int_{\mathbb{R}%
^{3}}F_{\beta }\left( u,v\right) dx-\frac{1}{2}\left\Vert (u,v)\right\Vert
_{H}^{2}}{\int_{\mathbb{R}^{3}}\phi _{u,v}(u^{2}+v^{2})dx}\leq C_{p,\beta
}^{2}.
\end{equation*}%
$\left( ii\right) $ Let $\left\{ \left( u_{n},v_{n}\right) \right\} \subset
H_{r}\setminus \{\left( 0,0\right) \}$ be a maximum sequence of $(\ref{1-7}).
$ First of all, we claim that $\{\left( u_{n},v_{n}\right) \}$ is bounded in
$H_{r}$. Suppose on the contrary. Then $\left\Vert (u_{n},v_{n})\right\Vert
_{H}\rightarrow \infty $ as $n\rightarrow \infty $. Since $0<\Lambda \left(
\beta \right) <\infty $ and%
\begin{equation*}
\frac{\frac{1}{p}\int_{\mathbb{R}^{3}}F_{\beta }\left( u_{n},v_{n}\right) dx-%
\frac{1}{2}\left\Vert \left( u_{n},v_{n}\right) \right\Vert _{H}^{2}}{\int_{%
\mathbb{R}^{3}}\phi _{u_{n},v_{n}}\left( u_{n}^{2}+v_{n}^{2}\right) dx}%
=\Lambda \left( \beta \right) +o\left( 1\right) ,
\end{equation*}%
there exists $C_{1}>0$ such that%
\begin{equation}
\widetilde{J}\left( u_{n},v_{n}\right) :=\frac{1}{2}\left\Vert \left(
u_{n},v_{n}\right) \right\Vert _{H}^{2}+C_{1}\int_{\mathbb{R}^{3}}\phi
_{u_{n},v_{n}}\left( u_{n}^{2}+v_{n}^{2}\right) dx-\frac{1}{p}\int_{\mathbb{R%
}^{3}}F_{\beta }\left( u_{n},v_{n}\right) dx\leq 0  \label{A-4}
\end{equation}%
for $n$ sufficiently large. Similar to (\ref{2-0}) and (\ref{2-00}), we have%
\begin{equation}
\frac{\sqrt{C_{1}}}{2}\int_{\mathbb{R}^{3}}(\left\vert u\right\vert
^{3}+v^{2}\left\vert u\right\vert )dx\leq \frac{1}{4}\int_{\mathbb{R}%
^{3}}\left\vert \nabla u\right\vert ^{2}dx+\frac{C_{1}}{4}\int_{\mathbb{R}%
^{3}}\phi _{u,v}\left( u^{2}+v^{2}\right) dx  \label{A-5}
\end{equation}%
and%
\begin{equation}
\frac{\sqrt{C_{1}}}{2}\int_{\mathbb{R}^{3}}(u^{2}\left\vert v\right\vert
+\left\vert v\right\vert ^{3})dx\leq \frac{1}{4}\int_{\mathbb{R}%
^{3}}\left\vert \nabla v\right\vert ^{2}dx+\frac{C_{1}}{4}\int_{\mathbb{R}%
^{3}}\phi _{u,v}\left( u^{2}+v^{2}\right) dx  \label{A-6}
\end{equation}%
for all $\left( u,v\right) \in H_{r}.$ Then it follows from (\ref{A-4})--(%
\ref{A-6}) that%
\begin{equation*}
\widetilde{J}\left( u_{n},v_{n}\right) \geq \frac{1}{4}\left\Vert \left(
u_{n},v_{n}\right) \right\Vert _{H}^{2}+\frac{C_{1}}{2}\int_{\mathbb{R}%
^{3}}\phi _{u_{n},v_{n}}\left( u_{n}^{2}+v_{n}^{2}\right) dx+\int_{\mathbb{R}%
^{3}}(f_{\beta }\left( u_{n}\right) +f_{\beta }\left( v_{n}\right) )dx,
\end{equation*}%
where $f_{\beta }\left( s\right) :=\frac{1}{4}s^{2}+\frac{\sqrt{C_{1}}}{2}%
s^{3}-\frac{1+\beta }{p}s^{p}$ for $s>0.$ It is clear that $f_{\beta }$ is
positive for $s\rightarrow 0^{+}$ or $s\rightarrow \infty ,$ since $2<p<3$
and $\beta \geq 0.$ Define%
\begin{equation*}
m_{\beta }:=\inf_{s>0}f_{\beta }(s).
\end{equation*}%
If $m_{\beta }\geq 0,$ then by (\ref{A-4}) we have
\begin{equation*}
0\geq \widetilde{J}\left( u_{n},v_{n}\right) \geq \frac{1}{4}\left\Vert
\left( u_{n},v_{n}\right) \right\Vert _{H}^{2}+\frac{C_{1}}{2}\int_{\mathbb{R%
}^{3}}\phi _{u_{n},v_{n}}(u_{n}^{2}+v_{n}^{2})dx>0,
\end{equation*}%
which is a contradiction. We now assume that $m_{\beta }<0.$ Then the set $%
\left\{ s>0:f_{\beta }\left( s\right) <0\right\} $ is an open interval $%
\left( s_{1},s_{2}\right) $ with $s_{1}>0.$ Note that constants $%
s_{1},s_{2},m_{\beta }$ depend on $p,\beta $ and $C_{1}$. Thus, there holds
\begin{eqnarray}
\widetilde{J}\left( u_{n},v_{n}\right)  &\geq &\frac{1}{4}\left\Vert \left(
u_{n},v_{n}\right) \right\Vert _{H}^{2}+\frac{C_{1}}{2}\int_{\mathbb{R}%
^{3}}\phi _{u_{n},v_{n}}(u_{n}^{2}+v_{n}^{2})dx+\int_{\mathbb{R}%
^{3}}(f_{\beta }\left( u_{n}\right) +f_{\beta }\left( v_{n}\right) )dx
\notag \\
&\geq &\frac{1}{4}\left\Vert \left( u_{n},v_{n}\right) \right\Vert _{H}^{2}+%
\frac{C_{1}}{2}\int_{\mathbb{R}^{3}}\phi
_{u_{n},v_{n}}(u_{n}^{2}+v_{n}^{2})dx+\int_{D_{n}^{\left( 1\right)
}}f_{\beta }\left( u_{n}\right) dx+\int_{D_{n}^{\left( 2\right) }}f_{\beta
}\left( v_{n}\right) dx  \notag \\
&\geq &\frac{1}{4}\left\Vert \left( u_{n},v_{n}\right) \right\Vert _{H}^{2}+%
\frac{C_{1}}{2}\int_{\mathbb{R}^{3}}\phi
_{u_{n},v_{n}}(u_{n}^{2}+v_{n}^{2})dx-\left\vert m_{\beta }\right\vert
\left( \left\vert D_{n}^{\left( 1\right) }\right\vert +\left\vert
D_{n}^{\left( 2\right) }\right\vert \right) ,  \label{A-7}
\end{eqnarray}%
where the sets $D_{n}^{\left( 1\right) }:=\left\{ x\in \mathbb{R}%
^{3}:u_{n}\left( x\right) \in \left( s_{1},s_{2}\right) \right\} $ and $%
D_{n}^{\left( 2\right) }:=\left\{ x\in \mathbb{R}^{3}:v_{n}\left( x\right)
\in \left( s_{1},s_{2}\right) \right\} .$ It follows from (\ref{A-4}) and (%
\ref{A-7}) that
\begin{equation}
\left\vert m_{\beta }\right\vert \left( \left\vert D_{n}^{\left( 1\right)
}\right\vert +\left\vert D_{n}^{\left( 2\right) }\right\vert \right) >\frac{1%
}{4}\left\Vert \left( u_{n},v_{n}\right) \right\Vert _{H}^{2},  \label{A-12}
\end{equation}%
which implies that $\left\vert D_{n}^{\left( 1\right) }\right\vert
+\left\vert D_{n}^{\left( 2\right) }\right\vert \rightarrow \infty $ as $%
n\rightarrow \infty ,$ since $\left\Vert (u_{n},v_{n})\right\Vert
_{H}\rightarrow \infty $ as $n\rightarrow \infty .$ Moreover, since $%
D_{n}^{\left( 1\right) }$ and $D_{n}^{\left( 2\right) }$ are spherically
symmetric, we define $\rho _{n}^{\left( i\right) }:=\sup \left\{ \left\vert
x\right\vert :x\in D_{n}^{\left( i\right) }\right\} $ for $i=1,2.$ Then we
can take $x^{\left( 1\right) },x^{\left( 2\right) }\in \mathbb{R}^{3}$ such
that $\left\vert x^{\left( i\right) }\right\vert =\rho _{n}^{\left( i\right)
}.$ Clearly, $u_{n}\left( x^{\left( 1\right) }\right) =v_{n}\left( x^{\left(
2\right) }\right) =s_{1}>0.$ Recall the following Strauss's inequality by
Strauss \cite{S}%
\begin{equation}
\left\vert z\left( x\right) \right\vert \leq c_{0}\left\vert x\right\vert
^{-1}\left\Vert z\right\Vert _{H^{1}}\text{ for all }z\in H_{r}^{1}(\mathbb{R%
}^{3})  \label{A-13}
\end{equation}%
for some $c_{0}>0.$ Thus, by $\left( \ref{A-12}\right) $ and $\left( \ref%
{A-13}\right) ,$ we have%
\begin{equation*}
0<s_{1}=u_{n}\left( x^{\left( 1\right) }\right) <c_{0}\left( \rho
_{n}^{\left( 1\right) }\right) ^{-1}\left\Vert u_{n}\right\Vert _{H^{1}}\leq
2c_{0}\left\vert m_{\beta }\right\vert ^{1/2}\left( \rho _{n}^{\left(
1\right) }\right) ^{-1}\left( \left\vert D_{n}^{\left( 1\right) }\right\vert
+\left\vert D_{n}^{\left( 2\right) }\right\vert \right) ^{1/2}
\end{equation*}%
and%
\begin{equation*}
0<s_{1}=v_{n}\left( x^{\left( 2\right) }\right) <c_{0}\left( \rho
_{n}^{\left( 2\right) }\right) ^{-1}\left\Vert v_{n}\right\Vert _{H^{1}}\leq
2c_{0}\left\vert m_{\beta }\right\vert ^{1/2}\left( \rho _{n}^{\left(
2\right) }\right) ^{-1}\left( \left\vert D_{n}^{\left( 1\right) }\right\vert
+\left\vert D_{n}^{\left( 2\right) }\right\vert \right) ^{1/2}.
\end{equation*}%
These imply that%
\begin{equation}
c_{i}\rho _{n}^{\left( i\right) }\leq \left( \left\vert D_{n}^{\left(
1\right) }\right\vert +\left\vert D_{n}^{\left( 2\right) }\right\vert
\right) ^{1/2}\text{ for some }c_{i}>0\text{ and }i=1,2.  \label{A-14}
\end{equation}%
On the other hand, since $\widetilde{J}\left( u_{n},v_{n}\right) \leq 0,$ we
have
\begin{eqnarray*}
&&\frac{2}{C_{1}}\left\vert m_{\beta }\right\vert \left( \left\vert
D_{n}^{\left( 1\right) }\right\vert +\left\vert D_{n}^{\left( 2\right)
}\right\vert \right)  \\
&\geq &\int_{\mathbb{R}^{3}}\phi _{u_{n},v_{n}}\left(
u_{n}^{2}+v_{n}^{2}\right) dx \\
&=&\int_{\mathbb{R}^{3}}\int_{\mathbb{R}^{3}}\frac{u_{n}^{2}(x)u_{n}^{2}(y)}{%
|x-y|}dxdy+\int_{\mathbb{R}^{3}}\int_{\mathbb{R}^{3}}\frac{%
v_{n}^{2}(x)v_{n}^{2}(y)}{|x-y|}dxdy+2\int_{\mathbb{R}^{3}}\int_{\mathbb{R}%
^{3}}\frac{u_{n}^{2}(x)v_{n}^{2}(y)}{|x-y|}dxdy \\
&\geq &\int_{D_{n}^{\left( 1\right) }}\int_{D_{n}^{\left( 1\right) }}\frac{%
u_{n}^{2}(x)u_{n}^{2}(y)}{|x-y|}dxdy+\int_{D_{n}^{\left( 2\right)
}}\int_{D_{n}^{\left( 2\right) }}\frac{v_{n}^{2}(x)v_{n}^{2}(y)}{|x-y|}%
dxdy+2\int_{D_{n}^{\left( 2\right) }}v_{n}^{2}(y)\left( \int_{D_{n}^{\left(
1\right) }}\frac{u_{n}^{2}(x)}{|x-y|}dx\right) dy \\
&\geq &s_{1}^{4}\left( \frac{\left\vert D_{n}^{\left( 1\right) }\right\vert
^{2}}{2\rho _{n}^{\left( 1\right) }}+\frac{\left\vert D_{n}^{\left( 2\right)
}\right\vert ^{2}}{2\rho _{n}^{\left( 2\right) }}\right)
+2\int_{D_{n}^{\left( 2\right) }}v_{n}^{2}(y)\left( \int_{D_{n}^{\left(
1\right) }}\frac{u_{n}^{2}(x)}{|x|+\left\vert y\right\vert }dx\right) dy \\
&\geq &s_{1}^{4}\left( \frac{\left\vert D_{n}^{\left( 1\right) }\right\vert
^{2}}{2\rho _{n}^{\left( 1\right) }}+\frac{\left\vert D_{n}^{\left( 2\right)
}\right\vert ^{2}}{2\rho _{n}^{\left( 2\right) }}\right) +\frac{%
2s_{1}^{4}\left\vert D_{n}^{\left( 1\right) }\right\vert \left\vert
D_{n}^{\left( 2\right) }\right\vert }{\rho _{n}^{\left( 1\right) }+\rho
_{n}^{\left( 2\right) }} \\
&\geq &s_{1}^{4}\left( \frac{\left\vert D_{n}^{\left( 1\right) }\right\vert
^{2}}{2\rho _{n}^{\left( 1\right) }}+\frac{\left\vert D_{n}^{\left( 2\right)
}\right\vert ^{2}}{2\rho _{n}^{\left( 2\right) }}+\frac{2\left\vert
D_{n}^{\left( 1\right) }\right\vert \left\vert D_{n}^{\left( 2\right)
}\right\vert }{\rho _{n}^{\left( 1\right) }+\rho _{n}^{\left( 2\right) }}%
\right) ,
\end{eqnarray*}%
and together with $\left( \ref{A-14}\right) ,$ we further have%
\begin{eqnarray*}
\frac{2}{C_{1}s_{1}^{4}}\left\vert m_{\beta }\right\vert \left( \left\vert
D_{n}^{\left( 1\right) }\right\vert +\left\vert D_{n}^{\left( 2\right)
}\right\vert \right)  &\geq &\frac{\left\vert D_{n}^{\left( 1\right)
}\right\vert ^{2}}{2\rho _{n}^{\left( 1\right) }}+\frac{\left\vert
D_{n}^{\left( 2\right) }\right\vert ^{2}}{2\rho _{n}^{\left( 2\right) }}+2%
\frac{\left\vert D_{n}^{\left( 1\right) }\right\vert \left\vert
D_{n}^{\left( 2\right) }\right\vert }{\rho _{n}^{\left( 1\right) }+\rho
_{n}^{\left( 2\right) }} \\
&\geq &\frac{c_{1}\left\vert D_{n}^{\left( 1\right) }\right\vert ^{2}}{%
2\left( \left\vert D_{n}^{\left( 1\right) }\right\vert +\left\vert
D_{n}^{\left( 2\right) }\right\vert \right) ^{1/2}}+\frac{c_{2}\left\vert
D_{n}^{\left( 2\right) }\right\vert ^{2}}{2\left( \left\vert D_{n}^{\left(
1\right) }\right\vert +\left\vert D_{n}^{\left( 2\right) }\right\vert
\right) ^{1/2}} \\
&&+\frac{2\left\vert D_{n}^{\left( 1\right) }\right\vert \left\vert
D_{n}^{\left( 2\right) }\right\vert }{(c_{1}^{-1}+c_{2}^{-1})\left(
\left\vert D_{n}^{\left( 1\right) }\right\vert +\left\vert D_{n}^{\left(
2\right) }\right\vert \right) ^{1/2}} \\
&\geq &\min \left\{ \frac{c_{1}}{2},\frac{c_{2}}{2}%
,(c_{1}^{-1}+c_{2}^{-1})^{-1}\right\} \left( \left\vert D_{n}^{\left(
1\right) }\right\vert +\left\vert D_{n}^{\left( 2\right) }\right\vert
\right) ^{3/2},
\end{eqnarray*}%
which implies that for all $n,$
\begin{equation*}
\left\vert D_{n}^{\left( 1\right) }\right\vert +\left\vert D_{n}^{\left(
2\right) }\right\vert \leq M\text{ for some }M>0.
\end{equation*}%
This contradicts with $\left\vert D_{n}^{\left( 1\right) }\right\vert
+\left\vert D_{n}^{\left( 2\right) }\right\vert \rightarrow \infty $ as $%
n\rightarrow \infty .$ Hence, we conclude that $\left\{ \left(
u_{n},v_{n}\right) \right\} $ is bounded in $H_{r}.$

Assume that $\left( u_{n},v_{n}\right) \rightharpoonup \left(
u_{0},v_{0}\right) $ in $H_{r}.$ Next, we prove that $\left(
u_{n},v_{n}\right) \rightarrow \left( u_{0},v_{0}\right) $ strongly in $%
H_{r}.$ Suppose on contrary. Then there holds%
\begin{equation*}
\left\Vert \left( u_{0},v_{0}\right) \right\Vert _{H}^{2}<\liminf \left\Vert
\left( u_{n},v_{n}\right) \right\Vert _{H}^{2},
\end{equation*}%
Since $H_{r}\hookrightarrow L^{r}(\mathbb{R}^{3})\times L^{r}(\mathbb{R}%
^{3}) $ is compact for $2<r<6$ (see \cite{S}), we have%
\begin{equation*}
\int_{\mathbb{R}^{3}}F_{\beta }\left( u_{n},v_{n}\right) dx\rightarrow \int_{%
\mathbb{R}^{3}}F_{\beta }\left( u_{0},v_{0}\right) dx.
\end{equation*}%
Moreover, it follows from Ruiz \cite[Lemma 2.1]{R1} that
\begin{equation*}
\int_{\mathbb{R}^{3}}\phi _{u_{n},v_{n}}\left( u_{n}^{2}+v_{n}^{2}\right)
dx\rightarrow \int_{\mathbb{R}^{3}}\phi _{u_{0},v_{0}}\left(
u_{0}^{2}+v_{0}^{2}\right) dx.
\end{equation*}%
These imply that%
\begin{equation*}
\frac{1}{p}\int_{\mathbb{R}^{3}}F_{\beta }\left( u_{0},v_{0}\right) dx-\frac{%
1}{2}\left\Vert \left( u_{0},v_{0}\right) \right\Vert _{H}^{2}>0
\end{equation*}%
and%
\begin{equation*}
\frac{\frac{1}{p}\int_{\mathbb{R}^{3}}F_{\beta }\left( u_{0},v_{0}\right) dx-%
\frac{1}{2}\left\Vert \left( u_{0},v_{0}\right) \right\Vert _{H}^{2}}{\int_{%
\mathbb{R}^{3}}\phi _{u_{0},v_{0}}\left( u_{0}^{2}+v_{0}^{2}\right) dx}%
>\Lambda \left( \beta \right) ,
\end{equation*}%
which is a contradiction. Hence, we conclude that $\left( u_{n},v_{n}\right)
\rightarrow \left( u_{0},v_{0}\right) $ strongly in $H_{r}$ and $\left(
u_{0},v_{0}\right) \in H_{r}\setminus \left\{ \left( 0,0\right) \right\} .$
Therefore, $\Lambda \left( \beta \right) $ is achieved. This completes the
proof.
\end{proof}

\begin{theorem}
\label{AT-2}Let $2<p<3$ and $\beta \geq 0$. Then the following statements
are true.\newline
$\left( i\right) $ $0<\overline{\Lambda }\left( \beta \right) <\infty ;$%
\newline
$\left( ii\right) $ $\overline{\Lambda }\left( \beta \right) $ is achieved,
i.e. there exists $\left( u_{0},v_{0}\right) \in H\setminus \left\{ \left(
0,0\right) \right\} $ such that
\begin{equation*}
\overline{\Lambda }\left( \beta \right) =\frac{\int_{\mathbb{R}^{3}}F_{\beta
}\left( u_{0},v_{0}\right) dx-\left\Vert \left( u_{0},v_{0}\right)
\right\Vert _{H}^{2}>0}{\int_{\mathbb{R}^{3}}\phi _{u_{0},v_{0}}\left(
u_{0}^{2}+v_{0}^{2}\right) dx}.
\end{equation*}
\end{theorem}

\begin{proof}
The proof is similar to the argument in Theorem \ref{AT-1}, and we omit it
here.
\end{proof}

\section*{Acknowledgments}

J. Sun was supported by the National Natural Science Foundation of China
(Grant No. 11671236) and Shandong Provincial Natural Science Foundation
(Grant No. ZR2020JQ01). T.F. Wu was supported by the National Science and
Technology Council, Taiwan (Grant No. 112-2115-M-390-001-MY3).




\begin{thebibliography}{99}
\bibitem{A} A. Ambrosetti, On the Schr\"{o}dinger-Poisson systems, Milan J.
Math. 76 (2008) 257--274.

\bibitem{AC} A. Ambrosetti, E. Colorado, Standing waves of some coupled
nonlinear Schr\"{o}dinger equations, J. London Math. Soc. 75 (2007) 67--82.

\bibitem{AP} A. Azzollini, A. Pomponio, Ground state solutions for the
nonlinear Schr\"{o}dinger--Maxwell equations, J. Math. Anal. Appl. 345
(2008) 90--108.

\bibitem{BDW} T. Bartsch, N. Dancer, Z.-Q. Wang, A Liouville theorem,
a-priori bounds, and bifurcating branches of positive solutions for a
nonlinear elliptic system. Calc. Var. Partial Differential Equations 37
(2010) 345--361.

\bibitem{BF} V. Benci, D. Fortunato, An eigenvalue problem for the Schr\"{o}%
dinger-Maxwell equations, Topol. Methods Nonlinear Anal. 11 (1998) 283--293.

\bibitem{BDH} P.A. Binding, P. Dr\'{a}bek, Y.X. Huang, On Neumann boundary
value problems for some quasilinear elliptic equations, Electron. J.
Differential Equations 5 (1997) 1--11.

\bibitem{BLi} H. Br\'{e}zis, E.H. Lieb,\ A relation between pointwise
convergence of functions and convergence of functionals, Proc. Amer. Math.
Soc. 88 (1983) 486--490.

\bibitem{BW1} K.J. Brown, T.F. Wu, A fibrering map approach to a semilinear
elliptic boundary value problem, Electron. J. Differential Equations 69
(2007) 1--9.

\bibitem{BW2} K.J. Brown, T.F. Wu, A fibering map approach to a potential
operator equation and its applications, Differential Integral Equations 22
(2009) 1097--1114.

\bibitem{BZ} K.J. Brown, Y. Zhang, The Nehari manifold for a semilinear
elliptic equation with a sign-changing weight function, J. Differential
Equations 193 (2003) 481--499.

\bibitem{CZ1} Z. Chen, W. Zou, Positive least energy solutions and phase
separation for coupled Schr\"{o}dinger equations with critical exponent.
Arch. Ration. Mech. Anal. 205 (2012) 515--551.

\bibitem{CZ2} Z. Chen, W. Zou, An optimal constant for the existence of
least energy solutions of a coupled Schr\"{o}dinger system. Calc. Var.
Partial Differential Equations 48 (2013) 695--711.

\bibitem{DM1} T. D'Aprile, D. Mugnai, Non-existence results for the coupled
Klein--Gordon--Maxwell equations, Adv. Nonlinear Stud. 4 (2004) 307--322.

\bibitem{DMS} P. d'Avenia, L. Maia, G. Siciliano, Hartree-Fock type systems:
Existence of ground states and asymptotic behavior, J. Differential
Equations 335 (2022) 580--614.

\bibitem{D1} P.A.M. Dirac, Note on exchange phenomena in the Thomas atom,
Proc. Cambridge Philos. Soc. 26 (1931) 376--385.

\bibitem{DP} P. Dr\'{a}bek, S. I. Pohozaev, Positive solutions for the $p$--Laplacian: application of the fibering method,
Proc. Roy. Soc. Edinburgh Sect. A 127 (1997) 703--726.

\bibitem{E} I. Ekeland, On the variational principle, J. Math. Anal. Appl.
17 (1974) 324--353.

\bibitem{E2} I. Ekeland, Convexity Methods in Hamiltonian Mechanics,
Springer, 1990.

\bibitem{F1} V. Fock, N\"{a}herungsmethode zur L\"{o}sung des
quantenmechanischen Mehrk\"{o}rperproblems, Z. Phys. 61 (1930) 126--148.

\bibitem{H} D. Hartree, The wave mechanics of an atom with a non-Coulomb
central field. Part I. Theory and methods, Proc. Cambridge Philos. Soc. 24
(1928) 89--312.

\bibitem{K} M.K. Kwong, Uniqueness of positive solution of $\Delta
u-u+u^{p}=0$ in $\mathbb{R}^{N}$, Arch. Ration. Mech. Anal. 105 (1989)
243--266.

\bibitem{LW} T.C. Lin, J. Wei, Ground state of n coupled nonlinear Schr\"{o}%
dinger equations in $\mathbb{R}^{N}$, $n\leq 3$, Comm. Math. Phys. 255
(2005) 629--653.

\bibitem{Li1} P.L. Lions, The concentration-compactness principle in the
calculus of variation. The locally compact case, part I, Ann. Inst. H.
Poincar\'{e} C Anal. Non Lin\'{e}aire 1 (1984) 109--145.

\bibitem{Lions} P.L. Lions, Solutions of Hartree--Fock equations for Coulomb
systems, Comm. Math. Phys. 109 (1984) 33--97.

\bibitem{LW1} Z. Liu, Z.-Q. Wang, Multiple bound states of nonlinear Schr%
\"{o}dinger systems, Comm. Math. Phys. 282 (2008) 721--731.

\bibitem{MMP} L.A. Maia, E. Montefusco, B. Pellacci, Positive solutions for
a weakly coupled nonlinear Schr\"{o}dinger system, J. Differential Equations
229 (2006) 743--767.

\bibitem{P} R. Palais, The Principle of symmetric criticality, Comm. Math.
Phys. 69 (1979) 19--30.

\bibitem{R} P. H. Rabinowitz, Minimax Methods in Critical Point Theory with
Applications to Differential Equations, Regional Conference Series in
Mathematics, American Mathematical Society, 1986.

\bibitem{R1} D. Ruiz, The Schr\"{o}dinger--Poisson equation under the effect
of a nonlinear local term, J. Funct. Anal. 237 (2006) 655--674.

\bibitem{R2} D. Ruiz, On the Schr\"{o}dinger--Poisson-Slater system:
behavior of minimizers, radial and nonradial cases, Arch. Ration. Mech.
Anal. 198 (2010) 349--368.

\bibitem{SS} O. S\'{a}nchez, J. Soler, Long-time dynamics of the Schr\"{o}dinger--Poisson--Slater system, J. Stat. Phys. 114 (2004) 179--204.

\bibitem{Sl1} J.C. Slater, A note on Hartree's method, Phys. Rev. 35 (1930)
210--211.

\bibitem{Sl2} J.C. Slater, A simplification of the Hartree--Fock method,
Phys. Rev. 81 (1951) 385--390.

\bibitem{S} W.A. Strauss, Existence of solitary waves in higher dimensions,
Comm. Math. Phys. 55 (1977) 149--162.

\bibitem{SWF} J. Sun, T.F. Wu, Z. Feng, Multiplicity of positive solutions
for a nonlinear Schr\"{o}dinger--Poisson system, J. Differential Equations
260 (2016) 586--627.

\bibitem{SWF1} J. Sun, T.F. Wu, Z. Feng, Non-autonomous Schr\"{o}%
dinger--Poisson problem in $\mathbb{R}^{3}$, Discrete Contin. Dyn. Syst. 38
(2018) 1889--1933.

\bibitem{SWF2} J. Sun, T.F. Wu, Z. Feng, Two positive solutions to
non-autonomous Schr\"{o}dinger--Poisson systems, Nonlinearity 32 (2019)
4002--4032.

\bibitem{T} G. Tarantello, On nonhomogeneous elliptic equations involving
critical Sobolev exponent, Ann. Inst. H. Poincar\'{e} C Anal. Non Lin\'{e}%
aire 9 (1992) 281--304.

\bibitem{T1} E. Timmermans, Phase separation of Bose--Einstein condensates.
Phys. Rev. Lett. 81 (1998) 5718--5721.

\bibitem{Wi} M. Willem, Minimax Theorems, Birkh\"{a}user, Boston, 1996.

\bibitem{ZZ} L. Zhao, F. Zhao, On the existence of solutions for the Schr%
\"{o}dinger--Poisson equations, J. Math. Anal. Appl. 346 (2008) 155--169.
\end{thebibliography}
\end{document}